\documentclass[a4paper]{amsart}
\makeatletter
\usepackage{geometry,aliascnt}

\usepackage[mathcal]{eucal}
\usepackage{microtype}
\usepackage{graphicx,multirow}
\usepackage[all]{xy}
\usepackage{tikz}
\usepackage[utf8x]{inputenc} 

\usepackage{enumerate,caption}
\usepackage{enumitem}

\usepackage{xspace}
\usepackage[modulo]{lineno}
\usepackage[linktocpage,colorlinks=true,linkcolor=blue,citecolor=blue,urlcolor=blue,citebordercolor={0
  0 1},urlbordercolor={0 0 1},linkbordercolor={0 0 1}]{hyperref} 
\usepackage[nameinlink]{cleveref}
\newgeometry{margin=1.7in}
\crefname{diagram}{Diagram}{Diagram}
\usepackage[normalem]{ulem}
\usepackage{soul}

\numberwithin{equation}{section}
\usepackage{amssymb,latexsym,amsmath,amsthm,mathrsfs}
\usepackage{tikz-cd,color}
\theoremstyle{plain}
\newtheorem{theorem}{Theorem}[section]
\newtheorem{corollary}[theorem]{Corollary}
\newtheorem{lemma}[theorem]{Lemma}

\newtheorem{claim}[theorem]{Claim}

\newtheorem{proposition}[theorem]{Proposition}

\newtheorem*{thm*}{Theorem}
\newtheorem{thmx}{Theorem}

\theoremstyle{definition}

\newtheorem{remark}[theorem]{Remark}

\newtheorem{definition}[theorem]{Definition}

\newtheorem{example}[theorem]{Example}
\newtheorem{fact}[theorem]{Fact}

\newtheorem{warn}[theorem]{Warning}

\newenvironment{notation}[1][Notation]
{\begin{trivlist} \item[\hskip \labelsep {\bfseries #1}]}
{\end{trivlist}}
\title{Variation of Stratifications From Toric GIT}
\author{Chi-yu Cheng}
\date{}
\DeclareMathOperator{\Relint}{Relint}
\DeclareMathOperator{\Proj}{Proj}
\DeclareMathOperator{\Spec}{Spec}
\DeclareMathOperator{\Hom}{Hom}

\DeclareMathOperator{\Amp}{Amp}

\DeclareMathOperator{\Pic}{Pic}
\DeclareMathOperator{\Nef}{Nef}
\DeclareMathOperator{\Cl}{Cl}
\DeclareMathOperator{\CDiv}{CDiv}

\DeclareMathOperator{\Sp}{Sp}
\DeclareMathOperator{\dist}{dist}
\begin{document}
\begin{abstract}
    When a reductive group acts on an algebraic variety, a linearized ample line bundle induces a stratification on the variety where the strata are ordered by the degrees of instability. In this paper, we study variation of stratifications coming from the group actions in the GIT quotient construction for projective toric varieties. Cox showed that each projective toric variety is a GIT quotient of an affine space by a diagonalizable group with respect to linearizations that come from ample divisors on the toric variety. We provide a sufficient conditions for two ample divisors to induce the same stratification and formulate two types of walls in the ample cone that completely capture two kinds of variations. We also prove that the variation is intrinsic to the primitive collections and the relations among ray generators of the fans.  
\end{abstract}
\maketitle
\tableofcontents
\section{Introduction}
Geometric Invariant Theory (abbreviated as GIT) was developed in \cite{MR1304906} by Mumford to construct quotients for group actions on algebraic varieties. While there are standardized quotient constructions in various mathematical categories including the category of differentiable manifolds, forming quotients within the category of algebraic varieties is less immediate. 

The local pieces of an algebraic variety are affine varieties. An affine variety is the spectrum of a ring, consisting of algebraic functions defined on the affine variety. Constructing varieties that parametrize orbits is not immediate even for group actions on affine varieties. While it is natural to consider the spectrum of the subring of functions that are constant on the orbits, namely, the invariant functions, there might not be enough invariants to separate the orbits.

Mumford realized some orbits are so exceptional, that they must be left out of the quotient. The notion of stability was then introduced by Mumford to specify an invariant open subset, known as the semistable locus that allows for a \emph{good quotient} \cite{MR309940}. Since then GIT has led to fruitful results in algebraic geometry, especially constructions of various moduli spaces. Examples include  sequence of linear subspaces \cite{MR0175899}, representations of quivers \cite{MR1315461}, vector bundles on a curve \cite{MR0175899}, and more generally coherent sheaves on a projective variety \cite{MR233371}.

GIT stability however, depends on the choice of a linearized line bundle. Variation of Geometric Invariant Theory (abbreviated as VGIT) quotients coming from different choices of linearized line bundles were well studied by \cite{MR1659282} and \cite{MR1333296} in the 90's. The main result is that when a reductive group $G$ acts on a normal projective variety over the field of complex numbers, the space of $G$-linearized ample line bundles has a finite wall and chamber decomposition such that 
\begin{enumerate}
    \item stability and the GIT quotient do not change inside each chamber, 
    \item the variation of GIT quotients after wall crossing is described by a flip.
\end{enumerate}

It is worthwhile pointing out that VGIT has applications to birational geometry. It not only provides examples of flips, but also realizes Mori theory as an instance of VGIT. In
\cite{MR1786494} it was shown that if a projective variety is a Mori dream space, it is a GIT quotient of an
affine variety by a torus. In this case the Mori chambers correspond to VGIT chambers. 

Our work takes on the theme of variations. We study the variation of stratifications in invariant theory (abbreviated as VSIT) that occurs in the GIT quotient construction of projective toric varieties. This is where instability in invariant theory (abbreviated as IIT) and toric varieties meet. 

The field of toric varieties is famous for being a good testing ground in algebraic geometry and toric varieties can be built from very concrete objects that are known as fans. Moreover, there is a well understood dictionary translating combinatorial properties of a fan to the geometric properties of its associated toric variety, giving this field a rich amount of computability. The main driving force behind this paper is the generous number of examples toric varieties offer (\Cref{Examples_and_counter_examples}) and the interesting connections many results in this paper have to the combinatorics of the fans (\Cref{Relations to the structure of the fan}).

IIT on the other hand, studies the points that are not semistable, namely the unstable points. Although the unstable locus is discarded when one takes a GIT quotient, it possesses interesting and useful properties. For instance, let $X$ be an affine variety  equipped with an action by a reductive group $G$ over the field $\mathbf{C}$ of complex numbers. A character $\chi:G\rightarrow\mathbf{C}^{\times}$ then gives a linearization of the trivial line bundle $\mathscr{O}_{X}.$ The pioneering IIT work \cite{MR506989} due to Kempf states that every $\chi$-unstable point is maximally destabilized by a one parameter subgroup of $G$ in a numerical sense that comes from the Hilbert-Mumford criterion (\cite{MR1304906},\cite{MR1315461}). In addition, such a one parameter subgroup when taken to be indivisible, is unique up to conjugacy by some parabolic subgroup of $G$. 

Therefore, a character $\chi$ of $G$ induces the following decomposition \begin{equation}\label{Kempf decompotision}
X=X^{\text{ss}}(\chi)\cup \big(\bigcup_{\lambda}S^{\chi}_{[\lambda]}\big)
\end{equation} where $X^{\text{ss}}(\chi)$ is the semistable locus with respect to $\chi$ and $S^{\chi}_{[\lambda]}$ consists of points in $X$ that are maximally destabilized by some element in the conjugacy class $[\lambda]$ of the indivisible one parameter subgroup $\lambda$ in $G$. This decomposition actually induces a stratification of $X$ which we now define.

\begin{definition}Let $Y$ be a topological space. A finite collection of locally closed subspaces  $\{Y_{a}\mid a\in\mathscr{A}\}$ forms a \emph{stratification of }$Y$ if $Y$ is a disjoint union of the strata $Y_{a}$ and there is a strict partial order on the index set $\mathscr{A}$ such that 
$\partial Y_{a^{\prime}}\cap Y_{a}\neq\emptyset$ only if $a>a^{\prime}$.\end{definition}

In \cite{MR553706}, Hesselink prescribes to each piece in (\ref{Kempf decompotision}) a degree of instability and showed that the collection of pieces, when strictly partially ordered by degrees of instability, forms a stratification of $X$. We shall refer to this stratification as \emph{the stratification of} $X$ \emph{induced by }$\chi$.

We would like to point out that the stratification just introduced is not solely aesthetic but has actual applications to obtaining certain topological invariants of the quotients. For instance, in \cite{MR766741} Kirwan used Hesselink's stratification to produce an inductive formula for the equivariant Betti numbers of the semistable loci, which in \emph{good cases} will be the Betti numbers of the quotients of nonsingular complex projective varieties.

The stratifications that we are concerned with in this paper are those that occur from the GIT quotient construction of projective toric varieties. To get a picture of the stratifications under this setting, let $\Sigma$ be a fan. We let $X_{\Sigma}$ be the toric variety built from the fan, $\Cl(X_{\Sigma})$ be its divisor class group, and $\Sigma(1)$ be the collection of rays in $\Sigma$. We also let $$\mathbf{C}^{\Sigma(1)}=\Spec \mathbf{C}[x_{\rho}\mid\rho\in \Sigma(1)]$$ be the affine space whose coordinates are indexed by rays in $\Sigma$.  In \cite{MR1299003} and \cite{MR2810322}, Cox showed that every projective toric variety $X_{\Sigma}$ is a GIT quotient of $\mathbf{C}^{\Sigma(1)}$ by the diagonalizable group $$G:=\Hom_{\mathbf{Z}}(\Cl(X_{\Sigma}),\mathbf{C}^{\times})$$ with respect to characters $\chi_{D}$ that come from ample divisors $D$ on $X_{\Sigma}$. Moreover, the $\chi_{D}$-unstable locus $(\mathbf{C}^{\Sigma(1)})^{\text{us}}(\chi_{D})$ can be described by the vanishing loci of certain collections $x_{\rho}$'s, known as the primitive collections of the fan $\Sigma$ (See \Cref{primitive_collection_def}). Specifically, \begin{equation}\label{amp_bad_locus_intro}
(\mathbf{C}^{\Sigma(1)})^{\text{us}}(\chi_{D})=\bigcup_{C}V(\{X_{\rho}\mid\rho\in C \})
\end{equation}where the union is taken over primtive collections $C$ of the fan.

With these data we can think about stratifications of $\mathbf{C}^{\Sigma(1)}$ and the variation induced by different ample divisors on $X_{\Sigma}$. Here is how we decide if the stratifications induced by two ample divisors are equivalent. 

\begin{definition}We say two stratifications $\{Y_{a}\mid a\in\mathscr{A}\}$ and $\{Y_{b}\mid b\in\mathscr{B}\}$ of a topological space $Y$ are equivalent if there is a bijection $\Phi:\mathscr{A}\rightarrow\mathscr{B}$ such that 
\begin{enumerate}
    \item $\Phi$ preserves strata. That is, $Y_{\Phi(a)}=Y_{a}$ for all $a\in\mathscr{A}$, and  
    \item $\Phi$ preserves order. That is, $\Phi(a)>\Phi(a^{\prime})$ if and only if $a>a^{\prime}$ for all $a,a^{\prime}\in\mathscr{A}$.
\end{enumerate} \end{definition}

With this notion we then say there is a \emph{type one variation} between the stratifications of $\mathbf{C}^{\Sigma(1)}$ induced by ample divisors $D$ and $D^{\prime}$ if there is no bijection that satisfies condition (1). We say there is a \emph{type two variation} between the stratifications of $\mathbf{C}^{\Sigma(1)}$ induced by ample divisors $D$ and $D^{\prime}$ if there is a bijection that satisfies condition (1), but not condition (2). Namely, there is a change of order among the strata.

It is important to point out that the variation studied in this paper is finer than VGIT in the following sense. Equality (\ref{amp_bad_locus_intro}) indicates that the semistable locus stays constant among all ample divisors. Hence by choosing different ample divisors, variation of stratifications only occurs in the fixed unstable locus. We are now ready to state the main results of this paper.

\subsection{The main results}
The major task of this paper is to nail down at which ample divisors the stratification undergoes variations, and provide a sufficient condition for two ample divisors to induce equivalent stratifications. 

\begin{thmx}[\Cref{struct_walls}, \Cref{finite_wall_semi-chamber}, \Cref{Toric_VSIT_decomposition_ample_cone}, \Cref{type_one_wall_captures_type_one_variation}, and  \Cref{type_two_variation_crosses_type_two_walls}]\label{thmA}
Let $X_{\Sigma}$ be a projective toric variety. There are two types of walls in the cone of ample divisors, called type one walls and type two walls respectively. The following properties hold for walls:
\begin{enumerate}
    \item There are finitely many walls. 
    \item A type one (resp. type two) wall is a rational hyperplane (resp. homogeneous quadratic hypersurface).
    \item Type one walls capture type one variations in the following sense: Let $D,D^{\prime}$ be two ample divisors on $X_{\Sigma}$. Then the stratifications induced by $D$ and $D^{\prime}$ undergo a type one variation only if the line segment $\overline{DD^{\prime}}$ intersects a type one wall properly in the ample cone. Namely, $\overline{DD^{\prime}}$ crosses a type one wall.
    \item Type two walls capture type two variations in the following sense:  Let $D,D^{\prime}$ be two ample divisors on $X_{\Sigma}$. If the stratifications induced by $D$ and $D^{\prime}$ undergo a type two variation and if the line segment $\overline{DD^{\prime}}$ intersects no type one walls properly, then $\overline{DD^{\prime}}$ intersects a type two wall properly in the ample cone. Namely, $\overline{DD^{\prime}}$ crosses a type two wall.\end{enumerate}
Away from the walls, the ample cone is decomposed into semi-chambers such that the following properties hold:
\begin{enumerate}
    \item There are finitely many semi-chambers.
    \item A semi-chamber is a cone, possibly not convex.
    \item The stratification stays constant in a semi-chamber in the following sense: If two ample divisors are in the same semi-chamber, then they induce equivalent stratifications of $\mathbf{C}^{\Sigma(1)}$.
\end{enumerate}
\end{thmx}
Next, we provide two results that have connections to the combinatorics of the fans. The first one is related to primitive relations for simplicial complete toric varieties (\Cref{primitive_relation}). The significance of primitive relations is that for simplicial projective toric varieties, its Mori cone can be generated by primitive relations (see \cite{MR1133869}).

The notion of primitive relations occurred to us in the context of Cox's quotient construction for projective toric varieties. Equality (\ref{amp_bad_locus_intro}) that describes the GIT unstable locus was established in \cite{MR2810322} by considering the invariants with respect to characters induced by ample divisors. We supply an alternative proof of (\ref{amp_bad_locus_intro}) in \Cref{bad_locus_ample}, using the Hilbert-Mumford criterion, \Cref{HilbMumAffine}.

Below is the theorem describing the connections between primitive relations and the destabilizing one parameter subgroups we used in \Cref{bad_locus_ample}. One fact to have in mind is that the one parameter subgroups of the group in Cox's quotient construction correspond to relations among the ray generators of the fan (\Cref{lemma5.1.1}).    
\begin{thmx}\label{thmB}
Let $X_{\Sigma}$ be a projective toric variety. For each primitive collection $C$ of the fan, there is a one parameter subgroup $\lambda(C)$ of the group $G=\Hom_{\mathbf{Z}}(\Cl(X_{\Sigma}),\mathbf{C}^{\times})$ depending only on $C$, that fails the Hilbert-Mumford criterion for all points in $V(\{x_{\rho}\mid\rho\in C\})$ and for all ample divisors. Moreover, if $\Sigma$ is simplicial, $-\lambda(C)$ is an integral positive multiple of the primitive relation of $C$.
\end{thmx}
As an important application, the above theorem allows us to consider stratifications induced by $\mathbf{R}$-ample divisors as well in \Cref{Extensions_to_the_ample_cone} .

Finally, we proved that the variation of stratifications is intrinsic to the relations among the ray generators and the primitive collections of the fan. More precisely, we consider two complete fans $\Sigma_{1}$ and $\Sigma_{2}$ of the same dimension that admit a bijection $$\Psi:\Sigma_{1}(1)\rightarrow\Sigma_{2}(1)$$ between the rays such that the following two properties hold:
\begin{enumerate}
    \item $\Psi(C)$ is a primitive collection for $\Sigma_{2}$ if and only if $C$ is a primitive collection for $\Sigma_{1}$, and 
    \item for any $\{a_{\rho}\mid{\rho\in\Sigma_{1}(1)}\}\subset \mathbf{Z}$, $\sum_{\rho}a_{\rho}u_{\rho}=0\Leftrightarrow\sum_{\rho}a_{\rho}u_{\Psi(\rho)}=0$ where $u_{\rho}$ (resp. $u_{\Psi(\rho)}$) stands for the ray generator of $\rho$ (resp. $\Psi(\rho)$).
\end{enumerate}
In \Cref{A VSIT adjunction}, we say such two complete fans $\Sigma_{1}$ and $\Sigma_{2}$ are \emph{amply equivalent}.

\begin{warn}
We note that ample equivalence between two fans does not result in an isomorphism between the toric varieties associated to them. A quick example is to consider the complete fan $\Sigma_{1}\subset\mathbf{R}^{2}$ with ray generators $(1,0),(0,1)$, and $(-1,-1)$ and the complete fan $\Sigma_{2}\subset\mathbf{R}^{2}$ with ray generators $(2,3),(1,-1)$, and $(-3,-2)$. Then $\Sigma_{1}$ and $\Sigma_{2}$ are amply equivalent. However, $\Sigma_{1}$ is a smooth fan but $\Sigma_{2}$ is not.
\end{warn}
If $\Sigma_{1}$ and $\Sigma_{2}$ are amply equivalent, then by Cox's quotient construction and the fact that $\Sigma_{1}(1)\simeq\Sigma_{2}(1)$, the affine spaces for the quotient constructions of $X_{\Sigma_{1}}$ and $X_{\Sigma_{2}}$ can be identified. Write $\mathbf{C}^{\Sigma(1)}$ as the common affine space. Based on the description (\ref{amp_bad_locus_intro}) using primitive collections and the assumption that $\Sigma_{1}$ and $\Sigma_{2}$ are amply equivalent, the unstable loci discarded in the quotient constructions for both $X_{\Sigma_{1}}$ and $X_{\Sigma_{2}}$ can be identified in $\mathbf{C}^{\Sigma(1)}$. Write $Z(\Sigma)\subset\mathbf{C}^{\Sigma(1)}$ as the common unstable locus.

For $i=1,2$, let $G_{i}$ be the group $\Hom_{\mathbf{Z}}(\Cl(X_{\Sigma_{i}}),\mathbf{C}^{\times})$ that realizes $X_{\Sigma_{i}}$ as the quotient of $\mathbf{C}^{\Sigma(1)}$ and $\pmb{\Gamma}(G_{i})$ be the set of one parameter subgroups of $G_{i}$. Since each $G_{i}$ is abelian, each $\pmb{\Gamma}(G_{i})$ has a natural abelian group structure. We then prove in \Cref{StructCompFan} and \Cref{why_amply_equivalent} that there is a $\mathbf{Q}$-linear isomorphism $\varphi:\pmb{\Gamma}(G_{1})_{\mathbf{Q}}\simeq \pmb{\Gamma}(G_{2})_{\mathbf{Q}}$ whose dual $\varphi^{\ast}:\Cl(X_{\Sigma_{2}})_{\mathbf{Q}}\rightarrow\Cl(X_{\Sigma_{1}})_{\mathbf{Q}}$ between the $\mathbf{Q}$-divisor class groups restricts to a bijection between their $\mathbf{Q}$-ample divisors. The stratifications induced by $\mathbf{Q}$-ample divisors on $X_{\Sigma_{1}}$ and $X_{\Sigma_{2}}$ are related by the following adjunction:
\begin{thmx}[\Cref{ToricVSIT}]\label{thmC}
For every $\mathbf{Q}$-ample divisor $D$ on $X_{\Sigma_{2}}$, the assignment of the strata \begin{equation*}\begin{tikzcd}\text{\huge $S$}^{\varphi^{\ast}(\chi_{D})}_{\lambda}\arrow[r,mapsto,shift left=1]&\text{\huge $S$}^{\chi_{D}}_{\varphi(\lambda)}\end{tikzcd}\end{equation*} defined for all $\lambda\in\pmb{\Gamma}(G_{1})_{\mathbf{Q}}$ is an equivalence of stratifications of $Z(\Sigma)$ induced by $D$ and $\varphi^{\ast}(D)$. The adjunction also holds for $\mathbf{R}$-ample divisors $D$ on $X_{\Sigma_{2}}$ and $\lambda\in\pmb{\Gamma}(G_{1})_{\mathbf{R}}$.
\end{thmx}
\subsection{VSIT vs. VGIT}
Suppose $X$ is a variety endowed with an action by a reductive group $G$ over $\mathbf{C}$. It is defined in VGIT that two $G$-linearized ample line bundles are \emph{GIT-equivalent} if their semistable loci in $X$ are the same. We can think about a finer equivalence than GIT-equivalence. We define two $G$-linearized ample line bundles to be \emph{SIT-equivalent} if they induce equivalent stratifications of $X$. Since the stratifications induced by two SIT-equivalent line bundles must have a common stratum of the least instability, which is the semistable locus, SIT-equivalence is finer than GIT-equivalence among $G$-linearized ample line bundles.

While GIT-equivalence classes are well-understood by VGIT, very little is known about the geometry of SIT-equivalence classes. For example, it is known that GIT-equivalence classes on $X$ corresponds to relative interiors of a collection of convex rational polyhedral cones that is named by \cite{MR1772211} as a \emph{GIT fan for the action of }$G$ \emph{on} $X$. In \Cref{blow_up_2_pts}, we will see that a semi-chamber does have to be convex and its closure may not be a polyhedral cone. Moreover, in VGIT different chambers correspond to different GIT-equivalence classes. We will  see an example in \Cref{Blow-up of the Hirzebruch surface at a point} where two different semi-chambers are contained in a single SIT-equivalence class. This phenomenon also shows that wall crossing is not sufficient for variations of stratifications but only necessary (\Cref{thmA}). 

\subsection{Outline of the paper}
This paper is notation heavy. In view of this, we start \Cref{Some linear algebra}, \Cref{Instability in invariant theory}, \Cref{Recap on toric varieties}, and \Cref{Toric_VSIT_Section} with conventions and notations. If at any point of a section the reader forgets what certain symbols stand for, the reader may consult the beginning of the section.

Most of the analysis and computations in this paper come down to basic linear algebra, especially linear programming. We therefore introduce results from linear programming that will be used later at the very beginning, \Cref{Some linear algebra}. 

In \Cref{Instability in invariant theory}, we then recall various important IIT theorems, including the Hilbert Mumford criterion (\Cref{HilbMumAffine}), Kempf's result (\Cref{finitenesstheorem}) about maximally destablizing one parameter subgroups, and Hesselink's stratification induced by a character (\Cref{HesAffine}). 

In \Cref{Recap on toric varieties} we recall Cox's GIT quotient construction of projective toric varieties, a numerical criterion for a divisor to be ample (\Cref{numampleness}), and the the description of the unstable locus using primitive collections (\Cref{bad_locus_ample}). 

\Cref{Toric_VSIT_Section} is the heart of the paper. We first extend the notions of stability and stratfications induced by ample divisors to those induced by $\mathbf{R}$-ample divisors in \Cref{Extensions_to_the_ample_cone}, allowing us to analyze VSIT in the entire ample cone. We then introduce the notion of walls (\Cref{walls_def}) and semi-chambers (\Cref{semi-chamber_def}) in the ample cone at \Cref{wall and semi-chamber} and prove the main result \Cref{thmA} at \Cref{The main results}. In \Cref{Variation of stratification-an_elementary_example} we supply an example where two types of walls and two types of variations are present in the ample cone and where walls and semi-chambers corresponds to SIT-equivalence classes. We then spend the rest of \Cref{Toric_VSIT_Section} exploring the connections of VSIT to the combinatorics of the fan in \Cref{Relations to the structure of the fan} where another main result \Cref{thmC} is established.

Due to the computability offered by toric varieties, we wrote a computer program to generate a number of examples. \Cref{The computer program} explains the functions and logic of the program and in \Cref{Examples_and_counter_examples}, we present some examples computed by the program to address certain questions we have.
\subsection{Acknowledgements}
This paper is a part of my PhD thesis. I would like to thank my PhD advisor Jarod Alper for many insightful directions and his patience for guiding me through the revision of my PhD thesis. Many results in this paper were also inspired by discussions with Dan Edidin, Matthew Satriano, and encouragements from Zinovy Reichstein and Jack Hall. 
\section{Some linear algebra}\label{Some linear algebra}
This section records all results from linear programming and linear algebra necessary for later discussions. \Cref{Distance to linear subspaces} will be used when we discuss type two walls at \Cref{wall and semi-chamber}.
\subsection{Notations}\label{Notations}
Let $V$ be a finite dimensional real vector space endowed with an inner product $(-,-):V\times V\rightarrow\mathbf{R}.$ Let
$||-||:V\rightarrow\mathbf{R}$ be the induced norm on $V$. The space $V$ and
any subspace $W\subset V$ have the induced metric topology. Obviously the Pythagorean law holds: If $(v,w)=0$,
then \begin{equation}\label{Pythagorean}||v+w||^{2}=||v||^{2}+||w||^{2}.\end{equation} 

For any vector $v\in V$ and any subspace $W\subset V$, we set
$\Proj_{W}v$ to be \emph{the projection of $v$ onto $W$ along the orthogonal complement}
$W^{\perp}$. Namely, $$\Proj_{W}v\in W \text{ and }(v-\Proj_{W}v,w)=0\text{
  for all }w\in W.$$

If $f:V\rightarrow\mathbf{R}$ is a linear
functional, we let \begin{itemize}
    \item $f^{\ast}\in V$ be the unique vector in $V$ such that  
\begin{equation}\label{normal_vector}f(v)=(v,f^{\ast})\text{ for all }v\in V.\end{equation}
    \item $H_{f}$ be the hyperplane $\{v\in V\vert f(v)=0\}$, and 
    \item $H_{f}^{+}$ be the half-space $\{v\in V|f(v)\geq 0\}$.\end{itemize}

\emph{A polyhedral cone} in $V$ is an intersection of finitely many half-spaces.  Let $\sigma$ be a polyhedral cone. The hyperplane $H_{f}$ is a \emph{supporting hyperplane for} $\sigma$ if
$\sigma\subset H_{f}^{+}$. A \emph{face} $F$ of $\sigma$ is
$H_{f}\cap\sigma$ for some supporting hyperplane $H_{f}$ of
$\sigma$. Notation-wise we write $$F\preceq\sigma.$$ Obviously any
face of a polyhedral cone is a polyhedral cone. The \emph{relative interior}
of a polyhedral cone $\sigma$ is the interior of $\sigma$ in the
subspace spanned by $\sigma$ and is written as $$\Relint(\sigma).$$ We note without proof
that for a polyhedral cone $\sigma$, its collection of faces is finite
and \begin{equation}\sigma=\bigsqcup_{F\preceq\sigma}\Relint(F).\end{equation}\label{facialdecomposition}

Finally, for a linear functional $f:V\rightarrow\mathbf{R}$ and a polyhedral cone $\sigma\subset V$, we define 
the following finite set of vectors $$\Lambda^{f}_{\sigma}=\{-\Proj_{\Sp(F)}f^{\ast}\bigm|-\Proj_{\Sp(F)}f^{\ast}\in\sigma,F\preceq\sigma\}$$ where $\Sp(F)$ is the linear subspace spanned by $F$.

\subsection{Linear programming}\label{Linear programming}
This section presents results from linear programming that will be used extensively throughout the paper, especially \Cref{compute_cone_min}. Specifically Kempf's theorem \Cref{finitenesstheorem}, the structure of strata described by \Cref{strata_struct}, and the computation of stratifications, the formulation of walls later at \Cref{Toric_VSIT_Section} all depend on results from this section. 

The story begins with the following theorem whose proof is a direct modification of the proof presented in \cite{MR506989} and is omitted.
\begin{theorem}\label{unique_cone_min}
Let $S\subset V$ be the unit sphere and $\sigma\subset V$ be a polyhedral cone. 
Suppose $f:V\rightarrow\mathbf{R}$ is a linear functional that assumes a negative value at some point $v\in \sigma$. Then 
$f$ attains the relative minimum on $S\cap \sigma$ at a unique point $s$. In this case, if $F\preceq\sigma$ is the face of $\sigma$ such that $s\in\Relint(F)$, then $$s=-\frac{\Proj_{\Sp(F)}f^{\ast}}{||\Proj_{\Sp(F)}f^{\ast}||}.$$ 
\end{theorem}

\begin{corollary}\label{compute_cone_min}
Let $S\subset V$ be the unit sphere and $\sigma\subset V$ be a polyhedral cone. 
Suppose $f:V\rightarrow\mathbf{R}$ is a linear functional that attains a negative value at some point $v\in \sigma$. Let $s\in S\cap \sigma$ be the unique point where $f$ achieves the relative minimum on $S\cap \sigma$. Then there is a unique vector $v\in\Lambda^{f}_{\sigma}$ such that 
$$s=\frac{v}{||v||}.$$ In this case 
     $||v||\geq||v^{\prime}||$ for all $v^{\prime}\in\Lambda^{f}_{\sigma}$ and $$f(s)=-||v||.$$
\end{corollary}

\begin{proof}
For any $F\preceq \sigma$, we have \begin{equation*}\begin{split}
    f\bigg(\frac{-\Proj_{\Sp(F)}f^{\ast}}{||\Proj_{\Sp(F)}f^{\ast}||}\bigg)= &\frac{(f^{\ast},-\Proj_{\Sp(F)}f^{\ast})}{||\Proj_{\Sp(F)} f^{\ast}||}=\frac{(\Proj_{\Sp(F)}f^{\ast},-\Proj_{\Sp(F)}f^{\ast})}{||\Proj_{\Sp(F)}f^{\ast}||}\\
    =&-\frac{||\Proj_{\Sp(F)}f^{\ast}||^{2}}{||\Proj_{\Sp(F)}f^{\ast}||}=-||\Proj_{\Sp(F)}f^{\ast}||.
\end{split}\end{equation*}
Hence any $v\in\Lambda^{f}_{\sigma}$ that has the longest length will achieve the relative minimum of $f$ on $\sigma$ by \Cref{unique_cone_min}. The statement that there is a unique longest $v$ in $\Lambda^{f}_{\sigma}$ follows from the uniqueness of $s$. 
\end{proof}

Let $f:V\rightarrow\mathbf{R}$ be a linear functional that takes a negative value on $\sigma$. We see that \Cref{compute_cone_min} reduces the search range for $s\in\sigma\cap S$ where $f$ attains the relative minimum to the finite set $\Lambda^{f}_{\sigma}$. We will use \Cref{compute_cone_min} repeatedly in the context of toric GIT later at \Cref{Toric_VSIT_Section}.

We now specialize \Cref{unique_cone_min} with some rationality assumptions. \Cref{unique_cone_min_rational} below is the basis of Kempf's theorem \Cref{finitenesstheorem}. For this we fix a lattice $M\subset V$ such that $M\otimes_{\mathbf{Z}}\mathbf{R}=V$. A lattice point $m\in M$ is \emph{indivisible} if $m=N\cdot m^{\prime}$ for some $N\in\mathbf{N}$ and $m^{\prime}\in M$ implies $N=1$. \emph{A rational polyhedral cone (with respect to M)} is a polyhedral cone whose supporting hyperplanes are define by linear functionals that take integral values on $M$. 
\begin{corollary}\label{unique_cone_min_rational}
$M\subset V$ be a lattice with $M\otimes_{\mathbf{Z}}\mathbf{R}=V$, $\sigma\subset V$ be a
rational polyhedral cone and $S\subset V$ be the unit sphere. Suppose the inner
product $(-,-)$ takes integral values on $M\times M$ and $f:V\rightarrow\mathbf{R}$ is a
linear functional that takes integral values on $M$. If $f$ assumes a negative value on
$\sigma$ and $f$ attains the relative minimum at $s$ on $S\cap \sigma$, then the ray $\mathbf{R}_{>0}\cdot s$ contains a unique nonzero indivisible lattice point in $M$. 
\end{corollary}
\begin{proof}
It is easy to check that the rationality assumptions imposed here imply that the ray $\mathbf{R}_{>0}\cdot(-\Proj_{\Sp(F)}f^{\ast})$ intersects $M$ away from $0$ for each face $F\preceq\sigma$ with $\Proj_{\Sp(F)}f^{\ast}\neq 0$. Now apply the description of $s$ in \Cref{unique_cone_min}.
\end{proof}
Below is an elementary observation that will be used later in \Cref{Extensions_to_the_ample_cone} to formulate stratifications induced by $\mathbf{R}$-ample divisors on a projective toric variety.
\begin{lemma}\label{lem_for_indexing_projections}
Let $v\in V$ and $W_{1},W_{2}$ be two subspaces of $V$. Then the following two conditions for $\Proj_{W_{1}}v$ and $\Proj_{W_{2}}v$ are equivalent:
\begin{enumerate}
    \item $\Proj_{W_{1}}v$ and $\Proj_{W_{2}}v$ are linearly dependent.
    \item $\Proj_{W_{1}}v=\Proj_{W_{2}}v$,
\end{enumerate}
In this case, we actually have $\Proj_{W_{1}}v=\Proj_{W_{1}\cap W_{2}}v=\Proj_{W_{2}}v$
\end{lemma}
\begin{proof}
Condition (2) obviously implies condition (1). Let us show (1) implies (2). For this, condition (1) implies both $\Proj_{W_{1}}v$ and $\Proj_{W_{2}}v$ are in $W_{1}\cap W_{2}$. Since $\Proj_{W_{1}}v\in W_{1}\cap W_{2}$, we have 
$$\Proj_{W_{1}}v=\Proj_{W_{1}\cap W_{2}}(\Proj_{W_{1}}v)=\Proj_{W_{1}\cap W_{2}}v.$$ Similarly, we have $\Proj_{W_{2}}v=\Proj_{W_{1}\cap W_{2}}v$. The lemma is proved. 
\end{proof}

\subsection{Distance to linear subspaces}\label{Distance to linear subspaces}
This section is intended to understand the structure of type two walls described by \Cref{struct_walls}. However, readers do not need the knowledge from this section to understand the definition of type two walls given at \Cref{wall and semi-chamber}. Accordingly readers may return to this section if needed.

It turns out that type two walls defined in \Cref{wall and semi-chamber} corresponds to collections of vectors that solve an equi-distance problem which we describe in this section. The main result \Cref{type_two_wall_is_codim_1} of this section will be translated into \Cref{struct_walls} to describe the structure of type two walls for simplicial projective toric varieties. 

Let $v$ be a vector in $V$ and $W\subset V$ be a subspace. We define the distance from $v$ to $W$ to be 
$$\dist(v,W)=||v-\Proj_{W}v||.$$ Our interest is to describe the set of vectors 
$$\{v\in V\vert \dist(v,W_{1})=\dist(v,W_{2})\}$$ that are equi-distant to a pair of subspaces $W_{1},W_{2}$. First note that the definition of $\dist(v,W)$ makes sense for if $w\in W$ is a vector, then 
\begin{equation*}
    \begin{split}
        ||v-w||&=||v-\Proj_{W}v+\Proj_{W}v-w||\\
        &=\sqrt{||v-\Proj_{W}v||^{2}+||\Proj_{W}v-w||^{2}}\geq ||v-\Proj_{W}v||
    \end{split}
\end{equation*}
by the Pythagorean law \Cref{Pythagorean}.
Namely, $\Proj_{W}v$ is closest to $v$ among all vectors in $W$.

There is an explicit way to compute $\dist(v,W)$. Let
$\mathscr{B}=\{f_{1}^{\ast},\ldots,f_{\epsilon}^{\ast}\}$ be an orthonormal basis
for $W^{\perp}.$ Equivalently,
$W$ is cut out by the linear functionals $\{f_{1},\ldots, f_{\epsilon}\}$
corresponding to $\mathscr{B}$. Then $$v-\Proj_{W}v=\Proj_{W^{\perp}}v=\sum_{i=1}^{\epsilon}f_{i}(v)\cdot
f^{\ast}_{i}.$$ Therefore,
\begin{equation*}
\dist(v,W)=\sqrt{\sum_{i=1}^{\epsilon}f_{i}(v)^{2}}.
\end{equation*}
Now if $W_{1},W_{2}$ are two subspaces of $V$, let 
\begin{itemize} 
\item $\{f_{1}^{\ast},\ldots,f^{\ast}_{\epsilon}\}$ be a basis of $W_{1}^{\perp}\cap W_{2}^{\perp}$,
\item $\mathscr{A}=\{f_{1}^{\ast},\ldots,f_{\epsilon}^{\ast},g_{1}^{\ast},\ldots, g_{m}^{\ast}\}$ be a basis of $W_{1}^{\perp}$, and
\item $\mathscr{B}=\{f_{1}^{\ast},\ldots, f_{\epsilon}^{\ast},h_{1}^{\ast},\ldots, h_{n}^{\ast}\}$ be a basis of $W_{2}^{\perp}$.
\end{itemize}
Applying the Gram Schmidt process, we may assume both $\mathscr{A},\mathscr{B}$ are orthonormal. With these it follows that 
\begin{equation}\label{equi_distance_equation}\dist(v,W_{1})=\dist(v,W_{2})\Leftrightarrow\sum_{i=1}^{m}g_{i}(v)^{2}=\sum_{j=1}^{n}h_{j}(v)^{2}\end{equation} Moreover, it follows from the construction that $\mathscr{A}\cup\mathscr{B}$ is a basis of $W_{1}^{\perp}+W_{2}^{\perp}$ so it can be extended to a full basis $\mathscr{C}$ of $V$. With the coordinate induced by the basis $\mathscr{C}$, \Cref{equi_distance_equation} looks like 
\begin{equation}\label{equi-distant_simplified}x_{1}^{2}+\cdots+x_{m}^{2}=y_{1}^{2}+\cdots+y_{n}^{2}\end{equation} where $m$ and $n$ are codimensions of $W_{1}$ and $W_{2}$ in $W_{1}+W_{2}$ respectively.

We now introduce some terminologies from differential geometry to describe the structure of the set of equi-distant vectors. Let $F:N\rightarrow M$ be a $C^{\infty}$ map between real
manifolds. We say $c\in M$ is a \emph{regular value of F} if either
$c$ is not in the image of $F$ or at every point $p\in F^{-1}(c)$, the
differential $(\text{d}F)_{p}:T_{p}N\rightarrow T_{F(p)}M$ is
surjective. The inverse image $F^{-1}(c)$ of a regular value $c$ is
called a \emph{regular level set}.
We now recall the regular level set theorem from differential geometry:
\begin{theorem}[Regular level set theorem]\label{regularlevelset}
Let $F:N\rightarrow M$ be a $C^{\infty}$ map of manifolds, with $\dim
N=n$ and $\dim M = m$. Then a nonempty regular level set $F^{-1}(c)$,
where $c\in M$, is a regular submanifold of $N$ of dimension equal to $n-m$.
\end{theorem} 

We may now describe the equi-distant collection in $V$.

\begin{proposition}\label{type_two_wall_is_codim_1}
Let $W_{1}$ and $W_{2}$ be two sub-spaces of $V$ and let $Z$ be the
collection of points $v\in V$ that are equi-distant to $W_{1}$ and
$W_{2}$. If there is a containment, say $W_{1}\subset W_{2}$, then $Z$ is the linear subspace $W_{1}$. If there is no containment between $W_{1}$ and
  $W_{2}$, then $Z$ is a regular submanifold of $V$ of codimension 1 away from a subspace of codimension at least 2.
\end{proposition}

\begin{proof}
The case that there is a containment between $W_{1}$ and $W_{2}$ is clear. If there are no containments between $W_{1}$ and $W_{2}$, then $W_{1}+W_{2}$ properly contains $W_{1}$ and $W_{2}$. By \Cref{equi-distant_simplified}, $Z$
is defined by $x_{1}^{2}+\cdots +
x_{m}^{2}-y_{1}^{2}-\cdots-y_{n}^{2}=0$ with both $m,n>0$. Let $F:V\rightarrow
\mathbf{R}$ be the $C^{\infty}$ function $x_{1}^{2}+\cdots +
x_{m}^{2}-y_{1}^{2}-\cdots-y_{n}^{2}$. Let $V^{\prime}\subset V$ be the
subspace $V(x_{1},\cdots,x_{m},y_{1},\cdots,y_{n})$. Then
$0\in\mathbf{R}$ is a regular value for the restriction
$F:V-V^{\prime}\rightarrow\mathbf{R}$ and $F^{-1}(0)$ is obviously
a nonempty regular level set. By the regular level set theorem, 
\Cref{regularlevelset}, $Z-V^{\prime}$ is a regular submanifold of
codimension 1.
\end{proof}

\section{Instability in invariant theory}\label{Instability in invariant theory}
In this section we recall some fundamental results from IIT. This field is relatively old and has a well-established literature. A list of standard references to this subject may include \cite{MR506989}, \cite{MR555709}, \cite{MR553706}, and \cite{MR766741}.

IIT originates from Mumford's numerical criterion presented in \cite{MR1304906} that reduces testing stability for arbitrary reductive group actions to one dimensional torus actions given by one parameter subgroups. Mumford's conjecture on the existence of the \emph{maximally destabilizing} one parameter subgroups that fail the numerical criterion was proved by Kempf in his famous paper \cite{MR506989}. Beyond existence Kempf actually established the uniqueness of maximally destabilizing one parameter subgroups up to conjugacy by some parabolic subgroup. Hesselink then showed in \cite{MR553706} that the unstable locus is stratified by the conjugacy classes of Kempf's maximally destabilizing one parameter subgroups.

We will first recall Mumford's definition of stability (\Cref{Stability_def}) and his numerical criterion (\Cref{HilbMumAffine}) in \Cref{Definition and a numerical criterion for instability}. We then make precise the meaning of maxiamally destabilizing one parameter subgroups (\Cref{worst_one_PS_def}), and present Kempf's theorem (\Cref{finitenesstheorem}) at \Cref{Numerical analysis of instability}. With these we would be able to recall Hesselink's stratification induced by a character (\Cref{HesAffine}) and define variations of stratfications induced by different characters (\Cref{strat_def}, \Cref{two_types_of_variation_def}). Finally, we describe the structure of a stratum (\Cref{strata_struct}) at \Cref{The structure of a stratum}.

Instead of supplying IIT results in full generality, we will focus on finite dimensional representations of a diagonalizable group $G$ as this is the setting for toric GIT. In this case the set of one parameter subgroups of $G$ has a group structure of a lattice and because $G$ is abelian, the conjugacy class of a one parameter subgroup $\lambda$ is just the singleton $\{\lambda\}$. Hence there would be no ambiguity given by conjugacy about the uniqueness of the maximally destabilizing one parameter subgroups.
\subsection{Notations and conventions}\label{Notations and conventions}
Throughout the section $G$ is a diagonalizable algebraic group over the field $\mathbf{C}$ of complex numbers and $X=\mathbf{A}^{n}_{\mathbf{C}}$ is a representation of $G$. We write $\mathbf{C}[G]$ and $\mathbf{C}[X]$ as the coordinate rings of $G$ and $X$ respectively.

Letting $\mathbf{G}_{m}=\Spec\mathbf{C}[t]_{t}$ be the one dimensional torus over $\mathbf{C}$,  we set $$\pmb{\chi}(G):=\{\chi:G\rightarrow\mathbf{G}_{m}\}$$ to be the group of
characters of $G$ and $$\pmb{\Gamma}(G):=\{\lambda:\mathbf{G}_{m}\rightarrow G\}$$ to be the set of one parameter subgroups of
$G$. Since $G$ is abelian, the set $\pmb{\Gamma}(G)$ has a natural abelian group structure. We set for $\pmb{\chi}(G)$ the vector spaces $$\pmb{\chi}(G)_{\mathbf{Q}}:=\pmb{\chi}(G)\otimes_{\mathbf{Z}}\mathbf{Q},\text{ }\pmb{\chi}(G)_{\mathbf{R}}:=\pmb{\chi}(G)\otimes_{\mathbf{Z}}\mathbf{R}$$ and similarly for $\pmb{\Gamma}(G)$.

Let 
$$\langle-,-\rangle:\pmb{\chi}(G)\times\pmb{\Gamma}(G)\rightarrow\mathbf{Z}$$
be the pairing defined by the
formula $$\chi(\lambda(t))=t^{\langle\chi,\lambda\rangle}\text{ for
  all }t\in\mathbf{G}_{m}.$$
\begin{remark}\label{OPSforTori}
If $G=(\mathbf{G}_{m})^{r}=\{(t_{1},\ldots,t_{r})|t_{i}\in\mathbf{G}_{m}\}$ is a
torus, $\pmb{\Gamma}(G)$ can be identified as $\mathbf{Z}^{r}$ where each 
$(b_{1},\ldots,b_{r})\in\mathbf{Z}^{r}$ induces a one parameter
subgroup $\lambda:\mathbf{G}_{m}\rightarrow G$ given by  
$$t\mapsto (t^{b_{1}},\ldots, t^{b_{r}}).$$ Likewise $\pmb{\chi}(G)$ is identified as $\mathbf{Z}^{r}$ where each
$(a_{1},\ldots,a_{r})\in\mathbf{Z}^{r}$ defines a character $\chi:G\rightarrow\mathbf{G}_{m}$ given by 
$$(t_{1},\ldots,t_{r})\mapsto t_{1}^{a_{1}}\cdots t_{r}^{a_{r}}.$$  One easily computes that 
$\langle\chi,\lambda\rangle=\sum_{i}a_{i}b_{i}.$
Hence the natural pairing $\langle-,-\rangle$ is a perfect pairing.

If $G$ is diagonalizable, then $$G\simeq (\mathbf{G}_{m})^{r}\times F$$ for some finite group $F$. In this case $\pmb{\chi}(G)\simeq \mathbf{Z}^{r}\oplus P$ where $P$ is a finite abelian group and $\pmb{\Gamma}(G)\simeq\mathbf{Z}^{r}$. The pairing $\langle-,-\rangle$ when extended over $\mathbf{Q}$ or $\mathbf{R}$, is a perfect pairing.
\end{remark}

Since $G$ is diagonalizable, we assume the action of $G$ on $X$ is diagonalized. Namely, we fix $n$ characters $\chi_{1},\ldots,\chi_{n}$ of $G$ so that for every $g\in G$ and $x=(x_{1},\ldots,x_{n})\in X$, we have 
$$g\cdot x=(\chi_{1}(g)\cdot x_{1},\ldots,\chi_{n}(g)\cdot x_{n}).$$

Letting $[n]=\{1,\ldots,n\}$, we define for any subset $S\subset[n]$ the following polyhedral cone in $\pmb{\Gamma}(G)_{\mathbf{R}}$ that is rational with respect to the lattice $\pmb{\Gamma}(G)$:
$$\sigma_{S}=\{v\in\pmb{\Gamma}(G)_{\mathbf{R}}\vert \langle\chi_{i},v\rangle\geq 0\text{ for all }i\in S\}.$$

For $x=(x_{1},\ldots,x_{n})\in X$, we define $$S_{x}=\{i\in[n]\vert x_{i}\neq 0\}$$ as the \emph{states of }$x$ and the following polyhedral cone  
$$\sigma_{x}:=\sigma_{S_{x}}.$$

Given a one parameter subgroup $\lambda\in\pmb{\Gamma}(G)$, we define the subset $$X_{\lambda}=\{x\in X\vert \lim\limits_{t\rightarrow 0}\lambda(t)\cdot x\text{ exists}\},$$ which is the following linear subspace of $X$: 
$$\{(x_{1},\ldots, x_{n})\in X\mid x_{i}=0\text{ if }\langle\chi_{i},\lambda\rangle<0\}.$$

Finally, a \emph{norm} on $\pmb{\Gamma}(G)$ is a real valued function $||-||:\pmb{\Gamma}(G)\rightarrow\mathbf{R}$ such that there is an inner product $(-,-)$ on $\pmb{\Gamma}(G)_{\mathbf{R}}\times\pmb{\Gamma}(G)_\mathbf{R}\rightarrow\mathbf{R}$ that takes integral values on $\pmb{\Gamma}(G)\times\pmb{\Gamma}(G)$, and for any $\lambda\in\pmb{\Gamma}(G)$, we have $$||\lambda||=(\lambda,\lambda)^{1/2}.$$ Note that a norm  always exists. One can either choose an isomorphism $\pmb{\Gamma}(G)\simeq\mathbf{Z}^{r}$ and take the standard norm, or choose an embedding $G\hookrightarrow (\mathbf{C}^{\times})^{s}$, and restrict the standard norm to $\pmb{\Gamma}(G)$ via the inclusion $\pmb{\Gamma}(G)\hookrightarrow\pmb{\Gamma}((\mathbf{C}^{\times})^{s})$.

We fix a norm $||-||$ on $\pmb{\Gamma}(G)$. Moreover, abusing the notations, we also write $\langle-,-\rangle$ and $||-||$ for their extensions over $\mathbf{Q}$ or $\mathbf{R}$. 

\subsection{Definition and a numerical criterion for instability}\label{Definition and a numerical criterion for instability}
Mumford's stability condition depends on the choice of a linearized line bundle. In the case of group action on an affine space, linearizations of the trivial line bundle correspond to characters of the group. 
\begin{definition}\label{Stability_def}
Let $\chi:G\rightarrow\mathbf{G}_{m}$ be a character. Let 
\begin{enumerate}
\item $\hat{\sigma}:\mathbf{C}[X]\rightarrow
\mathbf{C}[G]\otimes_{\mathbf{C}} \mathbf{C}[X]$ be the co-action, and 
\item $\chi^{\sharp}:\mathbf{C}[t]_{t}\rightarrow \mathbf{C}[G]$ be the
  map that corresponds to $\chi$.
\end{enumerate}
An element $f\in \mathbf{C}[X]$ is $\chi$-\emph{invariant of weight d} if $\hat{\sigma}(f)=\chi^{\sharp}(t)^{d}\otimes f$. We say a point $x\in X$
is $\chi$-\emph{semistable} if there is a $\chi$-invariant $f$ of positive
weight such that $f(x)\neq 0$. We say a point $x\in X$ is $\chi$-\emph{unstable} if $x$ is not $\chi$-semistable.
We write $X^{\text{ss}}(\chi)$ as the set of $\chi$-semistable points and $X^{\text{us}}(\chi)$ as the complement $X-X^{\text{ss}}(\chi).$
\end{definition}

It follows from the definition that $X^{\text{ss}}(\chi)$ is a $G$-invariant open subvariety and that $X^{\text{us}}(\chi)$ is a $G$-invariant closed subvariety. Moreover, $X^{\text{ss}}(\chi)=X^{\text{ss}}(\chi^{d})$ for any $d>0$. 

Here is the GIT quotient construction of $X$ by the group $G$ with respect to $\chi$: Let $\mathbf{C}[X]_{\chi,d}$ be the space of $\chi$-invariant elements of weight $d$. The space $\oplus_{d\geq 0}\mathbf{C}[X]_{\chi,d}$ has a natural graded ring structure. We define  $$X/\!\!/_{\chi}G:=\Proj(\oplus_{d\geq 0}\mathbf{C}[X]_{\chi,d}).$$
Then there is a map $X^{\text{ss}}(\chi)\rightarrow X/\!\!/_{\chi}G$ that is constant on $G$-orbits, submersive and induces a bijection between  points in $X/\!\!/_{\chi}G$ and closed orbits in $X^{\text{ss}}(\chi)$ (see \cite{MR1304906} for more detail). Moreover, $X/\!\!/_{\chi}G$ is a quasi-projective variety that is known as \emph{the GIT quotient of X by G with respect to }$\chi$. 

An alternative way to test stability is to restrict the group action to one parameter subgroups. Let $\lambda$ be a one parameter subgroup of $G$. We say $\lim\limits_{t\rightarrow 0}\lambda(t)\cdot x$ exists if the domain of the map $\lambda_{x}:\mathbf{G}_{m}\rightarrow X$ defined by $t\mapsto \lambda(t)\cdot x$ can be extended to $\mathbf{A}^{1}_{\mathbf{C}}.$
\begin{theorem}(Hilbert-Mumford criterion. \cite{MR1304906}, \cite{MR1315461})\label{HilbMumAffine}
A point $x\in X$ is $\chi$-semistable if and only if for each one 
parameter subgroup $\lambda:\mathbf{G}_{m}\rightarrow G$ such that the
limit $\lim\limits_{t\rightarrow 0}\lambda(t)\cdot x$ exists,
we have $\langle\chi,\lambda\rangle\geq 0$.
\end{theorem}

\begin{example}
Let the one dimensional torus $\mathbf{G}_{m}$ acts on $\mathbf{A}^{2}_{\mathbf{C}}$ by $$t\cdot(x,y)=(tx,t^{-1}y)\text{ for all }t\in\mathbf{G}_{m}\text{ and for all }(x,y)\in\mathbf{A}^{2}_{\mathbf{C}}.$$ Let $\chi$ be the identity $\text{id}:\mathbf{G}_{m}\rightarrow\mathbf{G}_{m}$. We will check $\chi$-semistability with $\chi$-invariants and then with the numerical criterion \Cref{HilbMumAffine}.

To begin with, an element $f\in\mathbf{C}[x,y]=\mathbf{C}[\mathbf{A}^{2}_{\mathbf{C}}]$ is $\chi$-invariant of weight $d$ if and only if $f(tx,t^{-1}y)=t^{d}\cdot f(x,y).$ This is equivalent to saying that $f\in x^{d}\cdot\mathbf{C}[xy]$. In particular, $f$ is divisible by $x^{d}$. This shows that $(\mathbf{A}^{2}_{\mathbf{C}})^{\text{ss}}(\chi)\subset D(x)$. Conversely, $x\in\mathbf{C}[x,y]$ is $\chi$-invariant of weight 1. It follows that $(\mathbf{A}^{2}_{\mathbf{C}})^{\text{ss}}(\chi)\supset D(x)$. In conclusion, we get $(\mathbf{A}^{2}_{\mathbf{C}})^{\text{ss}}(\chi)= D(x)$.

To test stability with \Cref{HilbMumAffine}, let $\lambda:\mathbf{G}_{m}\rightarrow\mathbf{G}_{m}$ be a one parameter subgroup given by $t\mapsto t^{a}$ and $p=(x,y)$ be a point in $\mathbf{A}^{2}_{\mathbf{C}}$. Suppose $\lim\limits_{t\rightarrow 0}\lambda(t)\cdot p=(t^{a}x,t^{-a}y)$ exists. If $x\neq 0$, we have $a=\langle\chi,\lambda\rangle\geq 0$. We see that $D(x)\subset(\mathbf{A}^{2}_{\mathbf{C}})^{\text{ss}}(\chi)$. Conversely, if $x=0$, then the one parameter subgroup given by $a=-1$ has limits at $p$ but $\langle\chi,\lambda\rangle=-1<0$. This shows that $V(x)\subset(\mathbf{A}^{2}_{\mathbf{C}})^{\text{us}}(\chi)$. In conclusion, we have $(\mathbf{A}^{2}_{\mathbf{C}})^{\text{ss}}(\chi)= D(x)$ as was shown earlier.

One also easily checks that $(\mathbf{A}^{2}_{\mathbf{C}})^{\text{ss}}(-\chi)=D(y)$, indicating the dependence of stability of a point on the choice of linearizations.
\end{example}

\subsection{Numerical analysis of instability}\label{Numerical analysis of instability}
In this section we are going to make precise in \Cref{worst_one_PS_def} on what it means for a one parameter subgroup to be a maximally destabilizing one, then recall Kempf's main theorem \Cref{finitenesstheorem}.

For a point $x\in X$, set $$C_{x}=\{\lambda\in\pmb{\Gamma}(G)|\lim_{t\rightarrow
  0}\lambda(t)\cdot x\text { exits}\}.$$
According to the numerical criterion, \Cref{HilbMumAffine}, $x\in X^{\text{us}}(\chi)$ if and only if there is a one parameter subgroup $\lambda$ such that 
\begin{enumerate}
    \item $\lambda\in C_{x}$, and 
    \item $\langle\chi,\lambda\rangle<0$.
\end{enumerate}
Therefore, for an $x\in X^{\text{us}}(\chi)$, it is natural to ask if there is a one parameter subgroup that contributes to the highest instability measured by the negative quantities $\langle\chi,\lambda\rangle$ among all $\lambda\in C_{x}$. An immediate problem is that $\langle\chi,\lambda^{N}\rangle=N\cdot\langle\chi,\lambda\rangle$ for any $\lambda\in\pmb{\Gamma}(G)$ and $N\in\mathbf{N}$. To get rid of the dependency on multiples of one parameter subgroups, we divide the function $\langle\chi,-\rangle:\pmb{\Gamma}(G)\rightarrow\mathbf{Z}$ by the norm $||-||$ on $\pmb{\Gamma}(G)$. For $x\in X^{\text{us}}(\chi)$, we set $$M^{\chi}(x)=
  \inf_{\lambda\in C_{x}\backslash\{0\}}\frac{\langle\chi,\lambda\rangle
}{||\lambda||}.$$ 
\begin{definition}\label{worst_one_PS_def} 
We say a one parameter subgroup $\lambda$ is $\chi$-\emph{adapted to }$x$ if
$\frac{\langle\chi,\lambda\rangle}{||\lambda||}= M^{\chi}(x)$. Moreover, we say $\lambda$ is $\chi$-\emph{adapted to a subset }$S\subset X^{\text{us}}(\chi)$ if $\lambda$ is $\chi$-adapted to every point in $S.$
\end{definition}
Note that it is not immediate that $M^{\chi}(x)$ is finite. Even so it is not immediate that the infimum is attained by a one parameter subgroup 
from  $C_{x}$. The following \Cref{finitenesstheorem}, a special case of a theorem due to Kempf, resolves these issues.

\begin{theorem}[Kempf, \cite{MR506989}]\label{finitenesstheorem}
Let $\chi:G\rightarrow\mathbf{G}_{m}$ be a character. For an $x\in X^{\text{us}}(\chi)$, we have  
\begin{enumerate}
\item the value $M^{\chi}(x)$ is finite,
\item the function
$M^{\chi}(-):X^{\text{us}}(\chi)\rightarrow\mathbf{R}$ assumes finitely many values,
\item there is a unique indivisible one parameter subgroup $\lambda_{\chi,x}$ that is $\chi$-adapted to $x$,
\item for any $g\in G$, $\lambda_{\chi,g\cdot x}=\lambda_{\chi,x}$, and in particular, 
\item $M^{\chi}(-)$ is constant
on $G$-orbits. Namely, for each $g\in G$ and $x\in X^{\text{us}}(\chi)$, we have 
$$M^{\chi}(x)=M^{\chi}(g\cdot x).$$
\end{enumerate}
\end{theorem}

\Cref{finitenesstheorem} is a consequence of \Cref{unique_cone_min_rational}. The connection is made by the 
\begin{lemma}\label{torusLP}
For each $x\in X$, $C_{x}\otimes_{\mathbf{Z}}\mathbf{R}\subset\pmb{\Gamma}(G)_{\mathbf{R}}$ is the rational polyhedral cone $\sigma_{x}$. Equivalently, the set of lattice points of $\sigma_{x}$ is exactly $C_{x}$.
\end{lemma}

\begin{proof}
Recall that the action is given by $$g\cdot (x_{1},\ldots,x_{n})=(\chi_{1}(g)\cdot x_{1},\ldots,\chi_{n}(g)\cdot x_{n})$$ for all $g\in G$ and $(x_{1},\ldots,x_{n})\in X$. 
If $\lambda:\mathbf{G}_{m}\rightarrow G$ is a one parameter subgroup and
if $t\in \mathbf{G}_{m}$, we get 
$$\lambda(t)\cdot(x_{1},\ldots,x_{n})=(t^{\langle
  \chi_{1},\lambda\rangle}x_{1},\ldots,t^{\langle\chi_{n},\lambda\rangle}x_{n}).$$
Hence $$\lim\limits_{t\rightarrow 0}\lambda(t)\cdot x\text{ exists }\Leftrightarrow\langle\chi_{i},\lambda\rangle\geq 0\text{ for all }i\text{ with }x_{i}\neq 0\Leftrightarrow \lambda\in \sigma_{x}.$$
\end{proof}

We now demonstrate how \Cref{finitenesstheorem} can be proved.
\begin{proof}Let $$f(-):=\langle\chi,-\rangle:\pmb{\Gamma}(G)_{\mathbf{R}}\rightarrow\mathbf{R}$$ be the linear functional. The assumption that $x$ is $\chi$-unstable implies that $f$ takes a negative value on $C_{x}$ and therefore on $\sigma_{x}$ by \Cref{torusLP}. \Cref{unique_cone_min_rational} then implies that
$$M^{\chi}(x)=\inf_{\lambda\in\sigma_{x}\backslash\{0\}}\frac{f(\lambda)}{||\lambda||},$$ and that $M^{\chi}(\chi)$ is attained at a unique indivisible one parameter subgroup $\lambda_{\chi,x}$. This proves statement (1) and (3). Statement (2) follows from the fact that there are only finitely many $\sigma_{x}$ as $x$ runs through $X^{\text{us}}(\chi)$. For statements (4), simply note that $\sigma_{x}=\sigma_{g\cdot x}$ for all $g\in G$.\end{proof}

\subsection{Stratification induced by a character}\label{Stratification of the null cone}
In this section we recall Hesselink's theorem \Cref{HesAffine} and define stratifications (\Cref{strat_def}), the equivalences and variations among them on a topological space (\Cref{two_types_of_variation_def}) in preparation of VSIT at \Cref{Toric_VSIT_Section}.

Let $\chi\in\pmb{\chi}(G)$ and $\lambda\in\pmb{\Gamma}(G)$. We define the following subset of $X$
$$S_{\lambda}^{\chi}=\{x\in X^{\text{us}}(\chi)\vert \lambda_{\chi,x}=\lambda\}$$ and the subset of $\pmb{\Gamma}(G)$
$$\Lambda^{\chi}=\{\lambda_{\chi,x}\vert x\in X^{\text{us}}(\chi)\}.$$ We point out that $\Lambda^{\chi}$ is a finite set. This follows from the discussion we had at the end of the previous section that each $\lambda_{\chi,x}$ comes from minimizing the function $\frac{\langle\chi,-\rangle}{||-||}:\pmb{\Gamma}(G)_{\mathbf{R}}\rightarrow\mathbf{R}$ on $\sigma_{x}$ and that there are only finitely many $\sigma_{x}$ as $x$ runs through $X^{\text{us}}(\chi).$ With these notations and \Cref{finitenesstheorem}, we have a decomposition 
$$X=X^{\text{ss}}(\chi)\cup(\bigcup_{\lambda\in\Lambda^{\chi}}S_{\lambda}^{\chi}).$$ Let us define a strict partial ordering on the set $\Lambda^{\chi}$ by setting $$\lambda>\lambda^{\prime}\text{ if }\frac{\langle\chi,\lambda\rangle}{||\lambda||}<\frac{\langle\chi,\lambda^{\prime}\rangle}{||\lambda^{\prime}||}.$$ For the convenience of this paper, we also define the same strict partial ordering on the collection $\{S_{\lambda}^{\chi}\}_{\lambda\in\Lambda^{\chi}}$ by $$S_{\lambda}^{\chi}>S_{\lambda^{\prime}}^{\chi}\Leftrightarrow\lambda>\lambda^{\prime}.$$ 
\begin{remark}
The reason why there is a flip of inequalities defining the strict partial ordering on $\Lambda^{\chi}$ is that $\frac{\langle\chi,\lambda\rangle}{||\lambda||}<\frac{\langle\chi,\lambda^{\prime}\rangle}{||\lambda^{\prime}||}$ implies the left hand side is more negative. Namely, the stratum $S^{\chi}_{\lambda}$ contains points that are higher in instability measured by the function $\frac{\langle\chi,-\rangle}{||-||}.$
\end{remark} The following theorem due to Hesselink states:

\begin{theorem}(Hesselink, \cite{MR553706})\label{HesAffine}
Let $\chi$ be a character of $G.$ Then 
$$X=X^{\text{ss}}(\chi)\cup(\bigcup_{\lambda}S_{\lambda}^{\chi})$$ is a finite disjoint union of $G$-invariant, locally closed subvarieties of $X$. Moreover, $$S^{\chi}_{\lambda}\cap \partial S_{\lambda^{\prime}}^{\chi}\neq\emptyset \text{ only if } \lambda>\lambda^{\prime}.$$
\end{theorem}

\begin{proof}
That each $S_{\lambda}^{\chi}$ is $G$-invariant follows from statement (4) of \Cref{finitenesstheorem}. Note that the last statement of the theorem we are proving implies for any $\lambda\in\Lambda^{\chi}$, we have 
   $$S^{\chi}_{\lambda}=\overline{S^{\chi}_{\lambda}}\backslash\bigcup_{\lambda^{\prime\prime}>\lambda}S^{\chi}_{\lambda^{\prime\prime}},$$ giving us the locally closedness for each stratum. Hence it remains to show the last statement. To do this, let $z\in S^{\chi}_{\lambda}\cap\partial S^{\chi}_{\lambda^{\prime}}\subset \overline{S^{\chi}_{\lambda^{\prime}}}$. Since $X_{\lambda^{\prime}}$ is closed and contains $S^{\chi}_{\lambda^{\prime}}$, we have $z\in X_{\lambda^{\prime}}$. Namely, $\lim\limits_{t\rightarrow 0}\lambda^{\prime}(t)\cdot z$ exists. Since $\lambda\neq\lambda^{\prime}$, it must be the case that $\frac{\langle\chi,\lambda\rangle}{||\lambda||}<\frac{\langle\chi,\lambda^{\prime}\rangle}{||\lambda^{\prime}||}.$
\end{proof}

The decomposition in \Cref{HesAffine} induces a stratification of $X$, which we now define. 
\begin{definition}\label{strat_def}
Let $Y$ be a topological space. A finite collection of locally closed subspaces  $\{Y_{a}\mid a\in\mathscr{A}\}$ forms a \emph{stratification of }$Y$ if $Y$ is a disjoint union of the strata $Y_{a}$ and there is a strict partial order on the index set $\mathscr{A}$ such that 
$\partial Y_{a^{\prime}}\cap Y_{a}\neq\emptyset$ only if $a>a^{\prime}$.
\end{definition}

For each character $\chi$, we index the semistable locus $X^{\text{ss}}(\chi)$ by the trivial one parameter subgroup $\mathbf{e}$ and assign $\mathbf{e}$ the new lowest order $0$ (as in zero instability) in the set $\Lambda^{\chi}\sqcup\{\mathbf{e}\}$. Doing so makes the decomposition of $X$ in \Cref{HesAffine}  a stratification. We refer to the stratification as \emph{the stratification induced by }$\chi$ and each stratum as a $\chi$-\emph{stratum}. 
Obviously the stratification depends on the choice of a character. To compare stratifications, we make the following  

\begin{definition}\label{two_types_of_variation_def}
We say two stratifications $\{Y_{a}\mid a\in\mathscr{A}\}$ and $\{Y_{b}\mid b\in\mathscr{B}\}$ of a topological space $Y$ are equivalent if there is a bijection $\Phi:\mathscr{A}\rightarrow\mathscr{B}$ such that 
\begin{enumerate}
    \item $\Phi$ preserves strata. That is, $Y_{\Phi(a)}=Y_{a}$ for all $a\in\mathscr{A}$, and  
    \item $\Phi$ preserves order. That is, $\Phi(a)>\Phi(a^{\prime})$ if and only if $a>a^{\prime}$ for all $a,a^{\prime}\in\mathscr{A}$.
\end{enumerate} 
Otherwise, we say
\begin{enumerate}
    \item there is a \emph{type one variation} between $\{Y_{a}\mid a\in\mathscr{A}\}$ and $\{Y_{b}\mid b\in\mathscr{B}\}$ if there is no bijection between $\mathscr{A}$ and $\mathscr{B}$ that satisfies condition (1), or
    \item there is a \emph{type two variation} between $\{Y_{a}\mid a\in\mathscr{A}\}$ and $\{Y_{b}\mid b\in\mathscr{B}\}$ if there is a bijection between $\mathscr{A}$ and $\mathscr{B}$ that satisfies condition (1), but not condition (2).
\end{enumerate}
\end{definition}

\begin{definition}\label{IIT_equi_class}
We say two characters $\chi_{1}$ and $\chi_{2}$ of $G$ are \emph{SIT}-\emph{equivalent} with respect to the action of $G$ on $X$ if
$\chi_{1}$ and $\chi_{2}$ induce equivalent stratifications of $X$.
\end{definition}

\subsection{The structure of a stratum}\label{The structure of a stratum}
In this section, we prove in \Cref{strata_struct} that for a character $\chi\in\pmb{\chi}(G)$ and for each $\lambda\in\Lambda^{\chi}$, the closure 
$\overline{S^{\chi}_{\lambda}}$ of the stratum $S^{\chi}_{\lambda}$ is the subspace $X_{\lambda}$, and 
$S^{\chi}_{\lambda}$ is obtained by removing a union of subspaces from $\overline{S^{\chi}_{\lambda}}$. This implies each stratum is irreducible and smooth. \Cref{strata_struct} is also going to be applied to \Cref{toric_quotient_strata} for the description of strata in the toric GIT setting. We first introduce some notations. 
\begin{itemize}
    \item If $S\subset [n]$, we let 
    \begin{itemize}
        \item $V(S)=V(\{x_{i}\vert i\in S\})\subset X$;
        \item $D(S)=D(\prod_{i\in S}x_{i})\subset X$;
        \item $L(S)=D(S)\cap V([n]-S)$;
        \end{itemize}
    \end{itemize} and if $A\subset C\subset[n]$, we set 
\begin{itemize}[resume]\item $L(A;C)=\cup_{B}L(B)$ where the union is over all $B$ with
  $A\subset B\subset C$.
\end{itemize}
Note that $L(A;C)=V([n]-C)\cap D(A)$ so all subsets of $X$ listed above are irreducible locally closed $G$-invariant subvarieties. 
With these notations, we have the immediate observation:

\begin{proposition}\label{stratum_structure_prelim}
Let $\chi\in\pmb{\chi}(G)$ be a character and $\lambda\in\Lambda^{\chi}$. The following statements for a point $x\in X$ are equivalent:
\begin{enumerate}
    \item $x\in S^{\chi}_{\lambda}$,
    \item the function $\frac{\langle\chi,-\rangle}{||-||}$ attains negative relative minimum on $\sigma_{x}$ at $\lambda$, and 
    \item $L(S_{x})\subset S^{\chi}_{\lambda}$.
\end{enumerate}
\end{proposition}

\begin{proof}
That statement (1) is equivalent to statement (2) follows from the proof of \Cref{finitenesstheorem}. To prove statement (1) implies statement (3), simply note that $S_{y}=S_{x}$ so that $\sigma_{y}=\sigma_{x}$ for all $y\in L(S_{x})$. Now the equivalence of statement (1) and statement (2) implies $y\in S^{\chi}_{\lambda}$ for all $y\in L(S_{x})$ if $x\in S^{\chi}_{\lambda}$. Finally statement (3) implies statement (1) because $x\in L(S_{x})$. Therefore, statement (1) is equivalent to statement (3).
\end{proof}

We obtain the first preliminary observation of a stratum:

\begin{corollary}\label{stratum_structure_prelim_cor}
Let $\chi\in\pmb{\chi}(G)$ be a character. Then for each $\lambda\in\Lambda^{\chi}$, the stratum $S_{\lambda}^{\chi}$ is a finite disjoint union of $G$-invariant locally closed subvarieties of the form $L(S)$ for some $S\subset[n]$. 
\end{corollary}

Next, we prove that among all subsets $S\subset[n]$ such that $L(S)\subset S^{\chi}_{\lambda}$, there is a unique maximal one.

\begin{lemma}\label{stratum_comb}
Let $\chi\in\pmb{\chi}(G)$ be a character. For every $\lambda\in\Lambda^{\chi}$, there is a unique maximal subset $\mathcal{M}\subset [n]$ with respect to the property that $L(\mathcal{M})\subset S^{\chi}_{\lambda}$. Moreover, $\mathcal{M}=\{i\in[n]\vert \langle\chi_{i},\lambda\rangle\geq 0\}$. In particular, $V([n]-\mathcal{M})= X_{\lambda}$.
\end{lemma}

\begin{proof}
Let $\mathcal{M}=\{i\in[n]\vert \langle\chi_{i},\lambda\rangle\geq 0\}$.
Let $S\subset[n]$ be a subset such that $L(S)\subset S^{\chi}_{\lambda}$. We need to show that 
\begin{enumerate}
    \item $S\subset \mathcal{M}$, and 
    \item $L(\mathcal{M})\subset S^{\chi}_{\lambda}$.
\end{enumerate}
For statement (1), since $S^{\chi}_{\lambda}\subset X_{\lambda}$, we have $L(S)\subset X_{\lambda}=V([n]-\mathcal{M})$. This implies $S\subset \mathcal{M}$, justifying statement (1). This in turn implies $\sigma_{\mathcal{M}}\subset \sigma_{S}$. For statement (2), note that $\lambda$ is by definition in $\sigma_{\mathcal{M}}$. In sum, we obtain 
$$\lambda\in\sigma_{\mathcal{M}}\subset\sigma_{S}.$$ By \Cref{stratum_structure_prelim}, the function $\frac{\langle\chi,-\rangle}{||-||}$ attains relative minimum on $\sigma_{S}$ at $\lambda$. Therefore, the function $\frac{\langle\chi,-\rangle}{||-||}$ also attains relative minimum on $\sigma_{\mathcal{M}}$ at $\lambda$. This proves statement (2).
\end{proof}

We also note that a stratum has the following absorption property:

\begin{lemma}\label{stratum_absorb}
Let $\chi\in\pmb{\chi}(G)$ be a character and $\lambda\in\Lambda^{\chi}$. Then if $C\subset A\subset [n]$ is a pair of subsets such that both $L(C)$ and $L(A)$ are contained in $S^{\chi}_{\lambda}$, we have $L(B)\subset S^{\chi}_{\lambda}$ for all $B$ with $C\subset B\subset A$.
\end{lemma}

\begin{proof}
Suppose we have $C\subset B\subset A$ and both $L(C)$ and $L(A)$ are contained in $S^{\chi}_{\lambda}$. Then we have 
$$\lambda\in\sigma_{A}\subset\sigma_{B}\subset\sigma_{C}.$$ The lemma is now a direct result of \Cref{stratum_structure_prelim}.
\end{proof}

Finally, we are ready to describe a stratum:

\begin{theorem}\label{strata_struct}
Let $\chi\in\pmb{\chi}(G)$ be a character and $\lambda\in\Lambda^{\chi}$. 
Let $\mathcal{M}$ be the maximal subset of $[n]$ as in 
\Cref{stratum_comb}. If each $A_{j}$ for $j=1,\ldots,N$ is a subset
of $[n]$ that is minimal with respect to the property that $L(A_{j})\subset S^{\chi}_{\lambda}$, then 
$$S^{\chi}_{\lambda}=V([n]-\mathcal{M})\bigcap\bigg(\bigcup_{j=1}^{N}
D(A_{j})\bigg)=X_{\lambda}\bigcap\bigg(\bigcup_{j=1}^{N}
D(A_{j})\bigg).$$ In particular, $S^{\chi}_{\lambda}$ as an open subvariety of $X_{\lambda}$, is irreducible,
connected and  
smooth. Moreover, $S^{\chi}_{\lambda}$ is obtained by removing a finite union of linear subspaces from  $X_{\lambda}$ and the zariski closure $\overline{S^{\chi}_{\lambda}}$ is the linear subspace $X_{\lambda}$.
\end{theorem}

\begin{proof}
By \Cref{stratum_structure_prelim_cor} and \Cref{stratum_absorb}, we have $$S^{\chi}_{\lambda}=\bigcup_{j=1}^{N}L(A_{j};\mathcal{M})=\bigcup_{j=1}^{N}\bigg(V([n]-\mathcal{M})\cap D(A_{j})\bigg)=V([n]-\mathcal{M})\bigcap\bigg(\bigcup_{j=1}^{N}D(A_{j})\bigg).$$ The theorem is a direct result of the above computation.
\end{proof}

\section{Recap on toric varieties}\label{Recap on toric varieties}
The main reference for this section is \cite{MR2810322}. In this section we recall the fact due to Cox that every projective toric variety $X$ is the GIT quotient of an affine space by a diagonalizable group with respect to characters induced by ample divisors on $X$. Our toric varieties are assumed to be normal. In this case a toric variety is constructed from a fan. We refer the reader for the definition of a fan and how to construct a toric variety from a fan to chapter 3 in \cite{MR2810322}.

Starting at \Cref{Set up}, we adopt notations and conventions from \cite{MR2810322} and recall the finite presentation (\ref{sesWDiv}) of the divisor class group of a toric variety.  With (\ref{sesWDiv}) we derive at \Cref{lemma5.1.1} several properties of the diagonalizable group that occurs in the quotient construction of toric varieties. We then briefly recall a numerical criterion, \Cref{numampleness} for a torus invariant divisor to be ample and describe in \Cref{nef_and_ample_cone_thm} the cone of ample divisors on a projective toric variety. Finally, in \Cref{GIT quotient construction of projective toric varieties} we recall Cox's GIT quotient construction for projective toric varieties at \Cref{thm5.1.11} and \Cref{bad_locus_ample}.
\subsection{Set up}\label{Set up}
Let $N$ be a lattice and $M=\Hom_{\mathbf{Z}}(N,\mathbf{Z})$ be the dual. Let 
$\langle-,-\rangle:M\times N\rightarrow \mathbf{Z}$ be the perfect pairing where $\langle
m,n\rangle=m(n)$. Set $N_{\mathbf{R}}$ and
$M_{\mathbf{R}}$ to be $N\otimes_{\mathbf{Z}}\mathbf{R}$ and
$M\otimes_{\mathbf{Z}}\mathbf{R}$ respectively. 

Let $\Sigma\subset\mathbf{N}_{\mathbf{R}}$ be a fan. We write $X_{\Sigma}$ as the toric variety constructed from the fan $\Sigma$. In this case
$T_{N}:=N\otimes_{\mathbf{Z}}\mathbf{C}^{\times}$ is the torus embedded
in $X_{\Sigma}$. The lattice $N$ is realized as the group of one parameter
subgroups $\pmb{\Gamma}(T_{N})$ of $T_{N}$ in the following way: given an element $n\in N$,
the assignment $\mathbf{C}^{\times}\rightarrow T_{N}$ defined by $$t\mapsto
n\otimes t$$ is a one parameter subgroup of $T_{N}$. The lattice $M$
is realized as the group of characters $\pmb{\chi}(T_{N})$ of $T_{N}$ in the following way: given an
element $m\in M$ and an element $n\otimes t\in T_{N}$, the assignment
$T_{N}\rightarrow\mathbf{C}^{\times}$ defined by $$n\otimes t\mapsto
t^{\langle m,n\rangle}$$ is a character of $T_{N}$. Under these identifications the pairing $\langle-,-\rangle:M\times N\rightarrow\mathbf{Z}$ agrees with the pairing defined for $\pmb{\chi}(T_{N})$ and $\pmb{\Gamma}(T_{N})$ in \Cref{Notations and conventions}.

By $\Sigma(i)$ we
mean the collection of $i$-dimensional cones in $\Sigma$ and $\Sigma_{\text{max}}$ means the collection of maximal cones in $\Sigma$. For each ray $\rho\in\Sigma(1)$, we write $u_{\rho}\in N$ as the indivisible lattice point that spans $\rho$ over $\mathbf{R}$. If $\sigma$ is a polyhedral cone, $\sigma(1)$ means the collection of extremal rays of $\sigma$. Namely, $\sigma(1)$ consists of one dimensional faces of $\sigma$.

By the orbit-cone correspondence (theorem 3.2.6 in \cite{MR2810322}), each ray $\rho\in\Sigma(1)$ corresponds to a $T_{N}$-invariant divisor $D_{\rho}\subset X_{\Sigma}$. 
Let $$\mathbf{Z}^{\Sigma(1)}=\bigoplus_{\rho\in\Sigma(1)}\mathbf{Z}\cdot D_{\rho}$$ be the free abelian group generated by the divisors $D_{\rho}$ for $\rho\in\Sigma(1)$. Each $m\in M$,
viewed as a character $\chi^{m}$ of $T_{N}$, is a rational function on
$X_{\Sigma}$. The principal divisor $\text{div}(\chi^{m})$ satisfies the formula 
\begin{equation}\label{DivChar}
\text{div}(\chi^{m})=\sum_{\rho\in\Sigma(1)}\langle
m,u_{\rho}\rangle\cdot D_{\rho}.
\end{equation} It is proved in theorem 4.1.3 in \cite{MR2810322} that
the natural map
$\mathbf{Z}^{\Sigma(1)}\rightarrow\text{Cl}(X_{\Sigma})$ is surjective
with kernel the image of the map $M\rightarrow\mathbf{Z}^{\Sigma(1)}$ defined by $m\mapsto \text{div}(\chi^{m})$. Namely, we have an exact sequence 
\begin{equation}\label{sesWDiv}
\begin{tikzcd}[sep =small]
 M\arrow[r] & \mathbf{Z}^{\Sigma(1)}\arrow[r] &\text{Cl}(X_{\Sigma})\arrow[r] & 0
\end{tikzcd}.
\end{equation}
Moreover, it is exact on the left if $\Sigma$ is complete. Let $\text{CDiv}_{T_{N}}(X_{\Sigma})$ be the group of torus
invariant Cartier divisors on $X_{\Sigma}$. Then there is the exact sequence \begin{equation}\label{sesCDiv}
\begin{tikzcd}
[sep =small]
 M\arrow[r] & \text{CDiv}_{T_{N}}(X_{\Sigma})\arrow[r] &\text{Pic}(X_{\Sigma})\arrow[r] & 0
\end{tikzcd}
\end{equation}
that fits into the commutative diagram of
short exact sequences:
$$\begin{tikzcd}
M\arrow[r]\arrow[d,equal] & \text{CDiv}_{T_{N}}(X_{\Sigma})\arrow[r]\arrow[d,hookrightarrow] &\text{Pic}(X_{\Sigma})\arrow[r]\arrow[d,hookrightarrow] & 0\\
M\arrow[r] & \mathbf{Z}^{\Sigma(1)}\arrow[r] &\text{Cl}(X_{\Sigma})\arrow[r] & 0
\end{tikzcd}.$$
When $\Sigma$ is smooth, there is no distinction between sequence
$(\ref{sesWDiv})$ and sequence $(\ref{sesCDiv}).$ 

When $\Sigma$ is complete, as $\mathbf{C}^{\times}$ is
divisible, we get another short exact sequence by applying
$\Hom_{\mathbf{Z}}(-,\mathbf{C}^{\times})$ to sequence
$(\ref{sesWDiv})$:
\begin{equation}\label{sesG}
\begin{tikzcd}[sep= small]
1\arrow[r]& G\arrow[r]&(\mathbf{C}^{\times})^{\Sigma(1)}\arrow[r]&
T_{N}\arrow[r]& 1
\end{tikzcd}
\end{equation}
where
$G=\Hom_{\mathbf{Z}}(\text{Cl}(X_{\Sigma}),\mathbf{C}^{\times})$. Facts about the group $G$ are summarized in the
\begin{proposition}\label{lemma5.1.1}
Let $\Sigma$ be a complete fan and $G=\Hom_{\mathbf{Z}}(\Cl(X_{\Sigma}),\mathbf{C}^{\times})$. Then 
\begin{enumerate}
\item $G$ is diagonalizable.
\item $\Cl(X_{\Sigma})\simeq\pmb{\chi}(G)$.
\item An element $t\in(\mathbf{C}^{\times})^{\Sigma(1)}$ is in $G$ if and only
  if $$\prod_{\rho\in\Sigma(1)}t_{\rho}^{\langle m ,u_{\rho}\rangle}=1\text{ for
    all }m\in M.$$
\item There is an exact sequence  \begin{equation}\begin{tikzcd}[sep=small]\label{one_PS_SES}
0\arrow[r]&\pmb{\Gamma}(G)\arrow[r]&\mathbf{Z}^{\Sigma(1)}\arrow[r,"\varphi"]&N
\end{tikzcd}\end{equation} dual to (\ref{sesWDiv}) where if $\{e_{\rho}\}$ is the standard basis for $\mathbf{Z}^{\Sigma(1)}$, $\varphi$ is defined by $e_{\rho}\mapsto u_{\rho}$. In particular, $\pmb{\Gamma}(G)$ records the relation among ray generators of $\Sigma(1).$ 
\end{enumerate}
\end{proposition}

\begin{proof}
Statement (1) follows from the inclusion $G\subset(\mathbf{C}^{\times})^{\Sigma(1)}$. Statement (2) follows from the construction of $G$. For statement (3), see lemma 5.1.1 in \cite{MR2810322}.  For statement(4), the inclusion $G\hookrightarrow (\mathbf{C}^{\times})^{\Sigma(1)}$
induces an inclusion
$\pmb{\Gamma}(G)\hookrightarrow\Gamma((\mathbf{C}^{\times})^{\Sigma(1)})\simeq\mathbf{Z}^{\Sigma(1)}$. By
statement (3), a one parameter subgroup
$b\in\mathbf{Z}^{\Sigma(1)}$ is in $\pmb{\Gamma}(G)$ if and only if 
$$\prod_{\rho\in\Sigma(1)}t^{b_{\rho}\cdot\langle m ,u_{\rho}\rangle}=t^{\langle
  m,\sum_{\rho}b_{\rho}u_{\rho}\rangle}=1\text{ for all
}t\in\mathbf{C}^{\times}\text{ and for all }m\in M.$$ Hence $\langle
m,\sum_{\rho\in\Sigma(1)}b_{\rho}u_{\rho}\rangle =0 \text{ for all }m\in M.$ Since the
pairing between $M$ and $N$ is perfect, we have 
$$\sum_{\rho\in\Sigma(1)}b_{\rho}u_{\rho}=0.$$ Namely, $\pmb{\Gamma}(G)$ is
exactly the relation among ray generators. Moreover, the dual of
$M\rightarrow\mathbf{Z}^{\Sigma(1)}$ defined in (\ref{sesWDiv}) is
$\varphi$. Hence (\ref{one_PS_SES})  is the dual of (\ref{sesWDiv}). 
\end{proof}

\begin{notation}
We will also write the pairing (and its extension over $\mathbf{R}$) between $\pmb{\chi}(G)$ and $\pmb{\Gamma}(G)$ as $\langle-,-\rangle$ when the context is clear. Moreover, for a divisor $D\in\Cl(X_{\Sigma})$, we write $\chi_{D}$ as the corresponding character of $G$.
\end{notation}

\subsection{The ample cone}
We supply a numerical criterion to test if a divisor on a toric variety is ample. 
\begin{theorem}\label{numampleness}
Let $\Sigma$ be a fan and $D=\sum_{\rho}a_{\rho} D_{\rho}$ be a divisor on
$X_{\Sigma}$. Then $D$ is Cartier if and only if for all cone
$\sigma\in\Sigma_{\text{max}}$, there is an $m_{\sigma}\in M$ such
that $$\langle m,u_{\rho}\rangle = -a_{\rho}\text{ for all }
\rho\in\sigma(1).$$ Such $m_{\sigma}$ is unique modulo
$\sigma^{\perp}\cap M$. Moreover, if $\Sigma$ is complete, and
$D=\sum_{\rho}a_{\rho} D_{\rho}$ is Cartier with $m_{\sigma}\in M$ for
each $\sigma\in\Sigma_{\text{max}}$ as in the first part of the theorem, then $D$ is ample if and only if 
$$\langle m_{\sigma},u_{\rho}\rangle>-a_{\rho}\text{ for all
}\sigma\in\Sigma_{\text{max}} \text{ and for all }\rho\notin\sigma(1).$$
\end{theorem}

\begin{corollary}\label{Pic_free}
Suppose there is full dimensional cone in
$\Sigma$. Then $\Pic(X_{\Sigma})$ is free of finite rank. In particular,
this holds when $\Sigma$ is complete.
\end{corollary}

\begin{proof}
Because of the short exact sequence (\ref{sesCDiv}), we know that
$\Pic(X_{\Sigma})$ is finitely generated. Hence it is sufficient to
prove that $\Pic(X_{\Sigma})$ is torsion free.
Let $\sigma$ be a full dimensional cone in $\Sigma$.
Let $D=\sum_{\rho}a_{\rho}D_{\rho}$ be a torus invariant Cartier
divisor and $K\cdot D$ is linearly equivalent to $0$ for some
$K>0$. We need to prove that $D$ is linearly equivalent to $0$. By 
\Cref{numampleness}, there exists an $m_{\sigma}\in M$ such
that $\langle m_{\sigma},u_{\rho}\rangle=-a_{\rho}$ for all
$\rho\in\sigma(1)$. By the short exact sequence (\ref{sesCDiv}), there
exists an $m\in M$ such that $\langle m,u_{\rho}\rangle=-K\cdot
a_{\rho}$ for all $\rho\in\Sigma(1)$. Hence $\langle K\cdot
m_{\sigma},u_{\rho}\rangle=\langle m,u_{\rho}\rangle$ for all
$\rho\in\sigma(1)$. Since $\sigma$ is full dimensional, $K\cdot
m_{\sigma}=m$. As $\mathbf{Z}^{\Sigma(1)}$ is torsion free, we also
have $\text{div}(\chi^{ m_{\sigma}})=D$. Hence $D$ is zero in $\Pic(X_{\Sigma}).$
\end{proof}

We write $$\Amp(X_{\Sigma})$$ as the collection of ample
divisors on $X_{\Sigma}$. The \emph{ample cone }for $X_{\Sigma}$ is
the real cone over $\Amp(X_{\Sigma})$ in
$\Pic(X_{\Sigma})_{\mathbf{R}}$ and is denoted by $$\Amp(X_{\Sigma})_{\mathbf{R}}.$$

The ample cone of projective toric varieties are well understood. It
is the interior of the nef cone. We recall theorem 6.3.22 from \cite{MR2810322}:
\begin{theorem}\label{nef_and_ample_cone_thm}
Let $X_{\Sigma}$ be a projective toric variety. Then the cone generated by nef divisors, called the nef cone is a full dimensional, strongly convex rational polyhedral cone
in $\text{Pic}(X_{\Sigma})_{\mathbf{R}}$. Moreover, a Cartier divisor is
ample if and only if it is in the interior of the nef cone. 
\end{theorem}
\subsection{GIT quotient construction of projective toric varieties}\label{GIT quotient construction of projective toric varieties}
In this section we recall Cox's GIT quotient construction for projective toric varieties. The main result of this section is \Cref{bad_locus_ample} where we supply an alternative calculation of the unstable locus by cooking up a one parameter subgroup to destabilize points in the linear subspace defined by each primitive collection. The one parameter subgroups we used in \Cref{bad_locus_ample} are related to primitive relations on simplicial projective toric varieties (\Cref{primitive_relation}). The primitive relations were known to generate the Mori cone on a simplicial projective toric variety \cite{MR1133869}. In addition, the one parameter subgroups introduced in \Cref{bad_locus_ample} will allow us to extend the notion of stratifications induced by ample divisors to $\mathbf{R}$-ample divisors in the ample cone later at \Cref{Extensions_to_the_ample_cone}.

To begin with, for each fan $\Sigma$, we define the affine space $$\mathbf{C}^{\Sigma(1)}=\Spec\mathbf{C}[x_{\rho}\mid\rho\in\Sigma(1)].$$ 
For each cone $\sigma\in\Sigma_{\text{max}}$, define the
monomial $$x^{\hat{\sigma}}:=\prod_{\rho\notin\sigma(1)}x_{\rho}.$$
Then define the
ideal $$B(\Sigma)=(x^{\hat{\sigma}}|\sigma\in\Sigma_{\text{max}})\subset\mathbf{C}[x_{\rho}\mid\rho\in\Sigma(1)].$$ We
then have the
vanishing locus $$Z(\Sigma)=V(B(\Sigma))\subset\mathbf{C}^{\Sigma(1)}.$$ There is a cleaner
description of $B(\Sigma)$ and $Z(\Sigma)$, using the notion of \emph{primitive
  collections}. 
  \begin{definition}\label{primitive_collection_def}
  A subset $C\subset\Sigma(1)$ is a \emph{primitive
  collection} if:
\begin{enumerate}
\item $C\not\subset\sigma(1)$ for all $\sigma\in\Sigma$, and 
\item for any proper subset $C^{\prime}$ of $C$, there is a cone
  $\sigma\in\Sigma$ such that $C^{\prime}\subset\sigma(1)$.
\end{enumerate}
\end{definition}

\begin{proposition}\label{proposition5.1.6}
The ideal $B(\Sigma)$ has the primary decomposition 
$$B(\Sigma)=\bigcap_{C}(\{x_{\rho}|\rho\in C\})$$ which induces 
$$Z(\Sigma)=\bigcup_{C}V(\{x_{\rho}|\rho\in C\}).$$ Here both the union and
the intersection are taken over all primitive collections $C$ of $\Sigma(1)$.
\end{proposition}

\begin{proof}
See proposition 5.1.6 from \cite{MR2810322}.
\end{proof}

The closed subset $Z(\Sigma)$ is used for Cox's quotient construction for toric varieties. Below is Cox's theorem from \cite{MR1299003}. We refer the reader to chapter 5 in \cite{MR2810322} for the definition of an almost geometric quotient of a variety by an algebraic group. 
\begin{theorem}\label{thm5.1.11}
Let $X_{\Sigma}$ be a toric variety and $G=\Hom_{\mathbf{Z}}(\Cl(X_{\Sigma}),\mathbf{C}^{\times})$ act on $\mathbf{C}^{\Sigma(1)}$ via the inclusion $G\subset(\mathbf{C}^{\times})^{\Sigma(1)}$. Then there is a map
$$\pi:\mathbf{C}^{\Sigma(1)}\backslash Z(\Sigma)\rightarrow X_{\Sigma}$$ that
realizes $X_{\Sigma}$ as an almost geometric quotient of
$\mathbf{C}^{\Sigma(1)}\backslash Z(\Sigma)$ by $G$.
\end{theorem}

Since $G$ acts on $\mathbf{C}^{\Sigma(1)}$ and $\text{Cl}(X_{\Sigma})\simeq\pmb{\chi}(G)$, we can associate the GIT
quotient $\mathbf{C}^{\Sigma(1)}/\!\!/_{\chi_{D}}G$ with any divisor $D\in\text{Cl}(X_{\Sigma})$. It is natural to ask if there is a divisor $D\in\Cl(X_{\Sigma})$ such that $\mathbf{C}^{\Sigma(1)}/\!\!/_{\chi_{D}}G\simeq X_{\Sigma}$. This amounts to ask if there is a divisor $D$ such that $(\mathbf{C}^{\Sigma(1)})^{\text{ss}}(\chi_{D})=\mathbf{C}^{\Sigma(1)}\backslash Z(\Sigma)$. The answer is affirmative if $X_{\Sigma}$ is projective and if $D\in\Amp(X_{\Sigma})$. 
In \cite{MR2810322} Cox proved this by computing the $\chi_{D}$-invariants. We give an alternative proof using the numerical criterion for semistability (\Cref{HilbMumAffine}), and the numerical criterion for ampleness (\Cref{numampleness}). 

\begin{proposition}\label{bad_locus_ample}
Let $X_{\Sigma}$ be a projective toric variety and
$D\in\Cl(X_{\Sigma})$ be an ample divisor. Then $$(\mathbf{C}^{\Sigma(1)})^{\text{ss}}(\chi_{D})=\mathbf{C}^{\Sigma(1)}\backslash
Z(\Sigma).$$
\end{proposition}

\begin{proof}
We first show that a point $x\notin Z(\Sigma)$ is
$\chi_{D}$-semistable for any ample divisor
$D=\sum_{\rho}a_{\rho}D_{\rho}$. Let
$\lambda\in\mathbf{Z}^{\Sigma(1)}$ be 
a one parameter subgroup of $G$ such that
$\lim\limits_{t\rightarrow 0}\lambda(t)\cdot x$ exists. For each
$\rho\in\Sigma(1)$, we have $$(\lambda(t)\cdot
x)_{\rho}=t^{\lambda_{\rho}}\cdot x_{\rho}.$$
By definition of $Z(\Sigma)$, there is a
maximal cone $\sigma\in\Sigma_{\text{max}}$ such that $x^{\hat{\sigma}}(x)\neq
0$. Hence $\lim\limits_{t\rightarrow 0}\lambda(t)\cdot x$ exists only
if $\lambda_{\rho}\geq 0$ for all $\rho\notin\sigma(1)$. Recall we
have $m_{\sigma}\in M$ from \Cref{numampleness}. Now 
\begin{equation*}
\begin{split}
\langle\chi_{D},\lambda\rangle=&\sum_{\rho}a_{\rho}\lambda_{\rho}=\sum_{\rho\in\sigma(1)}a_{\rho}\lambda_{\rho}+\sum_{\rho\notin\sigma(1)}a_{\rho}\lambda_{\rho}\\
=&-\sum_{\rho\in\sigma(1)}\langle
m_{\sigma},u_{\rho}\rangle\lambda_{\rho}+\sum_{\rho\notin\sigma(1)}a_{\rho}\lambda_{\rho}\\
\geq&-\sum_{\rho\in\sigma(1)}\langle
m_{\sigma},u_{\rho}\rangle\lambda_{\rho}-\sum_{\rho\notin\sigma(1)}\langle
m_{\sigma},u_{\rho}\rangle\lambda_{\rho}\\
=&-\langle m_{\sigma},\sum_{\rho}\lambda_{\rho}u_{\rho}\rangle=0\end{split}
\end{equation*}
by \Cref{numampleness} and the fact that $\pmb{\Gamma}(G)$ is the
relation among ray generators.
Conversely, let $x\in Z(\Sigma)$, we need to establish a one
parameter subgroup $\lambda$ of $G$ such that 
\begin{itemize}
\item $\lim\limits_{t\rightarrow 0}\lambda(t)\cdot x$ exists, and 
\item $\langle\chi_{D},\lambda\rangle<0$.
\end{itemize}
For this, there is a primitive collection $C$ such that $x\in
V(\{x_{\rho}|\rho\in C\})$. Let $$u=\sum_{\rho\in C}u_{\rho}.$$ Then since $\Sigma$ is complete, there is a maximal cone $\sigma\in\Sigma_{\text{max}}$ such that $u\in\sigma$. By
proposition 11.1.7 in \cite{MR2810322}, there is a subcollection
$S\subset\sigma(1)$ such that $u\in\text{Cone}(S)$ and such that
$\text{Cone}(S)$ is a simplicial cone. Hence there are
$b_{\rho}\in\mathbf{Q}_{\geq 0}$ for $\rho\in\sigma(1)$ such that
$u=\sum_{\rho\in\sigma(1)}b_{\rho}u_{\rho}$. Let $K$ be a positive
integer large enough so that $K\cdot b_{\rho}\in\mathbf{N}$ for all
$\rho\in\sigma(1)$. Then we have a relation 
$$K\cdot u = \sum_{\rho\in\sigma(1)}(K\cdot b_{\rho})u_{\rho}.$$
Let 
\begin{equation}\label{non_simplicial_primitive_relation}c_{\rho} = \begin{cases}
-K & \text{ if }\rho\in C\backslash\sigma(1),\\
K\cdot (b_{\rho}-1) & \text{ if }\rho\in C\cap\sigma(1),\\
K\cdot b_{\rho} & \text{ if }\rho\in\sigma(1)\backslash C,\\
0 & \text{ otherwise}.
\end{cases}\end{equation}
Since each $c_{\rho}$ is an integer and $\sum_{\rho}c_{\rho}u_{\rho}=0$, we see that the collection
$\{c_{\rho}\}_{\rho\in\Sigma(1)}$ induces a one parameter subgroup $\lambda$ of
$G$. Furthermore $c_{\rho}\geq 0$ for all $\rho\notin C$. Hence $\displaystyle{\lim_{t\rightarrow 0}}\lambda(t)\cdot x$ exists. It remains to
show that
$\langle\chi_{D},\lambda\rangle=\sum_{\rho}a_{\rho}c_{\rho}<0$. Since
$C$ is a primitive collection, there is at least one $\rho^{\prime}\in
C\backslash\sigma(1)$ and $c_{\rho^{\prime}}=-K<0$. Therefore, we have 
\begin{equation*}
\begin{split}
\sum_{\rho}a_{\rho}c_{\rho}=&\sum_{\rho\in\sigma(1)}a_{\rho}c_{\rho}+\sum_{\rho\notin\sigma(1)}a_{\rho}c_{\rho}\\
<&-\sum_{\rho\in\sigma(1)}\langle
m_{\sigma},u_{\rho}\rangle c_{\rho}-\sum_{\rho\notin\sigma(1)}\langle
m_{\sigma},u_{\rho}\rangle c_{\rho}\\
=&-\langle m_{\sigma},\sum_{\rho}c_{\rho}u_{\rho}\rangle =0,
\end{split}
\end{equation*}
 where the $m_{\sigma}$ comes from \Cref{numampleness}.
\end{proof}

It is worthwhile pointing out that the one parameter subgroup we cooked up in \Cref{non_simplicial_primitive_relation} does not depend on the divisor $D$, but only on the primitive collection $C$. Namely, the one parameter subgroup defined by \Cref{non_simplicial_primitive_relation} fails the Hilbert-Mumford criterion for every point $x\in V(\{x_{\rho}\vert\rho\in C\})$ and for every ample divisor $D$. 

The one parameter subgroup introduced in \Cref{non_simplicial_primitive_relation} is related to primitive relations for simplicial complete fans. 
Let $\Sigma$ be a complete simplicial fan and $C\subset\Sigma(1)$ be a primitive collection. Then the sum $u=\sum_{\rho\in C}u_{\rho}$ is in the relative interior of a unique cone $\sigma\in\Sigma$. Since $\Sigma$ is simplicial, there exists unique $q_{\rho}\in\mathbf{Q}_{\geq 0}$ for each $\rho\in\sigma(1)$ such that $u=\sum_{\rho\in\sigma(1)}q_{\rho}u_{\rho}.$\begin{definition}
\label{primitive_relation}
The tuple $a\in\mathbf{Q}^{\Sigma(1)}$ where 
$$a_{\rho}=\begin{cases}
    1,\text{ if }\rho\in C\cap \backslash\sigma(1)\\
    1-b_{\rho},\text{ if }\rho\in C\cap\sigma(1)\\
    b_{\rho},\text{ if }\rho\in\sigma(1)\backslash C\\
    0, \text{ otherwise}.
\end{cases}$$ is called the \emph{primitive relation} of $C$.\end{definition}

\section{Toric VSIT}\label{Toric_VSIT_Section}

Let $X_{\Sigma}$ be a projective toric variety. The tasks in this section break mainly down to the following:
\begin{enumerate}
    \item Extend the notions of stability and stratification induced by an ample divisor to those induced by any $\mathbf{R}$-ample divisor in the ample cone and explain the strategy to compute stratifications induced by $\mathbf{R}$-ample divisors at \Cref{Extensions_to_the_ample_cone}.
    \item Identify certain collections of $\mathbf{R}$-ample divisors on $X_{\Sigma}$ that may cause the stratification to change as walls at \Cref{wall and semi-chamber}.
    \item Provide a sufficient condition for two $\mathbf{R}$-ample divisors to induce equivalent stratifications and prove that walls capture variations at \Cref{The main results}. They are formulated as \Cref{Toric_VSIT_decomposition_ample_cone}, \Cref{type_one_wall_captures_type_one_variation} and \Cref{type_two_variation_crosses_type_two_walls} respectively.
    \item Relate VSIT (\Cref{ToricVSIT}) and several topological properties of the strata to the combinatorial properties of the fan $\Sigma$ at \Cref{Relations to the structure of the fan}.
\end{enumerate}

We also supply an example at \Cref{Variation of stratification-an_elementary_example} where we compute stratifications induced by all $\mathbf{R}$-ample divisors. The picture is simple enough for us to summarize stratifications and their variations at \Cref{strat_summary_blow_up_P2}.

\subsection{Notations and conventions}\label{Notations and conventions_Toric_VSIT_}
Let $X_{\Sigma}$ be a projective toric variety and $G=\Hom_{\mathbf{Z}}(\Cl(X_{\Sigma}),\mathbf{C}^{\times})$. We equip $\pmb{\Gamma}(G)_{\mathbf{R}}$ with the inner product 
$$(-,-):\pmb{\Gamma}(G)_{\mathbf{R}}\times\pmb{\Gamma}(G)_{\mathbf{R}}\rightarrow\mathbf{R}$$ obtained by restricting the standard inner product on $\mathbf{R}^{\Sigma(1)}$ via the $\mathbf{R}$-extension of the inclusion $\pmb{\Gamma}(G)\hookrightarrow\mathbf{Z}^{\Sigma(1)}$ in sequence (\ref{one_PS_SES}). The induced norm is written as $$||-||:\pmb{\Gamma}(G)_{\mathbf{R}}\rightarrow\mathbf{R}.$$  

If $D$ is an element in $\Pic(X_{\Sigma})_{\mathbf{R}}$, we write $\chi_{D}^{\ast}$ as the unique vector in $\pmb{\Gamma}(G)_{\mathbf{R}}$ such that $$(\chi_{D}^{\ast},v)=\langle\chi_{D},v\rangle\text{ for all }v\in\pmb{\Gamma}(G)_{\mathbf{R}}.$$ 

We let $\mathcal{L}$ be the set of subsets of $\Sigma(1)$ where $S\in\mathcal{L}$ if and only if there is a primitive collection $C$ such that $S\subset \Sigma(1)-C$. With this we see that a point $x\in\mathbf{C}^{\Sigma(1)}$ is in $Z(\Sigma)$ if and only if $S_{x}\in\mathcal{L}$ where $S_{x}$ was defined as the state of $x$ in \Cref{Notations and conventions}. If $S\in\mathcal{L}$, we define \begin{enumerate}
\item the polyhedral cone $\sigma_{S}=\{v\in\pmb{\Gamma}(G)_{\mathbf{R}}\vert \langle\chi_{D_{\rho}},v\rangle\geq 0\text{ for all }\rho\in S\}$, and
\item the subspace $W_{S}= \{v\in\pmb{\Gamma}(G)_{\mathbf{R}}\vert \langle\chi_{D_{\rho}},v\rangle= 0\text{ for all }\rho\in S\}$.\end{enumerate}
Moreover, if $D\in\Amp(X_{\Sigma})_{\mathbf{R}}$, we define the finite subset of $\pmb{\Gamma}(G)_{\mathbf{R}}$ \begin{enumerate}[resume]
\item $\Lambda^{D}_{S}=\{-\Proj_{W_{Z}}\chi^{\ast}_{D}\bigm| -\Proj_{W_{Z}}\chi^{\ast}_{D}\in\sigma_{S}, Z\subset S\}.$
\end{enumerate}
If $x$ is a point in $Z(\Sigma)$ and $D\in\Amp(X_{\Sigma})_{\mathbf{R}}$, we let $$\Lambda^{D}_{x}:=\Lambda^{D}_{S_{x}}.$$ Obviously the set $\Lambda^{D}_{x}$ depends only on the states of $x$. We will see in \Cref{Kempf_ext_to_R_ample} that for every $S\in\mathcal{L}$, there is a unique longest vector $\lambda^{D}_{S}\in\Lambda^{D}_{S}$. Letting $\lambda^{D}_{x}= \lambda^{D}_{S_{x}}$, we define 
$$\Lambda^{D}:=\{\lambda^{D}_{S}\mid S\in\mathcal{L}\}=\{\lambda^{D}_{x}\mid x\in Z(\Sigma)\}.$$
We note without proof of the
\begin{fact}\label{fact_subspace_spanned_by_a_face}
Each face of $\sigma_{S}$ is of the form $W_{Z}\cap\sigma_{S}$ for some $Z\subset S$ and the subspace spanned by a face of $\sigma_{S}$ is of the form $W_{Z}$ for some $Z\subset S$.\end{fact}
\subsection{Extensions to the ample cone}\label{Extensions_to_the_ample_cone}
We need a dedicated analysis of the ample cone to understand variation of stratifications. The first step is to extend the notions of stability and stratifications induced by ample divisors to the entire ample cone. For stability, the extension adopts numerical conditions imposed by the Hilbert-Mumford criterion. 

\begin{definition}
Let $X_{\Sigma}$ be a projective toric variety and $D\in\Amp(X_{\Sigma})_{\mathbf{R}}$. We say a point $x\in \mathbf{C}^{\Sigma(1)}$ is $\chi_{D}$-\emph{semistable} if the function $\langle\chi_{D},-\rangle:\pmb{\Gamma}(G)_{\mathbf{R}}\rightarrow\mathbf{R}$ takes non-negative values at the cone $\sigma_{x}$. Else, we say $x$ is $\chi_{D}$-\emph{unstable}.\end{definition}

We now show that the $\chi_{D}$-semistable locus stays constant throughout the entire ample cone.
\begin{proposition}\label{Unstable_locus_ample_cone_extension}
Let $X_{\Sigma}$ be a projective toric variety, $D\in\Amp(X_{\Sigma})_{\mathbf{R}}$, and $x\in \mathbf{C}^{\Sigma(1)}$. Then the linear function $\langle\chi_{D},-\rangle:\pmb{\Gamma}(G)_{\mathbf{R}}\rightarrow\mathbf{R}$ assumes a negative value at the cone $\sigma_{x}$ if and only if $x\in Z(\Sigma)$. 
\end{proposition}

\begin{proof}
Suppose $x\in Z(\Sigma)$. Then $x\in V(\{x_{\rho}\mid\rho\in C\})$ for some primitive collection $C$. By \Cref{bad_locus_ample}, there is a one parameter subgroup $\lambda(C)\in\sigma_{y}$ such that $\langle\chi_{D^{\prime}},\lambda(C)\rangle<0$ for all $D^{\prime}\in\Amp(X_{\Sigma})$ and for all $y\in V(\{x_{\rho}\mid \rho\in C\})$. Writing $D=\sum_{i}r_{i}\cdot D_{i}$ for $D_{i}\in\Amp(X_{\Sigma})$ and $r_{i}>0$, we see that $\langle\chi_{D},\lambda(C)\rangle<0$. Since $\lambda(C)\in\sigma_{x}$, the if part is proved. 

Conversely, suppose $x\notin Z(\Sigma)$. Then by \Cref{bad_locus_ample}, $\langle\chi_{D^{\prime}},\lambda\rangle\geq 0$ for all $\lambda\in\sigma_{x}$ and for all $D^{\prime}\in\Amp(X_{\Sigma})$. Again writing $D=\sum_{i}r_{i}D_{i}$ for $r_{i}>0$ and for $D_{i}\in \Amp(X_{\Sigma})$, we get $\langle\chi_{D},\lambda\rangle\geq 0$ for all $\lambda\in\sigma_{x}$. The proposition is proved.
\end{proof}

Hence \begin{equation}\label{Rample_bad_locus}\begin{split}(\mathbf{C}^{\Sigma(1)})^{\text{ss}}(\chi_{D})&=\mathbf{C}^{\Sigma(1)}-Z(\Sigma)\\
&=\mathbf{C}^{\Sigma(1)}-\bigcup_{C}V(x_{\rho}\mid\rho\in C)\text{ for all }D\in\Amp(X_{\Sigma})_{\mathbf{R}}.\end{split}\end{equation}

We can also perform numerical analysis of instability as in \Cref{Numerical analysis of instability} on the entire ample cone. 
\begin{definition}\label{one_PS_adapted_to_extended}
Let $X_{\Sigma}$ be a projective toric variety and let $x\in Z(\Sigma)$. For any $D\in\Amp(X_{\Sigma})_{\mathbf{R}}$, we set $$M^{D}(x)=\inf_{\lambda\in\sigma_{x}\backslash\{0\}}\frac{\langle\chi_{D},\lambda\rangle}{||\lambda||}.$$ We say an element $\lambda\in\pmb{\Gamma}(G)_{\mathbf{R}}$ is $\chi_{D}$-\emph{adapted to} $x$ if $\frac{\langle\chi_{D},\lambda\rangle}{||\lambda||}=M^{D}(x)$. We say $\lambda$ is $\chi_{D}$-\emph{adapted to a subset} $Y\subset Z(\Sigma)$ if $\lambda$ is $\chi_{D}$-adapted to all points in $Y$.
\end{definition}

Kempf's theorem, \Cref{finitenesstheorem} can now be formulated in the ample cone. 

\begin{theorem}\label{Kempf_ext_to_R_ample}
Let $X_{\Sigma}$ be a projective toric variety. 
For any $x\in Z(\Sigma)$ and $D\in\Amp(X_{\Sigma})_{\mathbf{R}}$, there is a unique vector $\lambda^{D}_{x}\in\Lambda^{D}_{x}$ that is $\chi_{D}$-adapted to $x$. In this case, $$M^{D}(x)=-||\lambda^{D}_{x}||.$$
\end{theorem}

\begin{proof}
By \Cref{Unstable_locus_ample_cone_extension} and \Cref{unique_cone_min} and,  the function $$\frac{\langle\chi_{D},-\rangle}{||-||}:\pmb{\Gamma}(G)_{\mathbf{R}}\backslash\{0\}\rightarrow\mathbf{R}$$ achieves the relative minimum at a unique ray in $\sigma_{x}$. The existence and the uniqueness of $\lambda^{D}_{x}$ in $\Lambda^{D}_{x}$ is a direct consequence of \Cref{fact_subspace_spanned_by_a_face} and \Cref{compute_cone_min}. Finally, the equality $M^{D}(x)=-||\lambda^{D}_{x}||$ is derived by the same calculation done in \Cref{compute_cone_min}.
\end{proof}

Next, we formulate Hesselink's stratification of $\mathbf{C}^{\Sigma(1)}$ for any $D\in\Amp(X_{\Sigma})_{\mathbf{R}}$. For each $\lambda\in \Lambda^{D}$, we define the subset $$S^{\chi_{D}}_{\lambda}=\{x\in Z(\Sigma)\mid \lambda^{D}_{x}=\lambda\}.$$ By \Cref{lem_for_indexing_projections}, there is a unique $\lambda\in\Lambda^{D}$ that is $\chi_{D}$-adapted to $x$ for each point $x\in Z(\Sigma)$. Therefore, we have the decomposition $$\mathbf{C}^{\Sigma(1)}=(\mathbf{C}^{\Sigma(1)})^{\text{ss}}(\chi_{D})\cup\big(\bigcup_{\lambda\in\Lambda^{D}}S^{\chi_{D}}_{\lambda}\big).$$  

We define a strict partial order on $\Lambda^{D}$ by setting $\lambda>\lambda^{\prime}$ in $\Lambda^{D}$ if $||\lambda||>||\lambda^{\prime}||$. Using the same argument from \Cref{HesAffine}, one then shows that the above decomposition induces a stratification of $\mathbf{C}^{\Sigma(1)}$ and that it is equivalent to Hesselink's stratification whenever $D$ is an ample divisor.

We refer to the stratification from the above decomposition as \emph{the stratification of} $\mathbf{C}^{\Sigma(1)}$ \emph{induced by }$D\in\Amp(X_{\Sigma})_{\mathbf{R}}$. To compute the stratification induced by $D\in\Amp(X_{\Sigma})_{\mathbf{R}}$, we first note that the vector $\lambda^{D}_{x}$ only depends on the states $S_{x}$ of $x$ so $\lambda^{D}_{S}$ is $\chi_{D}$-adapted to $L(S)$ for each $S$. Hence computing stratification induced by $D$ is simply grouping $\{L(S)\}_{S\in\mathcal{L}}$ by $\{\lambda_{S}^{D}\}_{S\in\mathcal{L}}$. To compute $\lambda^{D}_{S}$, one simply computes the finite set $\Lambda^{D}_{S}$, and finds the longest vector in it. For $S$ whose size is small,  $\Lambda^{D}_{S}$ can easily be computed by hand. The complexity of enumerating $\Lambda^{D}_{S}$ grows exponentially with respect to the size of $S$  because it involves the power set of $S$. For larger $S$, we would rely on the computer program we wrote. See \Cref{Compute_stratification_with_respect_to_ample_divisors} for more information. 

We now define SIT-equivalence for $\mathbf{R}$-ample divisors.
\begin{definition}\label{R_ample_SIT_equivalence}
Let $X_{\Sigma}$ be a projective toric variety. We say two $\mathbf{R}$-ample divisors on $X_{\Sigma}$ are \emph{SIT-equivalent} if they induce equivalent stratifications of $\mathbf{C}^{\Sigma(1)}$.
\end{definition}

We end this section with a continuity result, \Cref{order_is_continuous} that is needed later. In short, we prove that for every point $x\in Z(\Sigma)$, choosing the longest vector $\lambda^{D}_{x}\in\Lambda^{D}_{x}$ is continuous with respect to $D$ in the ample cone. 
\begin{proposition}
Let $X_{\Sigma}$ be a projective toric variety. For any $x\in Z(\Sigma)$, the assignment 
    $\xi_{x}:\Amp(X_{\Sigma})_{\mathbf{R}}\rightarrow\pmb{\Gamma}(G)_{\mathbf{R}}$ defined by \begin{equation}\label{decision_making}
    D\mapsto \lambda^{D}_{x}
\end{equation}
is continuous.
\end{proposition}

\begin{proof}
Let $S_{x}\subset\Sigma(1)$ be the states of $x$. For each $Z\subset S_{x}$, define the subset of the ample cone 
$$\Amp(X_{\Sigma})_{Z}=\{D\in\Amp(X_{\Sigma})_{\mathbf{R}}\mid -\Proj_{W_{Z}}\chi^{\ast}_{D}=\lambda^{D}_{x}\}.$$ By \Cref{unique_cone_min}, the union $\cup_{Z\subset S_{x}}\Amp(X_{\Sigma})_{Z}$ is the entire ample cone. The assignment $\xi_{x}$ on $\Amp(X)_{Z}$ is given by $$D\mapsto -\Proj_{W_{Z}}\chi^{\ast}_{D},$$ which is obviously continuous. 
\end{proof}

\begin{corollary}\label{order_is_continuous}
Let $X_{\Sigma}$ be a projective toric variety. For any $x\in Z(\Sigma)$, the assignment $D\mapsto M^{D}(x)$ defined for $D\in\Amp(X_{\Sigma})_{\mathbf{R}}$ is a negative valued continuous function to the real line $\mathbf{R}$.
\end{corollary}
\begin{proof}
The assignment $D\mapsto M^{D}(x)$ is the continuous map $\xi_{x}$ post composed with the continuous map $(-1)\cdot||-||:\pmb{\Gamma}(G)_{\mathbf{R}}\rightarrow\mathbf{R}$.
\end{proof}

\subsection{Wall and semi-chamber}\label{wall and semi-chamber}
Let $X_{\Sigma}$ be a projective toric variety. \Cref{Rample_bad_locus} implies that the $\chi_{D}$-semistable locus stays constant throughout  $D\in\Amp(X_{\Sigma})_{\mathbf{R}}$. Hence the variation of stratifications of $\mathbf{C}^{\Sigma(1)}$ induced by $\mathbf{R}$-ample divisors  only occurs in the $\chi_{D}$-unstable locus $Z(\Sigma)$. In addition, we have seen in \Cref{Extensions_to_the_ample_cone} that each $\chi_{D}$-stratum is a grouping of $\{L(S)\}_{S\in\mathcal{L}}$ by the list of vectors $\{\lambda^{D}_{S}\}_{S\in\mathcal{L}}$. We therefore deduce that: 
\begin{enumerate}
    \item There is a type one variation (\Cref{two_types_of_variation_def}) between the stratifications of $\mathbf{C}^{\Sigma(1)}$ induced by $D,D^{\prime}\in\Amp(X_{\Sigma})_{\mathbf{R}}$ if and only if there is  a pair $S_{1},S_{2}\in\mathcal{L}$ such that $\lambda^{D}_{S_{1}}\neq\lambda^{D}_{S_{2}}$ while $\lambda^{D^{\prime}}_{S_{1}}=\lambda^{D^{\prime}}_{S_{2}}$.
    \item If there is a type two variation (\Cref{two_types_of_variation_def}) between the stratifications of $\mathbf{C}^{\Sigma(1)}$ induced by $D,D^{\prime}\in\Amp(X_{\Sigma})_{\mathbf{R}}$, then there is  a pair $S_{1},S_{2}\in\mathcal{L}$ such that $$||\lambda^{D}_{S_{1}}|| <||\lambda^{D}_{S_{2}}||\text{ while } ||\lambda^{D^{\prime}}_{S_{1}}||\geq ||\lambda^{D^{\prime}}_{S_{2}}||.$$
\end{enumerate}  

Therefore, the $\mathbf{R}$-ample divisors $D\in\Amp(X_{\Sigma})_{\mathbf{R}}$ such that the stratifications undergo type one (resp. type two) variations should be captured by the collections $$\{D\in\Amp(X_{\Sigma})_{\mathbf{R}}\mid \Proj_{W_{Z_{1}}}\chi^{\ast}_{D}=\Proj_{W_{Z_{2}}}\chi^{\ast}_{D}\}$$ $$(\text{resp. } \{D\in\Amp(X_{\Sigma})_{\mathbf{R}}\mid ||\Proj_{W_{Z_{1}}}\chi^{\ast}_{D}||=||\Proj_{W_{Z_{2}}}\chi^{\ast}_{D}||\})$$ for various $Z_{1},Z_{2}\in\mathcal{L}$. 

We now formulate two types of walls that intuitively capture two types of variations.
\begin{definition}\label{walls_def}
Suppose $X_{\Sigma}$ is a projective toric variety. Let 
$\emptyset\neq H\subsetneq\Amp(X_{\Sigma})_{\mathbf{R}}$ be a proper subset. We say $H$ is a \emph{type
one wall }if there are $Z$ and a singleton $\{\rho\}$ in $\mathcal{L}$
such that
\begin{enumerate}
\item $W_{Z\cup\{\rho\}}$ is a codimension one subspace of $W_{Z}$, and 
\item $H=\{D\in\Amp(X_{\Sigma})_{\mathbf{R}}\mid
  \Proj_{W_{Z}}\chi^{\ast}_{D}=\Proj_{W_{Z\cup\{\rho\}}}\chi^{\ast}_{D}\}.$ In this case, we say $H$ is a type one wall with respect to $Z$ and $\{\rho\}$.
\end{enumerate}
We say $H$ is a \emph{type two wall} if there are 
$Z_{1},Z_{2}\in\mathcal{L}$ such that
\begin{enumerate}
\item there is no containment between $W_{Z_{1}}$ and $W_{Z_{2}}$, and 
\item $H=\{D\in\Amp(X_{\Sigma})_{\mathbf{R}}\mid||\Proj_{W_{Z_{1}}}\chi^{\ast}_{D}||=||\Proj_{W_{Z_{2}}}\chi^{\ast}_{D}||\}$. In this case, we say $H$ is a type two wall with respect to $Z_{1}$ and $Z_{2}$.
\end{enumerate}
\end{definition}

\begin{remark}\label{remark_on_walls}
Here are some remarks of the definition of walls.  First, let $Z,\{\rho\}\in\mathcal{L}$. For the collection $\{D\in\Amp(X_{\Sigma})_{\mathbf{R}}\mid \Proj_{W_{Z}}\chi^{\ast}_{D}=\Proj_{W_{Z\cup\{\rho\}}}\chi^{\ast}_{D}\}$ to be proper it only makes sense for $W_{Z\cup\{\rho\}}$ to be a proper subspace of $W_{Z}$. Since $W_{Z\cup\{\rho\}}$ is the vanishing locus of the single linear functional $\langle\chi_{D_{\rho}},-\rangle:\pmb{\Gamma}(G)_{\mathbf{R}}\rightarrow\mathbf{R}$ on $W_{Z}$, it has to be of codimension one in $W_{Z}$. 

As for type two walls, the non-containment assumption required for a type two wall is due to the following observation: Suppose $W_{Z_{2}}\subset W_{Z_{1}}$. Then for any $D\in\Amp(X_{\Sigma})_{\mathbf{R}}$, $||\Proj_{W_{Z_{1}}}\chi^{\ast}_{D}||=||\Proj_{W_{Z_{2}}}\chi^{\ast}_{D}||$ is equivalent to $\Proj_{W_{Z_{1}}}\chi^{\ast}_{D}=\Proj_{W_{Z_{2}}}\chi^{\ast}_{D}$. We will see in \Cref{walls_contain_all_critical_points} below that a non-empty collection $$\{D\in\Amp(X_{\Sigma})_{\mathbf{R}}\mid \Proj_{W_{Z_{1}}}\chi^{\ast}_{D}=\Proj_{W_{Z_{2}}}\chi^{\ast}_{D}\}$$ is either the whole ample cone or an intersection of type one walls. Hence when there is a containment between $W_{Z_{1}}$ and $W_{Z_{2}}$, a non-empty proper subset of $\Amp(X)_{\mathbf{R}}$ of the form $\{D\in\Amp(X_{\Sigma})_{\mathbf{R}}\mid ||\Proj_{W_{Z_{1}}}\chi^{\ast}_{D}||=||\Proj_{W_{Z_{2}}}\chi^{\ast}_{D}||\}$ belongs to type one walls already.
\end{remark}

\begin{proposition}\label{walls_contain_all_critical_points}
Let $X_{\Sigma}$ be a projective toric variety. Let $Z_{1}$ and $Z_{2}$ be two subsets of
$\mathcal{L}$. Then a non-empty collection  
$$H=\{D\in\Amp(X_{\Sigma})_{\mathbf{R}}\mid\Proj_{W_{Z_{1}}}\chi_{D}^{\ast}=\Proj_{W_{Z_{2}}}\chi_{D}^{\ast}\}$$ is either the whole ample cone or an intersection of type one walls. On the other hand, a non-empty collection $$H^{\prime}=\{D\in\Amp(X_{\Sigma})_{\mathbf{R}}\mid ||\Proj_{W_{Z_{1}}}\chi_{D}^{\ast}||=||\Proj_{W_{Z_{2}}}\chi_{D}^{\ast}||\}$$
is either the whole ample cone, an intersection of type one walls, or a type two wall.
\end{proposition}

\begin{proof}
For the collection $H$, note that $\Proj_{W_{Z_{1}}}\chi^{\ast}_{D}=\Proj_{W_{Z_{2}}}\chi^{\ast}_{D}$ is equivalent to $$\Proj_{W_{Z_{1}}}\chi^{\ast}_{D}=\Proj_{W_{Z_{1}\cup Z_{2}}}\chi^{\ast}_{D}\text{ and }\Proj_{W_{Z_{1}\cup Z_{2}}}\chi^{\ast}_{D}=\Proj_{W_{Z_{2}}}\chi^{\ast}_{D}.$$ We may choose a sequence $l_{1},\ldots, l_{k}$ in $Z_{2}$ such that $$ W_{Z_{1}\cup l_{1}\cup\ldots\cup l_{k}}=W_{Z_{1}\cup Z_{2}}$$ and similarly $r_{1},\ldots r_{m}$ in $Z_{1}$ such that $$W_{Z_{2}\cup r_{1}\cup\ldots\cup r_{m}}=W_{Z_{1}\cup Z_{2}}.$$ We can now express $H$ as the intersection \begin{equation*}\begin{split}
H=&\big(\cap_{i}\{D\in\Amp(X_{\Sigma})_{\mathbf{R}}\mid \Proj_{W_{Z_{1}}}\chi^{\ast}_{D}=\Proj_{W_{Z_{1}\cup l_{i}}}\chi^{\ast}_{D}\}\big)\bigcap\\&\big(\cap_{j}\{D\in\Amp(X_{\Sigma})\mid \Proj_{W_{Z_{2}}}\chi^{\ast}_{D}=\Proj_{W_{Z_{2}\cup r_{j}}}\chi^{\ast}_{D}\}\big).\end{split}\end{equation*} 

If $H$ is a proper subset of $\Amp(X_{\Sigma})_{\mathbf{R}}$, then some subsets that occurred in the above intersection are proper and are type one walls. The statement for $H$ is proved. 

For $H^{\prime}$ with $H^{\prime}\neq\Amp(X_{\Sigma})_{\mathbf{R}}$, if there is a containment $W_{Z_{1}}\subset W_{Z_{2}}$, then $H^{\prime}$ is the same as $H$. This is an intersection of type one walls as we have seen. If there is no containment between $W_{Z_{1}}$ and $W_{Z_{2}}$, then $H^{\prime}$ is a type two wall by definition. 
\end{proof}

Therefore, walls should completely capture the divisors $D$ such that the stratification undergoes variations. We will rigorously prove that type one walls (resp. type two walls) capture type one (resp. type two) variations at \Cref{type_one_wall_captures_type_one_variation} and  \Cref{type_two_variation_crosses_type_two_walls}. We now describe the structure of walls. 
\begin{proposition}\label{struct_walls}
Let $X_{\Sigma}$ be a projective toric variety. 
\begin{enumerate}
    \item A type one wall is of codimension 1 in $\Amp(X_{\Sigma})_{\mathbf{R}}$. It is the intersection of a codimension 1 subspace of
$\Pic(X_{\Sigma})_{\mathbf{R}}$ with
$\Amp(X_{\Sigma})_{\mathbf{R}}$. 
\item A type two wall is of codimension at least 1 in $\Amp(X_{\Sigma})_{\mathbf{R}}$. It is either the intersection of a linear subspace of $\Pic(X_{\Sigma})_{\mathbf{R}}$ with $\Amp(X_{\Sigma})_{\mathbf{R}}$, or the intersection of a regular codimension 1 submanifold of
$\Pic(X_{\Sigma})_{\mathbf{R}}$ with $\Amp(X_{\Sigma})_{\mathbf{R}}$ away from a subspace of codimension at
least 2. 
\item If $X_{\Sigma}$ is simplicial, then a type two wall is of codimension 1 in $\Amp(X_{\Sigma})_{\mathbf{R}}$. It is the intersection of a regular codimension 1 submanifold of
$\Pic(X_{\Sigma})_{\mathbf{R}}$ with $\Amp(X_{\Sigma})_{\mathbf{R}}$ away from a subspace of codimension at least 2. 
\end{enumerate}
Moreover, walls are defined by equations that satisfy the following properties: 
\begin{itemize}
\item the condition that a point
  $\sum_{\rho}a_{\rho}D_{\rho}\in\Amp(X_{\Sigma})_{\mathbf{R}}$ is in
  a type one wall corresponds to a linear equation of
  $a_{\rho}$'s with rational coefficient,
\item the condition that a point
  $\sum_{\rho}a_{\rho}D_{\rho}\in\Amp(X_{\Sigma})_{\mathbf{R}}$ is in
  a type two wall corresponds to a homogeneous quadratic equation of
  $a_{\rho}$'s with rational coefficients.
\end{itemize}
\end{proposition}

\begin{proof}
For (1), let $H$ be the type one wall with respect to $Z$ and $\{\rho\}$. Namely,  $$H=\{D\in\Amp(X_{\Sigma})_{\mathbf{R}}\mid\Proj_{W_{Z}}\chi^{\ast}_{D}=\Proj_{W_{Z_\cup\{\rho\}}}\chi^{\ast}_{D}\}.$$ Define the linear map $\nu_{\rho}:\Pic(X_{\Sigma})_{\mathbf{R}}\rightarrow\mathbf{R}$ by the composition
$$\begin{tikzcd}
\Pic(X_{\Sigma})_{\mathbf{R}}\arrow[r,"\ast"]&\pmb{\Gamma}(G)_{\mathbf{R}}\arrow[r,"\Proj_{W_{Z}}(-)"]&[20]W_{Z}\arrow[r,"\langle\chi_{D_{\rho}}\text{,}-\rangle"]&[13]\mathbf{R}
\end{tikzcd}.$$
Namely, $$\nu_{\rho}(D)=\langle\chi_{D_{\rho}},\Proj_{W_{Z}}\chi^{\ast}_{D}\rangle.$$
We then have 
\begin{equation*}
\begin{split}\Proj_{W_{Z}}\chi^{\ast}_{D}&=\Proj_{W_{Z\cup\{\rho\}}}\chi^{\ast}_{D}\Leftrightarrow\Proj_{W_{Z}}\chi^{\ast}_{D}\in
W_{Z\cup\{\rho\}}\\
&\Leftrightarrow\langle\chi_{D_{\rho}},\Proj_{W_{Z}}\chi^{\ast}_{D}\rangle=\nu_{\rho}(D)=0.\end{split}\end{equation*} Hence  $$H=\ker\nu_{\rho}\cap\Amp(X_{\Sigma})_{\mathbf{R}}.$$ Since $H$ is a proper subset, 
$\ker\nu_{\rho}$ is a
codimension one subspace of $\Pic(X_{\Sigma})_{\mathbf{R}}.$ \\
\indent For (2), let $H^{\prime}$ be the type two wall $$\{D\in\Amp(X_{\Sigma})_{\mathbf{R}}\mid ||\Proj_{W_{Z_{1}}}\chi^{\ast}_{D}||=||\Proj_{W_{Z_{2}}}\chi_{D}^{\ast}||\}$$ with respect to $Z_{1}$ and $Z_{2}$. Notice that $H^{\prime}$ is defined by the vanishing locus $V(Q)$ of the quadratic form $Q$ associated to the bilinear form $B$ on $\Pic(X_{\Sigma})_{\mathbf{R}}$ defined by  \begin{equation}\label{quadratic_form}B(\chi_{D},\chi_{D^{\prime}})=(\Proj_{W_{Z_{1}}}\chi^{\ast}_{D},\Proj_{W_{Z_{1}}}\chi_{D^{\prime}}^{\ast})-(\Proj_{W_{Z_{2}}}\chi_{D}^{\ast},\Proj_{W_{Z_{2}}}\chi_{D^{\prime}}^{\ast}).\end{equation} By Sylvester's law of inertia, $Q$ looks like
$x_{1}^{2}+\cdots+x_{m}^{2}-y_{1}^{2}-\cdots y_{n}^{2}$ in a suitable coordinate of $\Pic(X_{\Sigma})_{\mathbf{R}}$. If one of $m$
or $n$ is zero, then $V(Q)$ is a linear subspace. If both $m$ and
$n$ are nonzero, then by \Cref{regularlevelset}, $V(Q)$ is a regular submanifold of codimension
1 away from the codimension at least 2 linear subspace $V(x_{1},\ldots,x_{m},y_{1},\ldots, y_{n})$.\\
\indent For (3), let $H^{\prime}$ be the type two wall with respect to $Z_{1},Z_{2}$ as in (2). When $X_{\Sigma}$ is simplicial, $\Pic(X_{\Sigma})_{\mathbf{Q}}\simeq\Cl(X_{\Sigma})_{\mathbf{Q}}$ so the map $\Pic(X_{\Sigma})_{\mathbf{R}}\rightarrow\pmb{\Gamma}(G)_{\mathbf{R}}$ defined by $D\mapsto \chi^{\ast}_{D}$ is an isomorphism. Hence the collection $$\{D\in\Pic(X_{\Sigma})_{\mathbf{R}}\mid ||\Proj_{W_{Z_{1}}}\chi^{\ast}_{D}||=||\Proj_{W_{Z_{2}}}\chi^{\ast}_{D}|| \}$$ is diffeomorphic to the collection $$\{v\in\pmb{\Gamma}(G)_{\mathbf{R}}\mid ||\Proj_{W_{Z_{1}}}v||=||\Proj_{W_{Z_{2}}}v||\},$$ which is the same as equi-distant  collection $$\{v\in\pmb{\Gamma}(G)_{\mathbf{R}}\mid \dist(v,W_{Z_{1}})=\dist(v,W_{Z_{2}})\}.$$ By \Cref{type_two_wall_is_codim_1}, $H^{\prime}$ is the intersection of a regular submanifold away from a subspace of codimension at least 2 with $\Amp(X_{\Sigma})_{\mathbf{R}}$.\\
\indent
We now deal with rationality of the equations of walls. For this, fix an integral $\mathbf{R}$-basis  $\{\lambda_{1},\ldots,\lambda_{q}\}\subset\pmb{\Gamma}(G)$ of $W_{Z}$ for each $Z\in\mathcal{L}$. Let $D=\sum_{\rho}a_{\rho}D_{\rho}$ be an element in the ample cone. Then there exists
$b_{1},\ldots,b_{q}$ such that
$\sum_{i}b_{i}\lambda_{j}=\Proj_{W_{Z}}\chi^{\ast}_{D}$. We will prove that each $b_{i}$ is a linear sum of $a_{\rho}$'s with rational coefficients. If this is proved, then if $D$ is in a type one wall, there are a subset $Z$ and a $\rho\notin Z$ such that $\langle\chi_{D_{\rho}}\Proj_{W_{Z}}\chi^{\ast}_{D}\rangle=0$ as we have seen. Since $\langle\chi_{D_{\rho}},\lambda_{i}\rangle\in\mathbf{Z}$, $\langle\chi_{D_{\rho}}\Proj_{W_{Z}}\chi^{\ast}_{D}\rangle=0$ is translated to a linear equation of $a_{\rho}'s$ with rational coefficients. Similarly, the condition that each $b_{i}$ is a linear sum of $a_{\rho}$'s with rational coefficients imply rationality of the equations of type two walls. 

For each $i=1,\ldots, q$, we have 
$$\langle\chi_{D},\lambda_{i}\rangle=(\chi^{\ast}_{D},\lambda_{i})=(\Proj_{W_{Z}}\chi_{D}^{\ast},\lambda_{i})=\sum_{j=1}^{q}b_{j}(\lambda_{i},\lambda_{j}).$$
Letting $A$
be the $q\times q$ invertible matrix with integral entries
$(\lambda_{i},\lambda_{j})$, we see that 
\begin{equation}\label{potential_one_PS_matrix_eq}
A\cdot\begin{pmatrix}
b_{1}\\
\vdots\\
b_{q}
\end{pmatrix} =
\begin{pmatrix}
\langle\chi_{D},\lambda_{1}\rangle\\
\vdots\\
\langle\chi_{D},\lambda_{q}\rangle
\end{pmatrix}
\end{equation}
Since each $\langle\chi_{D},\lambda_{i}\rangle$ is an integral combination of $a_{\rho}$'s,  each $b_{i}$ is a
$\mathbf{Q}$-combination of $a_{\rho}$'s. The proposition is proved.
\end{proof}

Since $\mathcal{L}$ is finite, we have finitely many walls. We now define semi-chambers that are open subsets in the complement of the union of walls. 

\begin{definition}\label{semi-chamber_def}
Let $X_{\Sigma}$ be a projective toric variety. Let $\{F_{i}\}$ be the
defining equations for walls in
$\Amp(X_{\Sigma})_{\mathbf{R}}$. A \emph{semi-chamber} is a non-empty
open semialgebraic set of the
form $$\{D\in\Amp(X_{\Sigma})_{\mathbf{R}}\mid\pm
F_{i}(D)>0\text{ for all }i\}.$$
\end{definition}

\begin{proposition}\label{finite_wall_semi-chamber}
Each wall and semi-chamber is a cone. The ample cone is a finite union of walls and semi-chambers. Moreover, semi-chambers are mutually exclusive.
\end{proposition}

\begin{proof}
That each wall and semi-chamber is a cone follows from the fact that walls are defined by the vanishing loci of homogeneous polynomials. The rest follows directly from the fact that there are finitely many walls and from the definition of semi-chambers.
\end{proof}
\subsection{The main results}\label{The main results}
Let $X_{\Sigma}$ be a projective toric variety. In this section we prove that if two $\mathbf{R}$-ample divisors are in the same semi-chamber, then they are SIT-equivalent at \Cref{Toric_VSIT_decomposition_ample_cone}. Namely, they induce equivalent stratifications of $\mathbf{C}^{\Sigma(1)}$. We also prove that type one (resp. type two) walls capture type one (resp. type two) variations at \Cref{type_one_wall_captures_type_one_variation} (resp. \Cref{type_two_variation_crosses_type_two_walls}).

\begin{lemma}\label{lemma_local_constant_WPS}
Let $X_{\Sigma}$ be a projective toric variety. Let $D_{a}$ and $D_{b}$ be in the same semi-chamber in $\Amp(X_{\Sigma})_{\mathbf{R}}$. Then the following three statements hold:
\begin{enumerate}
    \item For any $Z,Z^{\prime}\in\mathcal{L}$, $\Proj_{W_{Z}}\chi_{D_{a}}^{\ast}=\Proj_{W_{Z^{\prime}}}\chi_{D_{a}}^{\ast}$ implies $\Proj_{W_{Z}}\chi^{\ast}_{D}=\Proj_{W_{Z^{\prime}}}\chi^{\ast}_{D}$ for all $D\in\Amp(X_{\Sigma})_{\mathbf{R}}$. In particular, $$\Proj_{W_{Z}}\chi^{\ast}_{D_{a}}=\Proj_{W_{Z^{\prime}}}\chi^{\ast}_{D_{a}}\Leftrightarrow\Proj_{W_{Z^{\prime}}}\chi^{\ast}_{D_{b}}=\Proj_{W_{Z^{\prime}}}\chi^{\ast}_{D_{b}}.$$
\item For any $Z,Z^{\prime}\in\mathcal{L}$, $||\Proj_{W_{Z}}D_{a}^{\ast}||<||\Proj_{W_{Z^{\prime}}}D_{a}^{\ast}||\Leftrightarrow||\Proj_{W_{Z}}D_{b}^{\ast}||<||\Proj_{W_{Z^{\prime}}}D_{b}^{\ast}||$.  
    \item For any $Z\subset S\in\mathcal{L}$, $-\Proj_{W_{Z}}\chi_{D_{a}}^{\ast}\in\Lambda^{D_{a}}_{S}\Leftrightarrow
-\Proj_{W_{Z}}\chi_{D_{b}}^{\ast}\in\Lambda^{D_{b}}_{S}$.
\end{enumerate}
\end{lemma}

\begin{proof}
For (1), suppose $\Proj_{W_{Z}}\chi_{D_{a}}^{\ast}=\Proj_{W_{Z^{\prime}}}\chi_{D_{a}}^{\ast}$. Then the collection $H=\{D\in \Amp(X_{\Sigma})_{\mathbf{R}}\mid \Proj_{W_{Z}}\chi^{\ast}_{D}=\Proj_{W_{Z^{\prime}}}\chi^{\ast}_{D}\}$ contains $D_{a}$ so is non-empty. \Cref{walls_contain_all_critical_points} dictates that $H$ is 
either $\Amp(X_{\Sigma})_{\mathbf{R}}$ or an intersection
of type one walls. Since $D_{a}$ is in a semi-chamber, $H=\Amp(X_{\Sigma})_{\mathbf{R}}$. Hence $\Proj_{W_{Z}}\chi_{D}^{\ast}=\Proj_{W_{Z^{\prime}}}\chi_{D}^{\ast}$ for all $D\in\Amp(X_{\Sigma})_{\mathbf{R}}$. Statement (1) is proved. 

For (2), suppose $||\Proj_{W_{Z}}\chi_{D_{a}}||<||\Proj_{W_{Z^{\prime}}}\chi_{D_{a}}||.$
We consider two cases where in one case there is a containment between $W_{Z}$
and $W_{Z^{\prime}}$, and no containment in the other. If
there is a containment between $W_{Z}$ and $W_{Z^{\prime}}$, since $||\Proj_{W_{Z}}\chi_{D_{a}}^{\ast}||<||\Proj_{W_{Z^{\prime}}}\chi_{D_{a}}^{\ast}||,$ we have $W_{Z}\subsetneq W_{Z^{\prime}}$. Hence
$$||\Proj_{W_{Z}}\chi_{D_{b}}^{\ast}||\leq
||\Proj_{W_{Z^{\prime}}}\chi_{D_{b}}^{\ast}||.$$ If
$||\Proj_{W_{Z}}\chi_{D_{b}}^{\ast}||=||\Proj_{W_{Z^{\prime}}}\chi_{D_{b}}^{\ast}||$,
then we would have
$$\Proj_{W_{Z}}\chi_{D_{b}}^{\ast}=\Proj_{W_{Z^{\prime}}}\chi_{D_{b}}^{\ast}.$$ 
However, statement (1) implies $\Proj_{W_{Z}}\chi_{D_{a}}^{\ast}=\Proj_{W_{Z^{\prime}}}\chi_{D_{a}}^{\ast}$, a contradiction. Hence
$||\Proj_{W_{Z}}\chi_{D_{b}}^{\ast}||<||\Proj_{W_{Z^{\prime}}}\chi_{D_{b}}||$ as desired. Suppose there is no containment between $W_{Z}$ and
$W_{Z^{\prime}}$. The collection 
$H^{\prime}=\{D\in\Amp(X_{\Sigma})_{\mathbf{R}}\mid ||\Proj_{W_{Z}}\chi_{D}^{\ast}||=||\Proj_{W_{Z^{\prime}}}\chi_{D}^{\ast}||\}$
is either empty or a type
two wall. If $H^{\prime}=\emptyset$, then by continuity of the quadratic $Q$ form (induced by a bilinear form defined similarly as in \Cref{quadratic_form}) and convexity of $\Amp(X_{\Sigma})_{\mathbf{R}}$, $$||\Proj_{W_{Z}}\chi^{\ast}_{D}||<||\Proj_{W_{Z^{\prime}}}\chi^{\ast}_{D}||\text{ for all }D\in\Amp(X_{\Sigma})_{\mathbf{R}}.$$ If $H^{\prime}\neq\emptyset$, it is a type two wall. As $D_{a},D_{b}$ are in the same semi-chamber, $Q(D_{a})$ and $Q(D_{b})$ have the same sign.  Hence $||\Proj_{W_{Z}}\chi_{D_{b}}^{\ast}||<||\Proj_{W_{Z^{\prime}}}\chi_{D_{b}}^{\ast}||$ as well. The argument can be reversed so statement (2) holds.

For (3), let $\rho\in S$ and 
$\nu_{\rho}:\Pic(X_{\Sigma})_{\mathbf{R}}\rightarrow\mathbf{R}$ be the linear map defined by
$D\mapsto
\langle\chi_{D_{\rho}},\Proj_{W_{Z}}\chi_{D}^{\ast}\rangle$. We will prove $$\nu_{\rho}(D_{a})< 0\Leftrightarrow\nu(D_{b})<0.$$ Suppose $\nu_{\rho}(D_{a})<0$. The set $H^{\prime\prime}:=\ker\nu_{\rho}\cap\Amp(X_{\Sigma})_{\mathbf{R}}$ is either empty, or a type one wall with respect to the subsets $Z$ and $Z\cup\{\rho\}$. If $H^{\prime\prime}=\emptyset$, then by continuity of $\nu_{\rho}$ and convexity of $\Amp(X_{\Sigma})_{\mathbf{R}}$, $\nu_{\rho}(D)<0$ for all $D\in\Amp(X_{\Sigma})_{\mathbf{R}}$. 
If $H^{\prime\prime}$ is a type one wall, then as $D_{a},D_{b}$ are in the same semi-chamber, $\nu_{\rho}(D_{b})<0$ as well. The argument can be reversed. Since we did this for every $\rho\in S$, we get 
$-\Proj_{W_{Z}}\chi_{D_{a}}^{\ast}\in\sigma_{S}$ if and only
if $-\Proj_{W_{Z}}\chi_{D_{b}}^{\ast}\in\sigma_{S}$. The validity of statement (3) is established and the lemma is proved.
\end{proof}

\begin{lemma}\label{local_constant_WPS}
Let $X_{\Sigma}$ be a projective toric variety and let $G=\Hom_{\mathbf{Z}}(\Cl(X_{\Sigma}),\mathbf{C}^{\times})$. Order the vectors in $\pmb{\Gamma}(G)_{\mathbf{R}}$ by their norms. If $D_{a}$ and $D_{b}$ are in the same semi-chamber in $\Amp(X_{\Sigma})_{\mathbf{R}}$, then there is an order preserving
bijection $$\Xi:\bigcup_{S\in\mathcal{L}}\Lambda^{D_{a}}_{S}\rightarrow\bigcup_{S\in\mathcal{L}}\Lambda^{D_{b}}_{S}$$
defined
by $$-\Proj_{W_{Z}}\chi_{D_{a}}^{\ast}\mapsto-\Proj_{W_{Z}}\chi_{D_{b}}^{\ast}$$ 
for all subsets $Z\subset S$ such that $-\Proj_{W_{Z}}\chi_{D_{a}}^{\ast}\in\Lambda^{D_{a}}_{S}$ and for all $S\in\mathcal{L}$. Moreover, for any subcollection $\{S_{1},\ldots S_{l}\}\subset\mathcal{L}$, $\Xi$ restricts to an order preserving bijection between $\cup_{i=1}^{l}\Lambda^{D_{a}}_{S_{i}}$ and $\cup_{i=1}^{l}\Lambda^{D_{b}}_{S_{i}}$.
\end{lemma}

\begin{proof}
Statement (1) in \Cref{lemma_local_constant_WPS} implies the map $\Xi$ is well defined and injective. Statement (3) in \Cref{lemma_local_constant_WPS} implies that the image of $\Xi$ lands inside $\cup_{S\in\mathcal{L}}\Lambda^{D_{b}}_{S}$ and that $\Xi$ is surjective. Statement (2) in \Cref{lemma_local_constant_WPS} implies that $\Xi$ is order preserving. The same logic works for any subcollection
 $\{S_{1},\ldots,S_{l}\}\subset\mathcal{L}$. \end{proof}

\begin{proposition}\label{last_before_AffineVIIT}
Let $X_{\Sigma}$ be a projective toric variety and suppose $D_{a},D_{b}$ are in the same semi-chamber in $\Amp(X_{\Sigma})_{\mathbf{R}}$. Let $S\in\mathcal{L}$ and $Z$ be a subset of $S$. Then 
    $$\lambda^{D_{a}}_{S}=-\Proj_{W_{Z}}\chi_{D_{a}}^{\ast}\Leftrightarrow\lambda^{D_{b}}_{S}=-\Proj_{W_{Z}}\chi_{D_{b}}^{\ast}.$$
\end{proposition}

\begin{proof}
This is a direct result of \Cref{local_constant_WPS}, applied to the single subset $S$.
\end{proof}
We are now ready to prove
\begin{theorem}\label{Toric_VSIT_decomposition_ample_cone}
Let $X_{\Sigma}$ be a projective toric variety. If two $\mathbf{R}$-ample divisors on $X_{\Sigma}$ are in the same semi-chamber, then they are SIT-equivalent.
\end{theorem}

\begin{proof}
Suppose $D_{a}$ and $D_{b}$ are two $\mathbf{R}$-ample divisors on $X_{\Sigma}$ and they are in the same semi-chamber. 
Let $S_{1},S_{2}$ be two elements in $\mathcal{L}$. We first prove that $L(S_{1})$ and $L(S_{2})$ are in the same $\chi_{D_{a}}$-stratum if and only if they are in the same $\chi_{D_{b}}$-stratum. This would give rise to a bijection between the set of $\chi_{D_{a}}$-strata and the set of $\chi_{D_{b}}$-strata that preserves each stratum as sets. For this, pick subsets $Z_{1},Z_{2}$ of $S_{1}$ and $S_{2}$ respectively so that $\lambda^{D_{a}}_{S_{1}}=-\Proj_{W_{Z_{1}}}\chi^{\ast}_{D_{a}}$ and $\lambda^{D_{a}}_{S_{2}}=-\Proj_{W_{Z_{2}}}\chi^{\ast}_{D_{a}}$. By \Cref{last_before_AffineVIIT}, we have $\lambda^{D_{b}}_{S_{1}}=-\Proj_{W_{Z_{1}}}\chi^{\ast}_{D_{b}}$ and  $\lambda^{D_{b}}_{S_{2}}=-\Proj_{W_{Z_{2}}}\chi^{\ast}_{D_{b}}$. Suppose $x,y$ are in the same $\chi_{D_{a}}$-stratum. We then have $\lambda^{D_{a}}_{S_{1}}=\lambda^{D_{a}}_{S_{2}}$. By \Cref{local_constant_WPS}, $\lambda^{D_{b}}_{S_{1}}=\lambda^{D_{b}}_{S_{2}}$. Hence $x,y$ are in the same $\chi_{D_{b}}$-stratum. The argument can be reversed. Hence we obtain a bijection between the set of $\chi_{D_{a}}$-strata and the set of $\chi_{D_{b}}$-strata that preserves strata as sets.

Finally, the bijection preserves the strict partial order by \Cref{local_constant_WPS}, applied to the entire collection $\mathcal{L}$.
\end{proof}

The rest of the section is to prove that each type of walls capture a type of variation. We consider line segments connecting pairs of $\mathbf{R}$-ample divisors that induce different stratifications. Since we do not know if a semi-chamber is convex or not, \Cref{local_constant_for_a_segment} below cannot be applied to establish the bijection part of \Cref{local_constant_WPS}. 
\begin{lemma}\label{local_constant_for_a_segment}
Let $X_{\Sigma}$ be a projective toric variety and $D,D^{\prime}$ be two elements in $\Amp(X_{\Sigma})_{\mathbf{R}}$ such that the line segment $\overline{DD^{\prime}}$ does not intersect any type one walls properly. Then for any $D^{\prime\prime}\in\overline{DD^{\prime}}$ and any $S\in\mathcal{L}$, the assignment $$\Psi_{S}:\Lambda^{D}_{S}\rightarrow\Lambda^{D^{\prime\prime}}_{S}$$ defined by  $-\Proj_{W_{Z}}\chi^{\ast}_{D}\mapsto-\Proj_{W_{Z}}\chi^{\ast}_{D^{\prime\prime}}$ is a bijection.
\end{lemma}

\begin{proof}
It comes down to show the following two statements:
\begin{enumerate}
    \item For any $Z\subset S$, $-\Proj_{W_{Z}}\chi^{\ast}_{D}\in\Lambda^{D}_{S}$ if and only if $-\Proj_{W_{Z}}\chi^{\ast}_{D^{\prime\prime}}\in\Lambda^{D^{\prime\prime}}_{S}$.
    \item For any $Z_{1},Z_{2}\subset S$, $-\Proj_{W_{Z_{1}}}\chi^{\ast}_{D}=-\Proj_{W_{Z_{2}}}\chi^{\ast}_{D}$ if and only if\\ $-\Proj_{W_{Z_{1}}}\chi^{\ast}_{D^{\prime\prime}}=-\Proj_{W_{Z_{2}}}\chi^{\ast}_{D^{\prime\prime}}$.
\end{enumerate}
Note that statement (1) implies the image of $\Psi_{S}$ lands in $\Lambda^{D^{\prime\prime}}_{S}$ and that $\Psi_{S}$ is surjective. Statement (2) implies $\Psi_{S}$ is well-defined and injective. The proofs are entirely analogous to \Cref{lemma_local_constant_WPS}. To be careful we supply the proofs here.

For statement (1), let $Z\subset S$. Define $\nu_{\rho}:\Pic(X_{\Sigma})_{\mathbf{R}}\rightarrow\mathbf{R}$ for every $\rho\in S$ by $$\nu_{\rho}(\tilde{D})=\langle \chi_{D_{\rho}},\Proj_{W_{Z}}\chi^{\ast}_{\tilde{D}}\rangle$$ where $D_{\rho}\in\Cl(X_{\Sigma})$ is the torus invariant divisor corresponding to $\rho$. We need to show that $$\nu_{\rho}(D)<0\Leftrightarrow\nu_{\rho}(D^{\prime\prime})<0\text{ for all }\rho\in S.$$ 

For this, suppose $\nu_{\rho}(D)<0$. The set $H=\ker\nu_{\rho}\cap\Amp(X_{\Sigma})_{\mathbf{R}}$ is then either empty or a type one wall. Hence  If $H$ is empty then by continuity of $\nu_{\rho}$ and convexity of $\Amp(X_{\Sigma})_{\mathbf{R}}$, $\nu_{\rho}(D^{\prime\prime\prime})<0$ for any $D^{\prime\prime\prime}\in\Amp(X_{\Sigma})_{\mathbf{R}}$. In particular, $\nu_{\rho}(D^{\prime\prime})<0$. If $H$ is a type one wall but $\nu_{\rho}(D^{\prime\prime})\geq 0$, then by continuity of $\nu_{\rho}$ on the line segment $\overline{DD^{\prime\prime}}$, the type one wall $H$ must intersect $\overline{DD^{\prime\prime}}$ properly, contradicting the assumption that $\overline{DD^{\prime}}$ does not intersect any type one wall properly. The argument can be reversed to show the converse. Since we can do this for any $Z\subset S$, statement (1) is proved. 

For statement (2), assume $\Proj_{W_{Z_{1}}}\chi^{\ast}_{D}=\Proj_{W_{Z_{2}}}\chi^{\ast}_{D}$ for some $Z_{1},Z_{2}\subset S$. Then the set $$H^{\prime}=\{\tilde{D}\in\Amp(X_{\Sigma})_{\mathbf{R}}\mid \Proj_{W_{Z_{1}}}\chi^{\ast}_{\tilde{D}}=\Proj_{W_{Z_{2}}}\chi^{\ast}_{\tilde{D}}\}$$ is either the entire $\Amp(X_{\Sigma})_{\mathbf{R}}$ or an intersection of type one walls by \Cref{walls_contain_all_critical_points}. Hence $H^{\prime}$ either contains the entire line segment $\overline{DD^{\prime}}$ or intersects it at the point $D$. Since $\overline{DD^{\prime}}$ does not intersect any type one wall properly, it can only be the case that $\overline{DD^{\prime}}\subset H^{\prime}$, in which case  $\Proj_{W_{Z_{1}}}\chi^{\ast}_{D^{\prime}}=\Proj_{W_{Z_{2}}}\chi^{\ast}_{D^{\prime}}$ is automatic. The argument can be reversed to show the converse. Statement (2) and therefore the lemma is proved.
\end{proof}

\begin{lemma}\label{decision_making_involves_type_one_walls}
Let $X_{\Sigma}$ be a projective toric variety, $D,D^{\prime}$ be two elements in $\Amp(X_{\Sigma})_{\mathbf{R}}$ and $S\in\mathcal{L}$. Suppose there are two subsets $Z_{1},Z_{2}\subset S$ such that the following conditions hold:
\begin{enumerate}
    \item At $D$, $\lambda^{D}_{S}=-\Proj_{W_{Z_{1}}}\chi^{\ast}_{D}\neq-\Proj_{W_{Z_{2}}}\chi^{\ast}_{D}$.
    \item At $D^{\prime}$, $\lambda^{D^{\prime}}_{S}=-\Proj_{W_{Z_{2}}}\chi^{\ast}_{D^{\prime}}\neq-\Proj_{W_{Z_{1}}}\chi^{\ast}_{D^{\prime}}$.
\end{enumerate}
Then the line segment $\overline{DD^{\prime}}$ intersects a type one wall properly.
\end{lemma}

\begin{proof}
Suppose on the contrary that $\overline{DD^{\prime}}$ does not intersect any type one wall properly. Then by \Cref{local_constant_for_a_segment}, both $-\Proj_{W_{Z_{1}}}\chi^{\ast}_{D^{\prime\prime}}$ and $-\Proj_{W_{Z_{2}}}\chi^{\ast}_{D^{\prime\prime}}$ are elements of $\Lambda^{D^{\prime\prime}}_{S}$ for all $D^{\prime\prime}\in\overline{DD^{\prime}}.$ This together with the assumptions of the lemma imply that there is no containment between $W_{Z_{1}}$ and $W_{Z_{2}}$. To see this, if say $W_{Z_{1}}\subset W_{Z_{2}}$, then $||\Proj_{W_{Z_{1}}}\chi^{\ast}_{D}||\leq ||\Proj_{W_{Z_{2}}}\chi^{\ast}_{D}||$, contradicting the assumption that $\lambda^{D}_{S}=-\Proj_{W_{Z_{1}}}\chi^{\ast}_{D}\neq-\Proj_{W_{Z_{2}}}\chi^{\ast}_{D}$. Similarly we cannot have $W_{Z_{2}}\subset W_{Z_{1}}$ for the assumptions at $D^{\prime}$.

Let $-\Proj_{W_{Z_{1}}}\chi^{\ast}_{D},\ldots,-\Proj_{W_{Z_{N}}}\chi^{\ast}_{D}$ be $N$ distinct elements in $\Lambda^{D}_{S}$ where there is no containment between $W_{Z_{i}}$ and $W_{Z_{j}}$ for any $i\neq j$. We can consider the $N-1$ conditions \begin{equation}\label{decision_making_type_two_walls}||\Proj_{W_{Z_{1}}}\chi^{\ast}_{\tilde{D}}||=||\Proj_{W_{Z_{i}}}\chi^{\ast}_{\tilde{D}}||\text{ for }\tilde{D}\in\Amp(X_{\Sigma})_{\mathbf{R}}\text{ and for }i=2.\ldots, N.\end{equation} Since $\lambda^{D}_{S}=-\Proj_{W_{Z_{1}}}\chi^{\ast}_{D}$, we have \begin{equation}\label{assumption_on_D}||\Proj_{W_{Z_{1}}}\chi^{\ast}_{D}||>||\Proj_{W_{Z_{i}}}\chi^{\ast}_{D}||\text{ for all }i=2,\ldots, N.\end{equation} Hence each of the $N-1$ conditions defines either the empty set or a type two wall and is not met by $D$. In particular, none of these conditions includes the whole line segment $\overline{DD^{\prime}}$.

By continuity of the quadratic form associated to the bilinear form defined at (\ref{quadratic_form}), the condition for $i=2$ in (\ref{decision_making_type_two_walls}) defines a type two wall that intersects the line segment$\overline{DD^{\prime}}$ properly. Hence the set $A$ of points where the type two walls from (\ref{decision_making_type_two_walls}) intersect $\overline{DD^{\prime}}$ properly is not empty. Note that $D\notin A$.

Since a type two wall intersects a line properly at at most two points, the set $A$ is finite. We can therefore pick the point $D_{a}\in A$ that is closest to $D$. By \Cref{local_constant_for_a_segment} and the construction of $a$, it must be the case that $\lambda^{D^{\prime\prime}}_{S}=-\Proj_{W_{Z_{1}}}\chi^{\ast}_{D^{\prime\prime}}$ for all $D^{\prime\prime}$ in the half open interval $\overline{DD_{a}}\backslash D_{a}$. We claim that $$\lambda^{D_{a}}_{S}=-\Proj_{W_{Z_{1}}}\chi^{\ast}_{D_{a}}.$$
If this is not the case then there is a $j\neq 1$ such that $$\lambda^{D_{a}}_{S}=-\Proj_{W_{Z_{j}}}\chi^{\ast}_{D_{a}}\neq-\Proj_{W_{Z_{1}}}\chi^{\ast}_{D_{a}}.$$ Hence $||\Proj_{W_{Z_{j}}}\chi^{\ast}_{D_{a}}||>||\Proj_{W_{Z_{1}}}\chi^{\ast}_{D_{a}}||$. By continuity, there will be a point $D_{b}\in\overline{DD_{a}}\backslash D_{a}$ such that $||\Proj_{W_{Z_{j}}}\chi^{\ast}_{D_{b}}||=||\Proj_{W_{Z_{1}}}\chi^{\ast}_{D_{b}}||,$ contradicting the construction of $a$. The claim is proved. 

However, if the point $a$ is the intersection of $\overline{DD^{\prime}}$ with the type two wall defined for $k$ in (\ref{decision_making_type_two_walls}), then the uniqueness imposed by \Cref{unique_cone_min} forces  $$\Proj_{W_{Z_{1}}}\chi^{\ast}_{D_{a}}=\Proj_{W_{Z_{k}}}\chi^{\ast}_{D_{a}}.$$ This implies $D_{a}$ is in the intersection of type one walls defined by the condition $$\Proj_{W_{Z_{1}}}\chi^{\ast}_{\tilde{D}}=\Proj_{W_{Z_{k}}}\chi^{\ast}_{\tilde{D}}\text{ for }\tilde{D}\in\Amp(X_{\Sigma})_{\mathbf{R}},$$ contradicting our assumption that $\overline{DD^{\prime}}$ intersects no type one walls properly. The lemma is proved.
\end{proof}

\begin{theorem}\label{type_one_wall_captures_type_one_variation}
Let $X_{\Sigma}$ be a projective toric variety. If $D,D^{\prime}$ are two elements in $\Amp(X_{\Sigma})_{\mathbf{R}}$ such that the stratifications induced by $D$ and $D^{\prime}$ undergo a type one variation, then the line segment $\overline{DD^{\prime}}$ intersects a type one wall properly.
\end{theorem}

\begin{proof}
Without loss of generality, we may pick a pair $S_{1},S_{2}$ of elements in $\mathcal{L}$ such that $L(S_{1})$ and $L(S_{2})$ are in different  $\chi_{D}$-strata but in the same $\chi_{D^{\prime}}$-stratum. Let $Z_{1},Z_{2}$ (resp. $Z_{1}^{\prime},Z^{\prime}_{2}$) be subsets of $S_{1}$ and $S_{2}$ so that $\lambda^{D}_{S_{1}}=-\Proj_{W_{Z_{1}}}\chi^{\ast}_{D}$, $\lambda^{D}_{S_{2}}=-\Proj_{W_{Z_{2}}}\chi^{\ast}_{D}$ (resp. $\lambda^{D^{\prime}}_{S_{1}}=-\Proj_{W_{Z_{1}^{\prime}}}\chi^{\ast}_{D^{\prime}}$, $\lambda^{D^{\prime}}_{S_{2}}=-\Proj_{W_{Z^{\prime}_{2}}}\chi^{\ast}_{D^{\prime}}.$) 

Suppose on the contrary that the line segment $\overline{DD^{\prime}}$ intersects no type one walls properly. By \Cref{local_constant_for_a_segment}, we then have 
\begin{enumerate}
    \item $\Proj_{W_{Z_{1}}}\chi^{\ast}_{D^{\prime\prime}}\neq \Proj_{W_{Z_{2}}}\chi^{\ast}_{D^{\prime\prime}}$ for all $D^{\prime\prime}\in\overline{DD^{\prime}}$, and 
    \item $\Proj_{W_{Z_{1}^{\prime}}}\chi^{\ast}_{D^{\prime\prime}}= \Proj_{W_{Z_{2}}^{\prime}}\chi^{\ast}_{D^{\prime\prime}}$ for all $D^{\prime\prime}\in\overline{DD^{\prime}}$\label{item2}. 
\end{enumerate}
We claim that $\Proj_{W_{Z_{i}}}\chi^{\ast}_{D}\neq\Proj_{W_{Z_{i}^{\prime}}}\chi^{\ast}_{D}$ for at least one of $i=1,2$. Suppose not, then by \Cref{item2}, we have $$\Proj_{W_{Z_{1}}}\chi^{\ast}_{D}=\Proj_{W_{Z_{1}^{\prime}}}\chi^{\ast}_{D}=\Proj_{W_{Z_{2}^{\prime}}}\chi^{\ast}_{D}=\Proj_{W_{Z_{2}}}\chi^{\ast}_{D},$$ contradicting the assumption that $L(S_{1})$ and $L(S_{2})$ are in different $\chi_{D}$-strata. The claim is proved.

Suppose $\Proj_{W_{Z_{1}}}\chi^{\ast}_{D}\neq\Proj_{W_{Z_{1}^{\prime}}}\chi^{\ast}_{D}$. By \Cref{local_constant_for_a_segment}, $$\Proj_{W_{Z_{1}}}\chi^{\ast}_{D^{\prime}}\neq \Proj_{W_{Z_{1}^{\prime}}}\chi^{\ast}_{D^{\prime}}$$ as well. By \Cref{decision_making_involves_type_one_walls} (applied to $S_{1}$), the line segment $\overline{DD^{\prime}}$ intersects a type one wall properly, a contradiction. The same arguments works for the case that $\Proj_{W_{Z_{2}}}\chi^{\ast}_{D}\neq\Proj_{W_{Z_{2}^{\prime}}}\chi^{\ast}_{D}$. The theorem is proved.
\end{proof}

\begin{theorem}\label{type_two_variation_crosses_type_two_walls}
Let $X_{\Sigma}$ be a projective toric variety and $D,D^{\prime}$ be two elements in $\Amp(X_{\Sigma})_{\mathbf{R}}$. If the stratifications induced by $D$ and $D^{\prime}$ undergo a type two variation and the line segment $\overline{DD^{\prime}}$ intersects no type one walls properly, then $\overline{DD^{\prime}}$ intersects a type two wall properly.
\end{theorem}

\begin{proof}
Without loss of generality, we may pick a pair $S_{1},S_{2}$ of elements in $\mathcal{L}$ so that $$M^{D}(x)>M^{D}(y)\text{ but }M^{D^{\prime}}(x)\leq M^{D^{\prime}}(y)\text{ for all }x\in L(S_{1}),y\in L(S_{2}).$$
By \Cref{order_is_continuous}, there exists a $D^{\prime\prime}\in\overline{DD^{\prime}}$ such that $M^{D^{\prime\prime}}(x)=M^{D^{\prime\prime}}(y)$ for all $x\in L(S_{1})$ and $y\in L(S_{2})$. Suppose $M^{D^{\prime\prime}}(x)=||\Proj_{Z_{1}^{\prime\prime}}\chi^{\ast}_{D^{\prime\prime}}||$ and $M^{D^{\prime\prime}}(y)=||\Proj_{W_{Z_{2}^{\prime\prime}}}\chi^{\ast}_{D}||$ where $Z_{1}^{\prime\prime}$, $Z_{2}^{\prime\prime}$ are subsets of $S_{1}$ and $S_{2}$ respectively. Let us consider the collection $$H=\{\tilde{D}\in\Amp(X_{\Sigma})_{\mathbf{R}}\bigm|||\Proj_{W_{Z_{1}^{\prime\prime}}}\chi^{\ast}_{\tilde{D}}||=||\Proj_{W_{Z_{2}^{\prime\prime}}}\chi^{\ast}_{\tilde{D}}||\}.$$

Obviously $D^{\prime\prime}\in H\cap\overline{DD^{\prime}}$. We will show that $H$ is a type two wall that intersects $\overline{DD^{\prime}}$ properly. For this, we prove that
\begin{enumerate}
    \item there is no containment between $W_{Z_{1}^{\prime\prime}}$ and $W_{Z_{2}^{\prime\prime}}$, and 
    \item $H\cap\overline{DD^{\prime}}\subsetneq \overline{DD^{\prime}}.$ 
\end{enumerate}
Note that statement (1) and the fact that $D^{\prime\prime}\in H$ implies that $H$ is either the entire ample cone or a type two wall. Statement (2) implies that $H$ is a type two wall and $H$ intersects the line segment $\overline{DD^{\prime}}$ properly.

To prove (1), suppose there is a containment, say $W_{Z_{1}^{\prime\prime}}\subset W_{Z_{2}^{\prime\prime}}$. Then the condition that $||\Proj_{W_{Z_{1}^{\prime\prime}}}\chi^{\ast}_{D^{\prime\prime}}||=||\Proj_{W_{Z_{2}^{\prime\prime}}}\chi^{\ast}_{D^{\prime\prime}}||$ actually implies $\Proj_{W_{Z_{1}^{\prime\prime}}}\chi^{\ast}_{D^{\prime\prime}}=\Proj_{W_{Z_{2}^{\prime\prime}}}\chi^{\ast}_{D^{\prime\prime}}$, causing a type one variation at $D^{\prime\prime}$ on the line segment $\overline{DD^{\prime}}$. By \Cref{type_one_wall_captures_type_one_variation}, the line segment $\overline{DD^{\prime\prime}}$ intersects a type one wall properly, a contradiction to the assumption about the line segment $\overline{DD^{\prime}}.$  Statement (1) is proved. 

For statement (2), let $Z_{1},Z_{2}$ be subsets of $S_{1}$ and $S_{2}$ respectively such that $\lambda^{D}_{S_{1}}=-\Proj_{W_{Z_{1}}}\chi^{\ast}_{D}$ and $\lambda_{S_{2}}^{D}=-\Proj_{W_{Z_{2}}}\chi^{\ast}_{D}$. We claim that $\Proj_{W_{Z_{i}}}\chi^{\ast}_{D}=\Proj_{W_{Z_{i}^{\prime\prime}}}\chi^{\ast}_{D}$ for both $i=1,2$.

If the claim is not true, say for $i=1$, then by \Cref{local_constant_for_a_segment}, $\Proj_{W_{Z_{1}}}\chi^{\ast}_{D^{\prime\prime}}\neq\Proj_{W_{Z_{1}^{\prime\prime}}}\chi^{\ast}_{D^{\prime\prime}}$ as well.  \Cref{decision_making_involves_type_one_walls} (applied to $S_{1}$) then implies that the line segment $\overline{DD^{\prime\prime}}$ intersects a type one wall properly, a contradiction to the assumption about $\overline{DD^{\prime}}$. The same argument works for the case $i=2$. Hence the claim is true. 

Now suppose on the contrary, that $H\cap\overline{DD^{\prime}}=\overline{DD^{\prime}}$. The claim together with the assumption that $H\cap\overline{DD^{\prime}}=\overline{DD^{\prime}}$ then imply 
\begin{equation*}\begin{split}
M^{D}(x)=&||\Proj_{W_{Z_{1}}}\chi^{\ast}_{D}||=||\Proj_{W_{Z_{1}^{\prime\prime}}}\chi^{\ast}_{D}||\\
=&||\Proj_{W_{Z_{2}^{\prime\prime}}}\chi^{\ast}_{D}||=
||\Proj_{W_{Z_{2}}}\chi^{\ast}_{D}||=M^{D}(y)\end{split}\end{equation*} for all $x\in L(S_{1})$ and $y\in L(S_{2})$, a contradiction. Statement (2) and therefore the theorem is proved.
\end{proof}
With \Cref{struct_walls}, \Cref{finite_wall_semi-chamber}, \Cref{Toric_VSIT_decomposition_ample_cone}, \Cref{type_one_wall_captures_type_one_variation}, and \Cref{type_two_variation_crosses_type_two_walls}, we obtain the summary:
\begin{theorem}
Let $X_{\Sigma}$ be a projective toric variety. There are two types of walls in the cone of ample divisors, called type one walls and type two walls respectively. The following properties hold for walls:
\begin{enumerate}
    \item There are finitely many walls. 
    \item A type one (resp. type two) wall is a rational hyperplane (resp. homogeneous quadratic hypersurface).
    \item Type one walls capture type one variations in the following sense: Let $D,D^{\prime}$ be two $\mathbf{R}$-ample divisors on $X_{\Sigma}$. Then the stratifications induced by $D$ and $D^{\prime}$ undergo a type one variation only if the line segment $\overline{DD^{\prime}}$ intersects a type one wall properly in the ample cone. Namely, $\overline{DD^{\prime}}$ crosses a type one wall.
    \item Type two walls capture type two variations in the following sense:  Let $D,D^{\prime}$ be two $\mathbf{R}$-ample divisors on $X_{\Sigma}$. If the stratifications induced by $D$ and $D^{\prime}$ undergo a type two variation and if the line segment $\overline{DD^{\prime}}$ intersects no type one walls properly, then $\overline{DD^{\prime}}$ intersects a type two wall properly in the ample cone. Namely, $\overline{DD^{\prime}}$ crosses a type two wall.\end{enumerate}
Away from the walls, the ample cone is decomposed into semi-chambers such that the following properties hold:
\begin{enumerate}
    \item There are finitely many semi-chambers.
    \item A semi-chamber is a cone, possibly not convex.
    \item The stratification stays constant in a semi-chamber in the following sense: If two $\mathbf{R}$-ample divisors are in the same semi-chamber, then they induce equivalent stratifications of $\mathbf{C}^{\Sigma(1)}$.
\end{enumerate}
\end{theorem}

\subsection{Variation of stratification - an elementary example}\label{Variation of stratification-an_elementary_example}
The following example illustrates the two types of variations. This is from the smooth toric surface obtained by blowing up $\mathbf{P}^{2}_{\mathbf{C}}$ at a point.

\subsubsection{Set up}
Consider the complete fan 
$$\begin{tikzpicture}
    \draw (-2,2) -- (0,0);
    \draw (0,0) -- (0,2);
    \draw (0,0) -- (2,0);
    \draw (0,0) -- (0,-2);
    \draw (-2.2,2.2) node {$u_{0}$};
    \draw (0,2.2) node {$u_{1}$};
    \draw (2.3,0) node {$u_{2}$};
    \draw (0,-2.2) node {$u_{3}$};
    \filldraw [black] (-2,2) circle (2pt);
    \filldraw [black] (0,2) circle (2pt);
    \filldraw [black] (2,0) circle (2pt);
    \filldraw [black] (0,-2) circle (2pt);
    \draw (-1.5,0) node {$\tau_{3}$};
    \draw (-0.8,1.3) node {$\tau_{0}$};
    \draw (1,1) node {$\tau_{1}$};
    \draw (1,-1) node {$\tau_{2}$};
\end{tikzpicture}$$
where \begin{enumerate}
    \item $u_{0}=(-1,1)$,
    \item $u_{1}=(0,1)$,
    \item $u_{2}=(1,0)$,
    \item $u_{3}=(0,-1)$,
\end{enumerate}
and 
\begin{enumerate}
    \item $\tau_{0}=\text{Cone }(u_{0},u_{1})$,
    \item $\tau_{1}=\text{Cone }(u_{1},u_{2})$,
    \item $\tau_{2}=\text{Cone }(u_{2},u_{3})$,
    \item $\tau_{3}=\text{Cone }(u_{0},u_{3})$.
\end{enumerate}
We will compute the stratification of $\mathbf{C}^{\Sigma(1)}=\mathbf{C}^{4}$ induced by all ample divisors on $X_{\Sigma}$. Let us first compute the ample cone of $X_{\Sigma}$.

\subsubsection{The ample cone}
To begin with, sequence (\ref{sesWDiv}) in this case is $0\rightarrow\mathbf{Z}^{2}\rightarrow\mathbf{Z}^{4}\rightarrow\Cl(X_{\Sigma})\rightarrow 0$ where $\mathbf{Z}^{2}\rightarrow\mathbf{Z}^{4}$ is defined by the matrix 
$\begin{pmatrix}
-1 & 1\\
0 & 1\\
1 & 0\\
0 & -1
\end{pmatrix}$. Therefore, we have $D_{0}=D_{2}$ and $D_{0}+D_{1}=D_{3}$ in $\Cl(X_{\Sigma})$. One also checks that $\{D_{0},D_{1}\}$ is linearly independent in $\Cl(X_{\Sigma})$ so that it forms a $\mathbf{Z}$-basis of $\Cl(X_{\Sigma})$. With respect to this basis sequence (\ref{sesWDiv}) is represented by  
\begin{equation}\label{sesWDiv_blow_up_P2}0\rightarrow\mathbf{Z}^{2}\rightarrow\mathbf{Z}^{4}\rightarrow\mathbf{Z}^{2}\rightarrow 0\end{equation} where $\mathbf{Z}^{4}\rightarrow\mathbf{Z}^{2}$ is given by the matrix $\begin{pmatrix}
1 & 0 & 1 & 1\\
0 & 1 & 0 & 1
\end{pmatrix}.$

Let $D=a_{0}D_{0}+a_{1}D_{1}$ be a divisor. Then we have 
\begin{enumerate}
    \item $m_{\tau_{0}}=(a_{0}-a_{1},-a_{1})$,
    \item $m_{\tau_{1}}=(0,-a_{1})$,
    \item $m_{\tau_{2}}=(0,0)$, and 
    \item $m_{\tau_{3}}=(a_{0},0)$.
\end{enumerate}
By \Cref{numampleness}, $D$ is ample if and only if 
$$\begin{cases}
\langle m_{\tau_{0}},u_{2}\rangle = a_{0}-a_{1}>0,\\
\langle m_{\tau_{0}},u_{3}\rangle = a_{1}>0,\\
\langle m_{\tau_{1}}, u_{0}\rangle = -a_{1}>-a_{0},\\
\langle m_{\tau_{1}}, u_{3}\rangle = a_{1} > 0,\\
\langle m_{\tau_{2}}, u_{0}\rangle = 0>-a_{0},\\
\langle m_{\tau_{2}}, u_{1}\rangle = 0>-a_{1},\\
\langle m_{\tau_{3}}, u_{1}\rangle = 0>-a_{1},\\
\langle m_{\tau_{3}}, u_{2}\rangle =a_{0} >0.
\end{cases}$$
This system of inequalities is essentially $a_{0}>a_{1}>0$. We now compute several things about $G$.

\subsubsection{The group $G$}
Applying $\mathrm{Hom}_{\mathbf{Z}}(-,\mathbf{C}^{\times})$ to sequence (\ref{sesWDiv_blow_up_P2}), we see that $G\simeq(\mathbf{C}^{\times})^{2}$ injects into $(\mathbf{C}^{\times})^{\Sigma(1)}=(\mathbf{C}^{\times})^{4}$ by the formula $$(t_{1},t_{2})\mapsto (t_{1},t_{2},t_{1},t_{1}\cdot t_{2}).$$ Applying $\mathrm{Hom}_{\mathbf{Z}}(-,\mathbf{Z})$ to sequence (\ref{sesWDiv_blow_up_P2}), we see that $\pmb{\Gamma}(G)\simeq\mathbf{Z}^{2}$ injects into $\mathbf{Z}^{4}$ by the matrix $\begin{pmatrix}
1 & 0\\
0 & 1\\
1 & 0\\
1 & 1
\end{pmatrix}$ with respect to the dual basis $\{D_{0}^{\ast},D_{1}^{\ast}\}$ of $\pmb{\Gamma}(G)$.
It follows that the inner product $(-,-)$ on $\pmb{\Gamma}(G)_{\mathbf{R}}$ satisfies 
$$(aD_{0}^{\ast}+bD_{1}^{\ast},cD_{0}^{\ast}+dD_{1}^{\ast})=ac+bd+ac+(a+b)\cdot (c+d).$$ 

From now on we will write a divisor on $X_{\Sigma}$ as a tuple $(a_{0},a_{1})$ instead of $a_{0}D_{0}+a_{1}D_{1}$. Similarly we write an element in $\pmb{\Gamma}(G)_{\mathbf{R}}$ as a tuple $(a,b)$ instead of $aD_{0}^{\ast}+bD_{1}^{\ast}$.
\subsubsection{Stratifications induced by \texorpdfstring{$\mathbf{R}$}{R}-ample divisors}
We are now ready to analyze the stratification induced by all $\mathbf{R}$-ample divisors. Here the primitive collections are $\{\rho_{0},\rho_{2}\}$ and $\{\rho_{1},\rho_{3}\}$ so $$\mathcal{L}=\{\emptyset\},\{0\},\{2\},\{1\},\{3\},\{0,2\},\{1,3\}.$$ The idea is to write the one parameter subgroup $\lambda^{D}_{S}$ that is $\chi_{D}$-adapted to $L(S)$ as a piecewise function of $D$ for each $S\in\mathcal{L}$. For brevity, when we write $\lambda^{D}_{S}$, $\Lambda^{D}_{S}$ $\sigma_{S}$ and $W_{S}$, we do not separate elements in $S$ by commas. We will write $L_{S}$ instead of $L(S)$ and without commas as well.

It is computed that \begin{equation}\label{chi_ast}
 -\chi^{\ast}_{D}=(\frac{-2a_{0}+a_{1}}{5},\frac{a_{0}-3a_{1}}{5})
.\end{equation} and 
$-\Proj_{W_{13}}\chi^{\ast}_{D}=(0,0)$ as $W_{13}=0$

For $S=\emptyset$, note that $L_{\emptyset}$ is the origin in $\mathbf{C}^{4}$ and $\sigma_{\emptyset}=\pmb{\Gamma}(G)_{\mathbf{R}}$. Hence it is obvious that $\lambda^{D}_{\emptyset}=-\chi^{\ast}_{D}$.

For $S=\{0\},\{2\}$, or $\{0,2\}$, we have $\sigma_{S}=\{(a,b)\in\pmb{\Gamma}(G)_{\mathbf{R}}\vert a\geq 0\}$. Looking at \Cref{chi_ast}, $-\chi^{\ast}_{D}\in\Lambda^{D}_{S}$ if and only if $-2a_{0}+a_{1}\geq 0$. This is not possible in the ample cone. It is computed that $-\Proj_{W_{S}}\chi^{\ast}_{D}=(0,-\frac{a_{1}}{2}).$ We conclude that $$\lambda^{D}_{0}=\lambda^{D}_{2}=\lambda^{D}_{02}=(0,-\frac{a_{1}}{2}).$$

For $S=\{1\}$, we have $\sigma_{1}=\{(a,b)\in\pmb{\Gamma}(G)_{\mathbf{R}}\vert b\geq 0\}$. Looking at \Cref{chi_ast}, we have $-\chi^{\ast}_{D}\in\Lambda^{D}_{1}$ if and only if $a_{0}-3a_{1}\geq 0$, which is possible in the ample cone. It is computed that $-\Proj_{W_{1}}\chi^{\ast}_{D}=(-\frac{a_{0}}{3},0).$ We conclude that 
$$\lambda^{D}_{1}=\begin{cases}
-\chi^{\ast}_{D}\text{ if }a_{0}-3a_{1}\geq 0, \text{ or}\\
-\Proj_{W_{1}}\chi^{\ast}_{D}=(-\frac{a_{0}}{3},0)\text{ otherwise}.
\end{cases}$$

For $S=\{3\}$, we have $\sigma_{3}=\{(a,b)\in\pmb{\Gamma}(G)_{\mathbf{R}}\vert a+b\geq 0\}.$ Looking at \Cref{chi_ast}, we have $-\chi^{\ast}_{D}\in\Lambda^{D}_{3}$ if and only if $-(a_{0}+2a_{1})\geq 0$, which is impossible in the ample cone. It is computed that $-\Proj_{W_{3}}\chi^{\ast}_{D}=(\frac{-a_{0}+a_{1}}{3},\frac{a_{0}-a_{1}}{3})$. We conclude that $$\lambda^{D}_{3}=-\Proj_{W_{3}}\chi^{\ast}_{D}=(\frac{-a_{0}+a_{1}}{3},\frac{a_{0}-a_{1}}{3}).$$

For $S=\{1,3\}$, we have $\sigma_{13}=\{(a,b)\in\pmb{\Gamma}(G)_{\mathbf{R}}\vert a+b,b\geq 0\}.$ By what we had, $-\chi^{\ast}_{D}$, $-\Proj_{W_{1}}\chi^{\ast}_{D}$ are not in $\Lambda^{D}_{13}$, but $-\Proj_{W_{3}}\chi^{\ast}_{D}\in\Lambda^{D}_{13}.$ On the other hand $W_{13}\subset W_{3}$ so that $||\Proj_{W_{13}}\chi_{D}^{\ast}||\leq ||\Proj_{W_{3}}\chi^{\ast}_{D}||$. We conclude that $$\lambda^{D}_{13}=-\Proj_{W_{3}}\chi^{\ast}_{D}=(\frac{-a_{0}+a_{1}}{3},\frac{a_{0}-a_{1}}{3}).$$

We obtain the following description of the set of $\chi_{D}$-strata for each ample divisor $D=(a_{0},a_{1})$: If $a_{0}-3a_{1}\geq 0$, then 
\begin{itemize}
    \item $L_{\emptyset}\cup L_{1}=V(x_{0},x_{2},x_{3})$ is the $\chi_{D}$-stratum indexed by $-\chi^{\ast}_{D}$,
    \item $L_{0}\cup L_{2}\cup L_{02}=V(x_{1},x_{3})-V(x_{0},x_{2})$ is the $\chi_{D}$-stratum indexed by $(0,-\frac{a_{1}}{2})$, and 
    \item $L_{3}\cup L_{13}=V(x_{0},x_{2})-V(x_{3})$ is the $\chi_{D}$-stratum indexed $(\frac{-a_{0}+a_{1}}{3},\frac{a_{0}-a_{1}}{3})$ 
\end{itemize}
If $a_{0}-3a_{1}<0$, then 
\begin{itemize}
    \item $L_{\emptyset}=V(x_{0},x_{1},x_{2},x_{3})$ is the $\chi_{D}$-stratum indexed by $-\chi^{\ast}_{D}$,
    \item $L_{1}=V(x_{0},x_{2},x_{3})-V(x_{1})$ is the $\chi_{D}$-stratum indexed by $(-\frac{a_{0}}{3},0)$,
    \item $L_{0}\cup L_{2}\cup L_{02}=V(x_{1},x_{3})-V(x_{0},x_{2})$ is the $\chi_{D}$-stratum indexed by $(0,-\frac{a_{1}}{2})$, and 
    \item $L_{3}\cup L_{13}=V(x_{0},x_{2})-V(x_{3})$ is the $\chi_{D}$-stratum indexed by $(\frac{-a_{0}+a_{1}}{3},\frac{a_{0}-a_{1}}{3})$. 
\end{itemize}
Crossing the linear wall $a_{0}-3a_{1}=0$, we see that the stratification undergoes a type one variation. There are type two variations among the order of the strata which we now describe. 

When $a_{0}-3a_{1}\geq 0$, we only need to compare the order between the $\chi_{D}$-stratum indexed by $(0,-\frac{a_{1}}{2})$ and $(\frac{-a_{0}+a_{1}}{3},\frac{a_{0}-a_{1}}{3})$. We have $$||(\frac{-a_{0}+a_{1}}{3},\frac{a_{0}-a_{1}}{3})||=\frac{a_{0}-a_{1}}{\sqrt{3}}\geq\frac{2a_{1}}{\sqrt{3}}>\frac{a_{1}}{\sqrt{2}}=||(0,-\frac{a_{1}}{2})||.$$ Hence the $\chi_{D}$-strata are ordered as $$V(x_{0},x_{2},x_{3})>V(x_{0},x_{2})-V(x_{3})>V(x_{1},x_{3})-V(x_{0},x_{2}).$$

When $a_{0}-3a_{1}<0$, we have to compare the order between the $\chi_{D}$-strata indexed by $(-\frac{a_{0}}{3},0),(0,-\frac{a_{1}}{2})$ and $(\frac{-a_{0}+a_{1}}{3},\frac{a_{0}-a_{1}}{3})$. Note that we always have $$||(-\frac{a_{0}}{3},0)||=\frac{a_{0}}{\sqrt{3}}>\frac{a_{0}-a_{1}}{\sqrt{3}}=||(\frac{-a_{0}+a_{1}}{3},\frac{a_{0}-a_{1}}{3})||.$$  It remains to compare $(0,-\frac{a_{1}}{2}), (\frac{-a_{0}+a_{1}}{3},\frac{a_{0}-a_{1}}{3})$ and $(-\frac{a_{0}}{3},0),(0,-\frac{a_{1}}{2})$.

It is computed that 
$$||(0,-\frac{a_{1}}{2})||>||(\frac{-a_{0}+a_{1}}{3},\frac{a_{0}-a_{1}}{3})||\Leftrightarrow a_{0}<(1+\sqrt{\frac{3}{2}})\cdot a_{1}$$ and $$||(0,-\frac{a_{1}}{2})||>||(-\frac{a_{0}}{3},0)||\Leftrightarrow a_{0}<\sqrt{\frac{3}{2}}\cdot a_{1}.$$

\subsubsection{Summary}
We may divide the ample cone into six regions where each region corresponds to a distinct class of stratification of $\mathbf{C}^{4}$ (see \Cref{strat_summary_blow_up_P2}). 

$$\begin{tikzpicture}[
roundnode/.style={circle, draw=green!60, fill=green!5, very thick, minimum size=7mm},
squarednode/.style={rectangle, draw=red!60, fill=red!5, very thick,inner sep =0pt, minimum size=2mm},
]
\filldraw[color=green] (0,0) -- (6,0) -- (6,6);
\draw[color = purple, ultra thick] (0,0) -- (6,0);
\draw[->] (0,0) -- (6,0);
\draw (6.3,0) node {$a_{0}$};
\draw (0,0)[color = green] -- (6,6);
\draw[color=red] (0,0) -- (6,2);
\draw [->](0,0) -- (0,6);
\draw (0,6.3) node {$a_{1}$};
\draw[color = blue] (6,4.9) -- (0,0);
\draw [color=blue] (0,0) -- (6,2.7);
\draw [color = purple, ultra thick] (0,0) -- (6,6);
\node[squarednode] at (5.8,1) {1};
\node[squarednode] at (5.8,2.15) {2};
\node[squarednode] at (5.8,2.6) {3};
\node[squarednode] at (5.8,4) {4};
\node[squarednode] at (5.8,4.8) {5};
\node[squarednode] at (5.8,5.5) {6};
\draw (6.9,2) node {$a_{0}=3a_{1}$};
\draw (7.5,2.7) node {$a_{0}=(1+\sqrt{\frac{3}{2}})a_{1}$};
\draw (7.0,4.9) node {$a_{0}=\sqrt{\frac{3}{2}}a_{1}$};
\draw (6.7,6) node {$a_{0}=a_{1}$};
\end{tikzpicture}$$
Here the line defined by $a_{0}=a_{1}$ and the $a_{0}$-axis form the boundary of the ample cone. The two blue lines $a_{0}=\sqrt{\frac{3}{2}}a_{1}$ and $a_{0}=(1+\sqrt{\frac{3}{2}})a_{1}$ were labeled as 5 and 3. They are the type two walls with respect to $\{0\},\{1\}$ and $\{0\},\{3\}$ respectively. The red line $a_{0}=3a_{1}$ is the type one wall with respect to $\emptyset,\{1\}$.
\begin{table}[ht]
    \caption{Stratifications induced by $\mathbf{R}$-ample divisors in each area. Here $\mathbf{0}$ represents the origin in $\mathbf{C}^{4}$.}
\begin{tabular}{|c|l|}
\hline
\text{Area} &\multicolumn{1}{|c|}{\text{Stratification}}\\
\hline
1     & $V(x_{0},x_{2},x_{3})>V(x_{0},x_{2})-V(x_{3})>V(x_{1},x_{3})-V(x_{0},x_{2})$\\
\hline
2&$\mathbf{0}>V(x_{0},x_{2},x_{3})-V(x_{1})>V(x_{0},x_{2})-V(x_{3})>V(x_{1},x_{3})-V(x_{0},x_{2})$\\
\hline
\multirow{2}{*}{3}&$\mathbf{0}>V(x_{0},x_{2},x_{3})-V(x_{1})>V(x_{0},x_{2})-V(x_{3})$\\
    &$\mathbf{0}>V(x_{0},x_{2},x_{3})-V(x_{1})>V(x_{1},x_{3})-V(x_{0},x_{2})$\\
\hline
4 & $\mathbf{0}>V(x_{0},x_{2},x_{3})-V(x_{1})>V(x_{1},x_{3})-V(x_{0},x_{2})>V(x_{0},x_{2})-V(x_{3})$\\
\hline
\multirow{2}{*}{5}&$\mathbf{0}>V(x_{0},x_{2},x_{3})-V(x_{1})>V(x_{0},x_{2})-V(x_{3})$\\
&$\mathbf{0}>V(x_{1},x_{3})-V(x_{0},x_{2})>V(x_{0},x_{2})-V(x_{3})$\\
\hline
6 & $\mathbf{0}>V(x_{1},x_{3})-V(x_{0},x_{2})>V(x_{0},x_{2},x_{3})-V(x_{1})>V(x_{0},x_{2})-V(x_{3})$\\
\hline
\end{tabular}\label{strat_summary_blow_up_P2}\end{table}

We see that:
\begin{enumerate}
    \item There are six areas corresponding to six SIT-equivalence classes of $\mathbf{R}$-ample divisors. 
    \item The stratification does not change inside a semi-chamber.
    \item Type two walls capture type two variations. More specifically crossing a type two wall swaps the order of a pair of strata and on the wall the strict partial order breaks up into two linearly ordered chains.
    \item The type one wall captures a type one variation. More specifically crossing the type one wall downwards combines $L_{1}$ with $\mathbf{0}$ into a single strata.
\end{enumerate}

\subsection{Relations to the structure of the fan}\label{Relations to the structure of the fan}
\subsubsection{A VSIT adjunction}\label{A VSIT adjunction}
In this section we prove that toric VSIT is intrinsic to the two combinatorial structures of a fan: its  primitive collections and the relations among its ray generators. For this we formulate an equivalence between complete fans that capture the two combinatorial structures.

\begin{definition}\label{ample_equivalence_between_fans}
Let $N_{1}$ and $N_{2}$ be two free abelian groups of the same rank. Let
$\Sigma_{1}\subset(N_{1})_{\mathbf{R}},\Sigma_{2}\subset(N_{2})_{\mathbf{R}}$
be two complete fans. We say $\Sigma_{1}$ and $\Sigma_{2}$ are \emph{amply equivalent} if there is a bijection $\Psi:\Sigma_{1}(1)\rightarrow\Sigma_{2}(1)$ between the rays such that the following two properties hold:
\begin{enumerate}
    \item $\Psi$ preserves primitive collections. That is, for any subset $C\subset\Sigma_{1}(1)$, $\Psi(C)$ is a primitive collection of $\Sigma_{2}$ if and only if $C$ is a primitive collection of $\Sigma_{1}$.
    \item $\Psi$ preserves relations among the ray generators. That is, for any tuple of integers $\{a_{\rho}\}_{\rho\in\Sigma_{1}(1)}$, $\sum_{\rho}a_{\rho}u_{\rho}=0$ if and only if $\sum_{\rho}a_{\rho}u_{\Psi(\rho)}=0$ where $u_{\rho}$ (resp. $u_{\Psi(\rho)}$) represents the ray generator of $\rho$ (resp. $\Psi(\rho)$). 
\end{enumerate}
\end{definition}

We will show that for any pair of amply equivalent fans, the toric varieties associated to them possess isomorphic $\mathbf{Q}$-ample cones. 
To construct such isomorphisms, we first need 

\begin{proposition}\label{StructCompFan}
Let $N_{1}$ and $N_{2}$ be two free abelian groups of rank $n$. Let
$\Sigma_{1}\subset(N_{1})_{\mathbf{R}},\Sigma_{2}\subset(N_{2})_{\mathbf{R}}$
be two amply equivalent fans with the bijection $\Psi:\Sigma_{1}(1)\rightarrow\Sigma_{2}(1)$ that preserves their primitive collections and relations among their ray generators. Then there is a $\mathbf{Q}$-linear isomorphism $\Phi:(N_{1})_{\mathbf{Q}}\rightarrow (N_{2})_{\mathbf{Q}}$ such that
\begin{enumerate}
\item $\Phi(u_{\rho})=u_{\Psi(\rho)}$ for all $\rho\in\Sigma_{1}(1)$, and
\item if $\Phi_{\mathbf{R}}$ is the extension of $\Phi$ over $\mathbf{R}$, then a cone $\sigma_{1}\subset(N_{1})_{\mathbf{R}}$ is a cone in $\Sigma_{1}$ if and only if
  $\Phi_{\mathbf{R}}(\sigma_{1})$ is a cone in $\Sigma_{2}.$
  \end{enumerate}
  Moreover, if $G_{i}=\Hom_{\mathbf{Z}}(\Cl(X_{\Sigma_{i}}),\mathbf{C}^{\times})$, then 
  \begin{enumerate}[resume]
 \item $\Phi$ induces an isomorphism $\pmb{\Gamma}(G_{1})_{\mathbf{Q}}\simeq\pmb{\Gamma}(G_{2})_{\mathbf{Q}}$.
\end{enumerate}
\end{proposition}

\begin{proof}
To construct the isomorphism $\Phi$, pick $n$ $\mathbf{Q}$-linearly
independent ray generators $\{u_{\rho_{1}},\dots, u_{\rho_{n}}\}$ from $\Sigma_{1}$. We may do so as $\Sigma_{1}$ is complete. We then have a  $\mathbf{Q}$-linear map
$\Phi:(N_{1})_{\mathbf{Q}}\rightarrow(N_{2})_{\mathbf{Q}}$ defined by
sending $u_{\rho_{i}}$ to $u_{\Psi(\rho_{i})}$. By the assumption on the relation among
ray generators, $\{u_{\Psi(\rho_{1})},\ldots, u_{\Psi(\rho_{n})}\}$ is independent over $\mathbf{Q}$ as well. The map $\Phi$ is therefore
an isomorphism.

We now show that $\Phi(u_{\rho})=u_{\Psi(\rho)}$ for any other $u_{\rho}$. Since $\{u_{\rho_{1}},\ldots,
u_{\rho_{n}}\}$ is a $\mathbf{Q}$-basis of $(N_{1})_{\mathbf{Q}}$, there
exist a nonzero $K\in\mathbf{Z}$ and integers $a_{i}$ such that
$K\cdot u_{\rho}=\sum_{i}a_{i}u_{\rho_{i}}.$ By the assumption on the
relation among ray generators, we also have $K\cdot
u_{\Psi(\rho)}=\sum_{i}a_{i}u_{\Psi(\rho_{i})}$. Therefore, we get 
$$K\cdot\Phi(u_{\rho})=\Phi(K\cdot
u_{\rho})=\sum_{i}a_{i}\Phi(u_{\rho_{i}})=\sum_{i}a_{i}u_{\Psi(\rho_{i})}=K\cdot
u_{\Psi(\rho)}.$$
Hence we have $\Phi(u_{\rho})=u_{\Psi(\rho)}$ for
all $\rho$. This proves statement(1).

For statement (2) on the cones, it is sufficient to show that $\Phi_{\mathbf{R}}(\sigma_{1})\in\Sigma_{2}(n)$ for any $\sigma_{1}\in\Sigma_{1}(n)$ because we can apply the result to $\Phi_{\mathbf{R}}^{-1}$ on the cones in $\Sigma_{2}(n)$ and because a fan is completely determined by its maximal cones. We can further reduce to show that for any $\sigma_{1}\in\Sigma_{1}(n)$, there exists a cone $\sigma_{2}\in\Sigma_{2}(n)$ such that $\Phi_{\mathbf{R}}(\sigma_{1})\subset \sigma_{2}$. For then we can apply the result to $\Phi_{\mathbf{R}}^{-1}$ on $\sigma_{2}$ and get another inclusion $$\sigma_{1}\subset\Phi^{-1}_{\mathbf{R}}(\sigma_{2})\subset \tilde{\sigma_{1}}\text{ for some }\tilde{\sigma_{1}}\in\Sigma_{1}(n).$$ The maximality of $\sigma_{1}$ as a cone in $\Sigma_{1}$ would imply $\sigma_{1}=\Phi_{\mathbf{R}}^{-1}(\sigma_{2}),$ which in turn implies $\Phi_{\mathbf{R}}(\sigma_{1})=\sigma_{2}$.

Let $\sigma_{1}$ be a cone in $\Sigma_{1}(n)$ and $L_{1}=\sigma_{1}(1)$ be the set of rays in
$\sigma_{1}$. Then any subset of $L_{1}$ (including $L_{1}$) does not form a primitive collection of $\Sigma_{1}$. By assumption any subset of $\Psi(L_{1})$ does not form a primitive collection of $\Sigma_{2}$. Therefore, either 
\begin{enumerate}
\item $\{u_{\Psi(\rho)}\}_{\rho\in L_{1}}$ is contained the list of ray
  generators of some cone in $\Sigma_{2}$, or 
\item there is a proper subset $L_{2}\subsetneq L_{1}$ such that 
  $\{u_{\Psi(\rho)}\}_{\rho\in L_{2}}$ is not contained in the list of ray
  generators of any cone in $\Sigma_{2}$.
\end{enumerate}
Suppose condition (1) fails and let $L_{2}\subsetneq L_{1}$ be the subset from condition (2). As $\Psi(L_{2})$ does not form a primitive
collection in $\Sigma_{2}$, the same two conditions can be said about
$\Psi(L_{2})$. By construction of $L_{2}$, condition (1) fails for $\Psi(L_{2})$ also. Hence we obtain a proper subset
$L_{3}\subsetneq L_{2}$ such that (1) fails for $\Psi(L_{3})$. This process cannot continue indefinitely as $L_{1}$ is a finite set. Therefore, condition (1) holds for $\Psi(L_{1})$ in the first place. Namely, $\{u_{\Psi(\rho)}\}_{\rho\in L_{1}}$ is contained the list of ray
generators of some cone in $\Sigma_{2}$ 

Let $\sigma_{2}$ be a cone in $\Sigma_{2}$ whose list of ray generators contains
$\{u_{\Psi(\rho)}\}_{\rho\in L_{1}}$. Then clearly $\Phi_{\mathbf{R}}(\sigma_{1})\subset\sigma_{2}$. The validity of statement (2) is established.

For statement (3), let $\Sigma(1)$ be an indexing set of $\Sigma_{1}$ and let $\{e_{i}\}_{i\in\Sigma(1)}$ be the standard basis of $\mathbf{Q}^{\Sigma(1)}$. Define the $\mathbf{Q}$-linear maps $$\pi_{1}:\mathbf{Q}^{\Sigma(1)}\rightarrow (N_{1})_{\mathbf{Q}}\text{ } \big(\text{resp. }\pi_{2}:\mathbf{Q}^{\Sigma(1)}\rightarrow(N_{2})_{\mathbf{Q}}\big)$$ by $$e_{i}\mapsto u_{\rho_{i}}\text{ }\big(\text{resp. }e_{i}\mapsto u_{\Psi(\rho_{i})}\big).$$ Then $\Phi$ is compatible with these two maps by statement (1). Each $\pi_{i}$ is surjective with kernel $\pmb{\Gamma}(G_{i})_{\mathbf{Q}}$ by statement (4) from \Cref{lemma5.1.1}. Hence the isomorphism $\Phi$ fits into the following commutative diagram of short exact sequences \begin{equation}\label{amply_equi_diagram_SES}\begin{tikzcd}[sep = small]
0\arrow[r]&\pmb{\Gamma}(G_{1})_{\mathbf{Q}}\arrow[r]\arrow[d,"\varphi"]&\mathbf{Q}^{\Sigma(1)}\arrow[r]\arrow[d,equal]&(N_{1})_{\mathbf{Q}}\arrow[r]\arrow[d,"\Phi"]&0\\
0\arrow[r]&\pmb{\Gamma}(G_{2})_{\mathbf{Q}}\arrow[r]&\mathbf{Q}^{\Sigma(1)}\arrow[r]&(N_{2})_{\mathbf{Q}}\arrow[r]&0
\end{tikzcd}.\end{equation}
Since the right and middle vertical map are isomorphisms, the induced map $\varphi:\pmb{\Gamma}(G_{1})_{\mathbf{Q}}\rightarrow\pmb{\Gamma}(G_{2})_{\mathbf{Q}}$ is an isomorphism.
\end{proof}

Ample equivalence has the following implication when a complete fan is amply equivalent to a smooth complete fan:

\begin{corollary}\label{StrucSmooCompFan}
Let $\Sigma_{1}\subset(N_{1})_{\mathbf{R}},\Sigma_{2}\subset(N_{2})_{\mathbf{R}}$
be two amply equivalent complete fans with $\Sigma_{1}$ smooth. If $\Phi:(N_{1})_{\mathbf{Q}}\rightarrow
(N_{2})_{\mathbf{Q}}$
is the $\mathbf{Q}$-linear isomorphism constructed in \Cref{StructCompFan}, its restriction to $N_{1}$ factors through $N_{2}$ as a $\mathbf{Z}$-linear monomorphism $N_{1}\hookrightarrow N_{2}$.
Moreover, $\Sigma_{2}$ is simplicial and the following statements are
equivalent:
\begin{enumerate}
\item One of the maximal cones in $\Sigma_{2}$ is smooth.
\item The restriction of $\Phi$ to $N_{1}$ factors through $N_{2}$ as a $\mathbf{Z}$-isomorphism $N_{1}\simeq N_{2}$.
\item $\Sigma_{2}$ is a smooth fan. 
\end{enumerate}
In this case, $\varphi$ induces a toric isomorphism $X_{\Sigma_{1}}\simeq X_{\Sigma_{2}}$.
\end{corollary}

\begin{proof}
Since $\Sigma_{1}$ is smooth and complete, there is a maximal cone
$\sigma_{1}\in\Sigma_{1}$ whose ray generators form a $\mathbf{Z}$-basis for
$N_{1}$. Hence every element in $N_{1}$ is a $\mathbf{Z}$-combination of ray generators of $\sigma_{1}$. By the construction of $\Phi$, we have $\Phi(N_{1})\subset N_{2}$. Namely, the restriction of $\Phi$ to $N_{1}$ factors through $N_{2}$.

Let $\varphi:N_{1}\rightarrow N_{2}$ be the factorization. Since $\Phi:(N_{1})_{\mathbf{Q}}\rightarrow (N_{2})_{\mathbf{Q}}$ is an
isomorphism, $\varphi$ is injective.

That $\Sigma_{2}$ is simplicial follows directly from statement (2) in \Cref{StructCompFan} that every cone in $\Sigma_{2}$ is the image of a smooth cone in $\Sigma_{1}$ under the isomorphism $\Phi_{\mathbf{R}}$. 

Finally, $\varphi$ is an isomorphism if and only if $\varphi$ sends a $\mathbf{Z}$-basis of $N_{1}$ to a $\mathbf{Z}$-basis of $N_{2}$. The equivalence of the three statements now follows from statement (2) in \Cref{StructCompFan}.
\end{proof}

Here is an example of a pair of amply equivalent fans where one of them is smooth but the other is only simplicial.

\begin{example}
Let $\Sigma_{1}\subset\mathbf{R}^{2}$ be the complete fan with ray generators $u_{0}=(1,0)$, $u_{2}=(0,1)$, and $u_{3}=(-1,-1)$ and let $\Sigma_{2}\subset\mathbf{R}^{2}$ be the complete fan with ray generators $v_{0}=(2,3)$, $v_{1}=(1,-1)$ and $v_{2}=(-3,-2)$. Then since $\sum_{i}u_{i}=\sum_{i}v_{i}=(0,0)$, the assignment $$\text{Cone }(u_{i})\mapsto\text{Cone }(v_{i})$$ is a bijection between rays that preserves relations among the ray generators. Moreover, both $\Sigma_{1}(1)$ and $\Sigma_{2}(1)$ are the only primitive collections for $\Sigma_{1}$ and $\Sigma_{2}$ respectively. Hence $\Sigma_{1}$ and $\Sigma_{2}$ are amply equivalent but $\Sigma_{2}$ is not smooth.
\end{example}

\Cref{adjunction} below is an equality that we will use extensively. Let $k$ be a field. Suppose $V_{i},W_{i}$ for $i=1,2$ are finite dimensional vector
spaces over $k$ with a perfect pairing
$\langle-,-\rangle_{i}:V_{i}\times W_{i}\rightarrow k$ that identifies
$V_{i}$ and $W_{i}$ as duals respectively. Then for any $k$-linear map
$f:W_{1}\rightarrow W_{2}$ and $v\in V_{2}$, $w\in W_{1}$, we have 
\begin{equation}\label{adjunction}\langle v_{2},f(w)\rangle_{2}=\langle f^{\ast}(v_{2}),w\rangle_{1}\end{equation} where $f^{\ast}:V_{2}\rightarrow V_{1}$ is the dual of $f$.

We may now make sense of the term "ample equivalence" between fans. We show that two complete toric varieties where their fans are amply equivalent have isomorphic ample cones over the field of rational numbers (and therefore over the field of real numbers as well).

\begin{proposition}\label{why_amply_equivalent}
Let $\Sigma_{1}$ and $\Sigma_{2}$ be two amply equivalent fans and $$G_{i}=\Hom_{Z}(\Cl(X_{\Sigma_{i}}),\mathbf{C}^{\times})\text{ for }i=1,2.$$ Let $\varphi:\pmb{\Gamma}(G_{1})_{\mathbf{Q}}\rightarrow\pmb{\Gamma}(G_{2})_{\mathbf{Q}}$ be the isomorphism from statement (3) of \Cref{StructCompFan}. Then the dual of $\varphi$ is an isomorphism $\varphi^{\ast}:\Cl(X_{\Sigma_{2}})_{\mathbf{Q}}\rightarrow\Cl(X_{\Sigma_{1}})_{\mathbf{Q}}$ that restricts to $\Pic(X_{\Sigma_{2}})_{\mathbf{Q}}\simeq\Pic(X_{\Sigma_{1}})_{\mathbf{Q}}$ with $\varphi^{\ast}(\Amp(X_{\Sigma_{2}})_{\mathbf{Q}})= \Amp(X_{\Sigma_{1}})_{\mathbf{Q}}$.
\end{proposition}
\begin{proof}
Let $\Psi:\Sigma_{1}(1)\simeq\Sigma_{2}(1)$ be the bijection between rays that preserves the primitive collections and relations among the ray generators. By statment (4) from \Cref{lemma5.1.1}, the short exact sequences dual to those in diagram (\ref{amply_equi_diagram_SES}) fit into the following commutative diagram: $$\begin{tikzcd}[sep=small]
0\arrow[r]&(M_{2})_{\mathbf{Q}}\arrow[r]\arrow[d,"\Phi^{\ast}"]&\mathbf{Q}^{\Sigma(1)}\arrow[r]\arrow[d,equal]&\Cl(X_{\Sigma_{2}})_{\mathbf{Q}}\arrow[r]\arrow[d,"\varphi^{\ast}"]&0\\
0\arrow[r]&(M_{1})_{\mathbf{Q}}\arrow[r]&\mathbf{Q}^{\Sigma(1)}\arrow[r]&\Cl(X_{\Sigma_{1}})_{\mathbf{Q}}\arrow[r]&0
\end{tikzcd}$$
where $\varphi^{\ast}$ is an isomorphism and sends a class of $\mathbf{Q}$-Weil divisor
$\sum_{i}q_{i}D_{\Psi(\rho_{i})}$ in $\Cl(X_{\Sigma_{2}})_{\mathbf{Q}}$
to the class of $\sum_{i}q_{i}D_{\rho_{i}}$ in $\Cl(X_{\Sigma_{1}})_{\mathbf{Q}}$. We now show that $\varphi^{\ast}$ restrict to an isomorphism $\Pic(X_{\Sigma_{2}})_{\mathbf{Q}}\simeq\Pic(X_{\Sigma_{1}})_{\mathbf{Q}}$. 

For this it is enough to show that $\varphi^{\ast}(\Pic(X_{\Sigma_{2}})_{\mathbf{Q}})\subset\Pic(X_{\Sigma_{1}})_{\mathbf{Q}}$ because we may apply the result to the inverse of $\varphi^{\ast}$ and obtain another inclusion $(\varphi^{\ast})^{-1}(\Pic(X_{\Sigma_{1}})_{\mathbf{Q}})\subset \Pic(X_{\Sigma_{2}})_{\mathbf{Q}}$. We then have  $$\Pic(X_{\Sigma_{2}})_{\mathbf{Q}}\subset(\varphi^{\ast})^{-1}(\Pic(X_{\Sigma_{1}})_{\mathbf{Q}})\subset \Pic(X_{\Sigma_{2}})_{\mathbf{Q}}$$ so that $\Pic(X_{\Sigma_{2}})_{\mathbf{Q}})=(\varphi^{\ast})^{-1}(\Pic(X_{\Sigma_{1}})_{\mathbf{Q}}),$ which in turn implies $\varphi^{\ast}(\Pic(X_{\Sigma_{2}})_{\mathbf{Q}})=\Pic(X_{\Sigma_{1}})_{\mathbf{Q}}$.

Let $D=\sum_{i}q_{i}D_{\Psi(\rho_{i})}$ be a 
$\mathbf{Q}$-Cartier divisor on $X_{\Sigma_{2}}$. We need to construct an integer $K>0$
such that $K\cdot \varphi^{\ast}(D)$ is Cartier on $X_{\Sigma_{1}}$. For this, first let us choose an integer $K^{\prime}>0$ such that 
\begin{enumerate}
\item $K^{\prime}\cdot q_{\rho}\in\mathbf{Z}$ for all $\rho$, and
\item $K^{\prime}\cdot D$ is Cartier on $X_{\Sigma_{2}}$.
\end{enumerate}
Write $K^{\prime}\cdot D$ as $\sum_{i}a_{i}D_{\Psi(\rho_{i})}$. Then by \Cref{numampleness}, for every maximal cone $\sigma_{2}$ in $\Sigma_{2}$, there exists an $m_{\sigma_{2}}\in M_{2}$ such that $$\langle m_{\sigma_{2}},u_{\Psi(\rho_{i})}\rangle=-a_{i}\text{ for all }\Psi(\rho_{i})\in\sigma_{2}(1).$$ Now choose an integer $K^{\prime\prime}$ large enough so that $K^{\prime\prime}\cdot \Phi^{\ast}(m_{\sigma_{2}})\in M_{1}$ for all $\sigma_{2}\in\Sigma_{2}(n)$. We claim that $K^{\prime}\cdot K^{\prime\prime}\cdot \varphi^{\ast}(D)$ is Cartier on $X_{\Sigma_{1}}$.

For this, let $\sigma_{1}$ be a maximal cone in $\Sigma_{1}$. Then by \Cref{StructCompFan}, $\Phi(\sigma_{1})$ is a maximal cone in $\Sigma_{2}$. Let
$\langle-,-\rangle_{i}:(M_{i})_{\mathbf{Q}}\times (N_{i})_{\mathbf{Q}}\rightarrow\mathbf{Q}$ be the $\mathbf{Q}$-extension of the 
natural perfect pairing. Then by \Cref{adjunction}, we have 
$$\langle K^{\prime\prime}\cdot \Phi^{\ast}(m_{\Phi(\sigma_{1})}),u_{\rho_{i}}\rangle_{1}=\langle K^{\prime\prime}\cdot m_{\Phi(\sigma_{1})},u_{\Psi(\rho_{i})}\rangle_{2}=-K^{\prime\prime}\cdot a_{i}\text{ for all }\rho_{i}\in\sigma_{1}(1).$$ By \Cref{numampleness}, $K^{\prime\prime}\cdot K^{\prime}\cdot\varphi^{\ast}(D)$ is Cartier on $X_{\Sigma_{1}}$, proving the equality $\varphi^{\ast}(\Pic(X_{\Sigma_{2}})_{\mathbf{Q}})=\Pic(X_{\Sigma_{1}})_{\mathbf{Q}}.$
The statement on ample cones can be proved similarly using \Cref{numampleness} and \Cref{adjunction}. The proposition is proved.
\end{proof}

We are now ready to prove that the variation of stratifications from toric GIT is intrinsic to the primitive collections and relations among the ray generators of a fan. If two complete fans $\Sigma_{1},\Sigma_{2}$ are amply equivalent, then by \Cref{bad_locus_ample}, the unstable loci in their GIT quotient constructions can be identified. We will write $Z(\Sigma)$ as the common unstable locus. 

\begin{theorem}[VSIT Adjunction]\label{ToricVSIT}
Suppose $\Sigma_{1}$ and $\Sigma_{2}$ are two amply equivalent fans. Let $$G_{i}=\Hom_{\mathbf{Z}}(\Cl(X_{\Sigma_{i}}),\mathbf{C}^{\times})\text{ for }i=1,2$$ and $\varphi:\pmb{\Gamma}(G_{1})_{\mathbf{Q}}\simeq\pmb{\Gamma}(G_{2})_{\mathbf{Q}}$ be the isomorphism from \Cref{StructCompFan}.  Then the assignment of the strata  \begin{equation}\label{adjunction_stratification}\begin{tikzcd}\text{\huge $S$}^{\varphi(\chi_{D})}_{\lambda}\arrow[r,mapsto,shift left=1]&\text{\huge $S$}^{\chi_{D}}_{\varphi^{\ast}(\lambda)}\end{tikzcd}\end{equation} defined for all $D\in\Amp(X_{\Sigma_{2}})_{\mathbf{Q}}$ and $\lambda\in\pmb{\Gamma}(G_{1})_{\mathbf{Q}}$ is an equivalence of stratification of $Z(\Sigma)$. The theorem also holds for $D\in\Amp(X_{\Sigma_{2}})_{\mathbf{R}}$ and $\lambda\in\pmb{\Gamma}(G_{1})_{\mathbf{R}}$.
\end{theorem}

\begin{proof}
Recall that we had a diagram on short exact
sequences 
$$\begin{tikzcd}[sep = small]
0\arrow[r]&\pmb{\Gamma}(G_{1})_{\mathbf{Q}}\arrow[r]\arrow[d,"\varphi"]&\mathbf{Q}^{\Sigma(1)}\arrow[r]\arrow[d,equal]&(N_{1})_{\mathbf{Q}}\arrow[r]\arrow[d,"\Phi"]&0\\
0\arrow[r]&\pmb{\Gamma}(G_{2})_{\mathbf{Q}}\arrow[r]&\mathbf{Q}^{\Sigma(1)}\arrow[r]&(N_{2})_{\mathbf{Q}}\arrow[r]&0
\end{tikzcd}.$$
Let $||-||_{i}$ be the norm on $\pmb{\Gamma}(G_{i})_{\mathbf{R}}$. Recall that the norms on $\pmb{\Gamma}(G_{i})$ comes
from restricting the standard norm on
$\pmb{\Gamma}((C^{\times})^{\Sigma(1)}).$ Therefore, for every 
$\lambda\in\pmb{\Gamma}(G_{1})_{\mathbf{Q}}$, we have
$||\lambda||_{1}=||\varphi(\lambda)||_{2}$. Because of this and
\Cref{adjunction}, we get 
\begin{equation}\label{StratAdjunction}\frac{\langle\varphi^{\ast}(\chi_{D}),\lambda\rangle_{G_{1}}}{||\lambda||_{1}}=\frac{\langle\chi_{D},\varphi(\lambda)\rangle_{G_{2}}}{||\varphi(\lambda)||_{2}}\text{
    for all }\chi_{D}\in\Cl(X_{\Sigma_{2}})_{\mathbf{Q}} \text{ and
  } \lambda\in\pmb{\Gamma}(G_{1})_{\mathbf{Q}}\end{equation}
where
$\langle-,-\rangle_{G_{i}}:\pmb{\chi}(G_{i})_{\mathbf{Q}}\times\pmb{\Gamma}(G_{i})_{\mathbf{Q}}\rightarrow\mathbf{Q}$
denotes the $\mathbf{Q}$-extension of the natural pairing. 

Recall that for any $x\in Z(\Sigma)$ and $D\in \Amp(X_{\Sigma_{2}})_{\mathbf{R}}$, $\lambda^{D}_{x}$ is the vector in $\pmb{\Gamma}(G_{2})_\mathbf{R}$ that is $\chi_{D}$-adapted to $x$ (\Cref{one_PS_adapted_to_extended}). Also note that when $D$ is in $\Amp(X_{\Sigma_{2}})_{\mathbf{Q}}$, the vector $\lambda^{D}_{x}$ is in $\pmb{\Gamma}(G_{2})_{\mathbf{Q}}$.

\Cref{StratAdjunction} and the fact that $\varphi$ preserves norms imply that for any $x\in Z(\Sigma)$, $$\lambda=\lambda^{\varphi^{\ast}(D)}_{x}\Leftrightarrow\varphi(\lambda)=\lambda^{D}_{x}\text{ for all
}D\in\Amp(X_{\Sigma_{2}})_{\mathbf{Q}},\lambda\in\pmb{\Gamma}(G_{1})_{\mathbf{Q}}.$$ This in turn implies that the assignment (\ref{adjunction_stratification})  is a bijection of strata that preserves each stratum as sets. That assignment (\ref{adjunction_stratification}) preserves strict partial order follows from the fact that $\varphi$ preserves norms.  The proof easily extends over $\mathbf{R}$. The theorem is proved.
\end{proof}
\subsubsection{Topological properties of the strata}
In this section we supply some topological properties of the strata. We first note a property satisfied by the primitive collections of a fan.
\begin{proposition}\label{union_primitive_collection}
Let $\Sigma\subset N_{\mathbf{R}}$ be a complete fan. Then the union of primitive collections is $\Sigma(1).$
\end{proposition}

\begin{proof}
Let $\rho\in\Sigma(1)$ be a ray. We will construct a primitive collection that contains $\rho$. Since $\Sigma$ is complete, there exists a $\sigma\in\Sigma_{\text{max}}$ such that $-u_{\rho}$ is in $\sigma.$ Then since $\sigma$ is strongly convex, $\text{Cone }(u_{\rho})\cap\sigma=\{0\}.$ Hence $\tilde{\sigma}:=\text{Cone }(u_{\rho},\sigma)$ is a cone properly containing $\sigma$. This implies  $\tilde{\sigma}$ is not contained in $\sigma^{\prime}$ for any $\sigma^{\prime}\in\Sigma_{\text{max}}$. This implies $\{\rho\}\cup\sigma(1)$ is not contained in any $\sigma^{\prime}(1)$. We make take $S\subset\sigma(1)$ to be a subset minimal with respect to this property: $\{\rho\}\cup S$ is not contained in any $\sigma^{\prime}(1)$ for any $\sigma^{\prime}\in\Sigma_{\text{max}}$. It follows from the definition of primitive collections (\Cref{primitive_collection_def}) that $\{\rho\}\cup S$ is a primitive collection.
\end{proof}

Note that the analysis we did in \Cref{The structure of a stratum} can be extended over the ample cone. In particular, \Cref{strata_struct} can be applied to stratifications induced by $\mathbf{R}$-ample divisors to derive the 

\begin{theorem}\label{toric_quotient_strata}
Let $D$ be an $\mathbf{R}$-ample divisor of the projective toric variety $X_{\Sigma}$. Then for each $\lambda\in\Lambda^{D}$, the stratum 
$S^{\chi_{D}}_{\lambda}$ is smooth, irreducible and open inside the zariski closure
$\overline{S^{\chi_{D}}_{\lambda}}$. More precisely, $$S^{\chi_{D}}_{\lambda}= V(\{x_{\rho}|\rho\in
S\})-\cup_{R}V(\{x_{\rho}|\rho\in R\})$$ where $S$ is a subset of
$\Sigma(1)$ containing some primitive collection and $R$ runs
through some collection of subsets of $\Sigma(1)-S$.
\end{theorem}

\begin{proof}
By \Cref{strata_struct}, each $\chi_{D}$-stratum is of the form $$S^{\chi_{D}}_{\lambda}=V(\{x_{\rho}|\rho\in S\})\cap\big(\cup_{j}(D(A_{j})\big)$$ where each $A_{j}$ is contained in the complement $\Sigma(1)-S$. Since $\overline{S^{\chi_{D}}_{\lambda}}=V(\{x_{\rho}|\rho\in S\})\subset Z(\Sigma)$ is irreducible, by \Cref{bad_locus_ample}, there is a primitive collection $C$ such that $V(\{x_{\rho}|\rho\in S\})\subset V\{(x_{\rho}|\rho\in C\})$. Equivalently this means $C\subset S$. The theorem is proved.
\end{proof}

\begin{corollary}\label{closure_strata_primitive_collection}
Let $D$ be an $\mathbf{R}$-ample
divisor of the projective toric variety $X_{\Sigma}$. Then for every primitive collection $C\subset\Sigma(1)$ , there is a unique
$\lambda_{C}\in\Lambda^{D}$ such that $\overline{S^{\chi_{D}}_{\lambda_{C}}}=
V(\{x_{\rho}|\rho\in C\}).$
\end{corollary}

\begin{proof}
By \Cref{proposition5.1.6}, each primitive collection $C$
corresponds to the irreducible component $V(\{x_{\rho}|\rho\in C\})$ of
$Z(\Sigma)$. We also have $$Z(\Sigma)=\bigcup_{\lambda\in
  \Lambda^{D}}S^{\chi_{D}}_{\lambda}=\bigcup_{\lambda\in
  \Lambda^{D}}\overline{S^{\chi_{D}}_{\lambda}}.$$ Since each $\overline{S^{\chi_{D}}_{\lambda}}$ is
irreducible, the maximal $\overline{S^{\chi_{D}}_{\lambda}}$'s correspond
to the irreducible components of $Z(\Sigma)$. Namely, for each primitive collection $C$, there is a $\lambda_{C}\in\Lambda^{D}$ such that $$\overline{S^{\chi_{D}}_{\lambda_{C}}}=V(\{x_{\rho}\mid\rho\in C\}).$$ The uniqueness of
$\lambda_{C}$ follows from the fact that the boundary of a stratum
only contains higher strata.
\end{proof}

We can also derive the
\begin{corollary}\label{toric_strata_intersection}
Let $D$ be an $\mathbf{R}$-ample divisor of the projective toric variety $X_{\Sigma}$. For every primitive collection $C\subset\Sigma(1)$, let $\lambda_{C}$ be as in \Cref{closure_strata_primitive_collection}. We have 
$$\mathbf{0}=\bigcap_{\lambda\in\Lambda^{D}}\overline{S^{\chi_{D}}_{\lambda}}=\bigcap_{\text{primitive collection }C}\overline{S^{\chi_{D}}_{\lambda_{C}}}$$
where $\mathbf{0}$ stands for the origin in $\mathbf{C}^{\Sigma(1)}$.
\end{corollary}

\begin{proof}
By \Cref{union_primitive_collection}, the intersection
$\bigcap_{C}V(\{x_{\rho}|\rho\in C\})$ taken over all primitive collections
$C$ of $\Sigma(1)$ is the origin. By 
\Cref{closure_strata_primitive_collection}, we have
$\bigcap_{C}\overline{S^{\chi_{D}}_{\lambda_{C}}}=\bigcap_{C}V(\{x_{\rho}|\rho\in
C\})$. Moreover, each $\chi_{D}$-stratum's closure  
$\overline{S^{\chi_{D}}_{\lambda}}$ is a linear subspace so 
$\mathbf{0}\in\bigcap_{\lambda\in\Lambda^{D}}\overline{S^{\chi_{D}}_{\lambda}}$. In sum, we then have 
$$\mathbf{0}\in\bigcap_{\lambda\in\Lambda^{D}}\overline{S^{\chi_{D}}_{\lambda}}\subset\bigcap_{C}\overline{S^{\chi_{D}}_{\lambda_{C}}}=\bigcap_{C}V(\{x_{\rho}|\rho\in C\})=\mathbf{0}.$$ The corollary is proved.
\end{proof}

\section{The computer program}\label{The computer program}
Given a projective toric variety, the computer program works out the
\begin{enumerate}
    \item ample cone (\Cref{Computing_the_Picard_group}, \Cref{Computing_the_ample_cone}),
    \item potential real one parameter subgroups that index the unstable strata in its GIT quotient construction (\Cref{Potential_one_parameter_subgroups}),
    \item walls in the ample cone (\Cref{Enumerating_walls}),
    \item stratification induced by a particular ample divisor (\Cref{Compute_stratification_with_respect_to_ample_divisors}), and
    \item visualization of the wall and semi-chamber decomposition of the ample cone if the ample cone is less than three dimensional (\Cref{Visualizing_the_ample_cone_decomposition}).
\end{enumerate}

The computer program is solely sageMath where a number of packages for toric varieties are already available, such as constructing fans and enumerating their primitive collections. Although there are computer programs in SageMath and Macaulay2 that compute ample cones, the output is difficult to utilize internally in our program. Hence it became more desirable that we design the ample cone computation by ourselves. The first challenge is to figure out the basis of the Picard group for a toric variety. Again the output from SageMath and Maculay2 is not easy to dispose of within our program. We therefore formulate in \Cref{Computing_the_Picard_group} a particular class of projective toric varieties whose basis of the Picard group is fairly easy to describe and with which the computer program performs correctly and self-containedly. 

However, the ampleness condition and the equations of walls can be computed in terms of divisors of the form $D=\sum_{\rho}a_{\rho}D_{\rho}$ without referencing to any basis of the Picard group (\Cref{numampleness}, \Cref{struct_walls}). Because of this, our program outputs ampleness conditions and walls  as inequalities and equations respectively in terms of $\{a_{\rho}\}_{\rho\in\Sigma(1)}$. An advantage to this is that users have the flexibility to specialize the outputs to any basis they have at their hands.

We will walk through the computer program with the easiest non-trivial example given by the projective toric variety $\mathbf{P}^{1}_{\mathbf{C}}\times\mathbf{P}^{1}_{\mathbf{C}}$. It comes from the fan $\Sigma_{ex}$ in $\mathbf{R}^{2}$ consisting of the four rays 
\begin{enumerate}
    \item $\rho_{0}=\text{Cone }(e_{1})$,
    \item $\rho_{1}=\text{Cone }(-e_{1})$,
    \item $\rho_{2}=\text{Cone }(e_{2})$,
    \item $\rho_{3}=\text{Cone }(-e_{2})$,
\end{enumerate}
and the four maximal cones 
\begin{enumerate}
    \item $\text{Cone }(\rho_{0},\rho_{2})$,
    \item $\text{Cone }(\rho_{1},\rho_{2})$,
    \item $\text{Cone }(\rho_{1},\rho_{3})$, 
    \item $\text{Cone }(\rho_{0},\rho_{3}).$
\end{enumerate}  
 
At the start of the program, the user is asked to enter a fan as a
list of the coordinates of rays, followed by list of ambient indicies for each
maximal cone in the fan. In the case of $\Sigma_{ex}$, the list of rays is $(1,0),(-1,0),(0,1),(0,-1)$ and the list of ambient indicies is $(0,2),(1,2),(1,3),(0,3).$ 

\subsection{Computing the Picard group}\label{Computing_the_Picard_group}
We consider projective toric varieties $X_{\Sigma}$ whose Picard groups satisfy \begin{equation}\label{divisoral_simplicial}\Pic(X_{\Sigma})_{\mathbf{R}}\simeq\bigoplus_{\rho\notin\sigma(1)}\mathbf{R}\cdot D_{\rho}\text{ for some maximal cone }\sigma\in\Sigma.\end{equation} We will see in \Cref{CDiv_basis} that condition (\ref{divisoral_simplicial}) includes all simplicial projective toric varieties. 

However, the computer program does not detect if condition (\ref{divisoral_simplicial}) holds when a fan is given. The output of ample cone may not be correct if condition (\ref{divisoral_simplicial}) fails. Users may use a characterization given as  \Cref{divisorial_simplicial_characterization} and \Cref{numampleness} to determine if their projective toric variety satisfies condition (\ref{divisoral_simplicial}). First, we need
\begin{proposition}\label{Pic_sublattice}
Let $\Sigma$ be a fan with a full dimensional cone $\sigma_{0}$. Then the natural map
$\CDiv_{T_{N}}(X_{\Sigma})\rightarrow\Pic(X_{\Sigma})$ induces an
isomorphism 
$$\varphi:\{D=\sum_{\rho}a_{\rho}D_{\rho}\in\CDiv_{T_{N}}(X_{\Sigma})\mid
a_{\rho}=0\text{ for all }\rho\in\sigma_{0}(1)\}\simeq\Pic(X_{\Sigma}).$$
\end{proposition}
\begin{proof}
Let $D=\sum_{\rho}a_{\rho}D_{\rho}\in\Pic(X_{\Sigma})$. Then there exists $m_{\sigma_{0}}\in M$ as in \Cref{numampleness}. Then $D\sim D-\text{div}(\chi^{m_{\sigma_{0}}})$
in $\Pic(X_{\Sigma})$ and the later is in the image of $\varphi$. This
shows $\varphi$ is surjective. To prove that $\varphi$ is injective,
suppose $\varphi(D)=0$ where
$D=\sum_{\rho\notin\sigma_{0}(1)}a_{\rho}D_{\rho}.$ Then by the short
exact sequence (\ref{sesCDiv}), there exists an $m\in M$ such that
$D=\text{div}(\chi^{m})$. In particular, $\langle m,u_{\rho}\rangle=0$
for all $\rho\in\sigma_{0}(1)$. Since $\sigma_{0}$ is full
dimensional, $m=0$, This proves injectivity.
\end{proof}
Here is a characterization about when (\ref{divisoral_simplicial}) holds:
\begin{corollary}\label{divisorial_simplicial_characterization}
Let $X_{\Sigma}$ be a toric variety and let $\sigma_{0}\in\Sigma$ be a full dimensional cone. Then the following statements are equivalent:
\begin{enumerate}
    \item The inclusion $\Pic(X_{\Sigma})\subset\mathbf{Z}^{\Sigma(1)}$ induces $\Pic(X_{\Sigma})_{\mathbf{R}}\simeq\bigoplus_{\rho\notin\sigma_{0}(1)}\mathbf{R}\cdot D_{\rho}$,
    \item the inclusion $\Pic(X_{\Sigma})\subset\mathbf{Z}^{\Sigma(1)}$ induces $\Pic(X_{\Sigma})_{\mathbf{Q}}\simeq\bigoplus_{\rho\notin\sigma_{0}(1)}\mathbf{Q}\cdot D_{\rho}$, and 
    \item $D_{\rho}$ is $\mathbf{Q}$-Cartier for all $\rho\notin\sigma_{0}(1)$.
\end{enumerate}
\end{corollary}
\begin{proof}
There is an inclusion $\Pic(X_{\Sigma})\subset\mathbf{Z}^{\Sigma(1)}$ due to \Cref{Pic_sublattice}. Statement (1) and (2) are equivalent because $\mathbf{R}$ is faithfully flat over $\mathbf{Q}$. The equivalence of (2) and (3) follows from \Cref{Pic_sublattice}.
\end{proof}

To see why the class of toric varieties that satisfies (\ref{divisoral_simplicial}) includes projective simplicial toric varieties, we need the 
\begin{proposition}\label{simplicial_picard_group}
Let $X_{\Sigma}$ be a toric variety.
The following statements are equivalent:
\begin{enumerate}
\item $X_{\Sigma}$ is simplicial.
\item Every $D\in\Cl(X_{\Sigma})$ is $\mathbf{Q}$-Cartier. Namely, for every $D\in\Cl(X_{\Sigma})$, there
  exists an integer $K$ such that $K\cdot D\in\Pic(X_{\Sigma})$.
\item  $\Pic(X_{\Sigma})$ is of finite index in
  $\Cl(X_{\Sigma})$. Namely, there exists an integer $K$ such that $K\cdot
  D\in\Pic(X_{\Sigma})$ for all $D\in\Cl(X_{\Sigma})$.  
\item The inclusion $\Pic(X_{\Sigma})\subset\Cl(X_{\Sigma})$ induces $\Pic(X_{\Sigma})_{\mathbf{Q}}\simeq\Cl(X_{\Sigma})_{\mathbf{Q}}$.
\end{enumerate}
\end{proposition}

\begin{proof}
Statements (2),(3),(4) are equivalent because $\Cl(X_{\Sigma})$ is
finitely generated and tensoring $\mathbf{Q}$ is flat and kills
torsion. The only non-trivial part is to prove (2) is equivalent to
(1). For this, let $\sigma$ be a cone in $\Sigma$. Viewing $\sigma$
itself as a fan, we have the following short exact sequence 
$$\begin{tikzcd}[sep=small]
M\arrow[r]&\mathbf{Z}^{\sigma(1)}\arrow[r]&\Cl(X_{\sigma})\arrow[r]&0
\end{tikzcd}$$
whose dual is $$\begin{tikzcd}[sep=small]
0\arrow[r]&\Cl(X_{\sigma})^{\vee}\arrow[r]&\mathbf{Z}^{\sigma(1)}\arrow[r,"\psi_{\sigma}"]&N
\end{tikzcd}$$ by (\ref{one_PS_SES}). We see that $X_{\Sigma}$ is
simplicial if and only if $\psi_{\sigma}$ is injective. This is
equivalent to requiring that $\Cl(X_{\sigma})^{\vee}=0$, which in turn
is equivalent to $\Cl(X_{\sigma})$ is torsion. Because there are only finitely many $\sigma$ and
$\{X_{\sigma}\}_{\sigma\in\Sigma}$ covers $X_{\Sigma}$, (1) implies
(2). For the converse, it is enough to show that $\Cl(X_{\Sigma})$ is
torsion for all $\sigma\in \Sigma.$ Since the restriction
$\Cl(X_{\Sigma})\rightarrow\Cl(X_{\sigma})$ is surjective, we may
write a divisor in $\Cl(X_{\sigma})$ as $D|_{X_{\sigma}}$ for some
divisor $D$ on
$X_{\Sigma}$. By assumption, there exists a $K>0$ such that $K\cdot D$
is Cartier. Then $K\cdot D|_{X_{\sigma}}$ is also Cartier. The result
follows from the fact that affine normal toric varieties does not
have non-trivial Picard group. See proposition 4.2.2 in \cite{MR2810322}. 
\end{proof}

\begin{corollary}\label{CDiv_basis}
Let $\Sigma$ be simplicial and $\sigma_{0}\in\Sigma$ be a full dimensional cone. We have that  
$\Pic(X_{\Sigma})_{\mathbf{Q}}\simeq\Cl(X_{\Sigma})_{\mathbf{Q}}$ has
$\mathbf{Q}$-basis $\{D_{\rho}\}_{\rho\notin\sigma_{0}(1)}$. If
$\Sigma$ is smooth, $\Pic(X_{\Sigma})\simeq\Cl(X_{\Sigma})$ is free
with $\mathbf{Z}$-basis $\{D_{\rho}\}_{\rho\notin\sigma_{0}(1)}$.
\end{corollary}

\begin{proof}
By \Cref{simplicial_picard_group}, every $D_{\rho}$ is $\mathbf{Q}$-Cartier. Using statement (3) from \Cref{divisorial_simplicial_characterization}, we get the result.
\end{proof}
Hence we get that the class of toric varieties that satisfies (\ref{divisoral_simplicial}) includes simplicial projective toric varieties. For the fan $\Sigma_{\text{ex}}$, condition (\ref{divisoral_simplicial}) is satisfied say for the cone $\text{Cone}(\rho_{1},\rho_{3})$. For brevity, for a ray $\rho_{i}$, instead of writing $D_{\rho_{i}}$ (resp. $a_{\rho_{i}}$), we write $D_{i}$ (resp. $a_{i}$). With these notations,  $\Pic(X_{\Sigma_{\text{ex}}})_{\mathbf{R}}$ has basis $D_{0},D_{2}.$ In fact, $D_{0}=\mathscr{O}(1,0)$ and $D_{2}=\mathscr{O}(0,1)$ on $\mathbf{P}^{1}_{\mathbf{C}}\times\mathbf{P}^{1}_{\mathbf{C}}$ so that $a_{0}D_{0}+a_{2}D_{2}=\mathscr{O}(a_{0},a_{2}).$

\subsection{Computing the ample cone}\label{Computing_the_ample_cone}
Recall that $\Sigma_{\text{max}}$ means the collection of maximal cones in $\Sigma$. 
The logic behind the computation of ample cones is mainly the following:
\begin{proposition}
Let $\Sigma$ be a complete fan and $D=\sum_{\rho}a_{\rho}D_{\rho}$ be
a Cartier divisor. For every $\sigma\in\Sigma_{\text{max}}$ and every
$\rho^{\prime}\notin\sigma(1)$, fix a relation
$\sum_{\rho\in\sigma(1)}b_{\rho}u_{\rho}=u_{\rho^{\prime}}$ with
$b_{\rho}\in\mathbf{Q}$ for all $\rho\in\sigma(1)$. Then $D$ is ample
if and only if $\sum_{\rho\in\sigma(1)}b_{\rho}a_{\rho}<a_{\rho^{\prime}}$ for
all $\sigma\in\Sigma_{\text{max}}$ and $\rho^{\prime}\notin\sigma(1)$.
\end{proposition}

\begin{proof}
Let $\Sigma$ be a fan in $N_{\mathbf{R}}$ and $M$ be the dual lattice
of $M$. Since $\Sigma$ is complete, all maximal cones in $\Sigma$ are
full dimensional. Hence we can
solve for rational relations for every ray outside of a maximal
cone. Fix a relation (not necessarily unique)
$\sum_{\rho\in\sigma(1)}b_{\rho}u_{\rho}=u_{\rho^{\prime}}$ for each
maximal cone $\sigma$ and each $\rho^{\prime}\notin\sigma(1)$. Since $D$ is Cartier, for every $\sigma\in\Sigma_{\text{max}}$, there
exists an unique $m_{\sigma}\in M$ such that $\langle
m_{\sigma},u_{\rho}\rangle=-a_{\rho}$ for all $\rho\in\sigma(1)$ by 
\Cref{numampleness}. In order for $D$ to be ample, we also need
$\langle m_{\sigma},u_{\rho^{\prime}}\rangle>-a_{\rho^{\prime}}$ for all
$\rho^{\prime}\notin\sigma(1)$. But $$\langle m_{\sigma},
u_{\rho^{\prime}}\rangle=\sum_{\rho\in\sigma(1)}\langle
m_{\sigma},b_{\rho}u_{\rho}\rangle=-\sum_{\rho\in\sigma(1)}a_{\rho}b_{\rho}.$$
The result follows.
\end{proof}
Therefore, computing the ample cone breaks down into the following steps:
\begin{enumerate}
    \item Loop through each maximal cone
        $\sigma$ in the fan and solve relations
        $u_{\rho^{\prime}}=\sum_{\rho\in\sigma(1)}b_{\rho}u_{\rho}$ for each
        $\rho^{\prime}\notin \sigma(1)$.
    \item Produce the list of normal vectors of the inequalities $$\{a_{\rho^{\prime}}-\sum_{\rho\in\sigma(1)}b_{\rho}a_{\rho}>0\bigm|\sigma\in\Sigma_{\text{max}},\rho^{\prime}\notin\sigma(1)\}.$$ Call this list $L_{\text{gen}}$.
\end{enumerate}
Now we know when a Cartier divisor $\sum_{\rho\in\Sigma(1)}a_{\rho}D_{\rho}$ is ample. Therefore the last step is to,
\begin{enumerate}[resume]
    \item Specialize $L_{\text{gen}}$ to $\Pic(X_{\Sigma})_{\mathbf{R}}$. Call this list $L_{\text{red}}$.
\end{enumerate}

The $gen$ in the list produced in step (2) stands for generic as we look at divisors $\sum_{\rho}a_{\rho}D_{\rho}$ without restricting to any basis of $\Pic(X_{\Sigma})_{\mathbf{R}}.$ To go from step (2) to step(3), if condition (\ref{divisoral_simplicial}) is satisfied for a maximal cone $\sigma$, then it is simply setting the $\rho$-th component of vectors in $L_{\text{gen}}$ to be 0 for every $\rho\in\sigma(1)$. In the program we simply remove the $\rho$-th component of the vectors in $L_{\text{gen}}$ for $\rho\in\sigma(1)$. We now have a reduced list of normal vectors. Call it $L_{\text{red}}$. The vectors in $L_{\text{red}}$ can be either thought of as coefficients of strict inequalities for a divisor $\sum_{\rho\notin\sigma(1)}a_{\rho}D_{\rho}$ to be ample, or as the normal vectors of the supporting hyperplanes for the nef cone  $\Nef(X_{\Sigma})$.\\
\indent 
Take $\Sigma_{ex}$ for example, the fan is smooth so all divisors are Cartier. The only relations we have are $$u_{\rho_{0}}+u_{\rho_{1}}=0\text{ and }u_{\rho_{2}}+u_{\rho_{3}}=0.$$ Hence the divisor $\sum_{i=0}^{3}a_{i}D_{i}$ is ample if and only if $$a_{0}+a_{1}>0\text{ and }a_{2}+a_{3}>0.$$ The output of $L_{\text{gen}}$ is the list $(1,1,0,0),(0,0,1,1)$ of normal vectors. We also know that $D_{0}=\mathscr{O}(1,0)$ and $D_{2}=\mathscr{O}(0,1)$ form a basis for $\Pic(X_{\Sigma})$ where (\ref{divisoral_simplicial}) is satisfied for the cone $\sigma =\text{Cone }(\rho_{1},\rho_{3})$. In this case $L_{\text{red}}=(1,0),(0,1)$, corresponding to the fact that $a_{0} D_{0}+a_{2} D_{2}$ is ample if and only if $a_{0},a_{2}>0$. 
\subsection{Potential real one parameter subgroups}\label{Potential_one_parameter_subgroups}
This is the core part of the program as its output would be used internally to compute the stratification induced by an ample divisor, to enumerate walls, and to plot walls in the ample cone. The reader may want to review the notations introduced in \Cref{Notations and conventions_Toric_VSIT_}.

Specifically, this part of the program computes $-\Proj_{W_{Z}}\chi^{\ast}_{D}\in\pmb{\Gamma}(G)_{\mathbf{R}}$ for any subset $Z\in\mathcal{L}$ and for any $D=\sum_{\rho}a_{\rho}D_{\rho}$. The adjective potential comes from the fact we deduced in \Cref{Extensions_to_the_ample_cone} that the real one parameter subgroup in $\pmb{\Gamma}(G)_{\mathbf{R}}$ that is $\chi_{D}$-adapted to $L(S)$ for an $S\in\mathcal{L}$ is of the form $-\Proj_{W_{Z}}\chi^{\ast}_{D}$ for some subset $Z$ of $S$. We note that 

\begin{remark}\label{remark_subspaces_for_worst_one_PS}
The description of the subspace $W_{Z}\subset\pmb{\Gamma}(G)_{\mathbf{R}}$ is rather simple. Viewing $\pmb{\Gamma}(G)_{\mathbf{R}}$ as a subspace of $\mathbf{R}^{\Sigma(1)}$ via the inclusion from (\ref{one_PS_SES}), we have that a point $v\in\mathbf{R}^{\Sigma(1)}$ is in $W_{Z}$ if and only if $v\in\pmb{\Gamma}(G)_{\mathbf{R}}$ and $v_{\rho}=0$ for all $\rho\in Z$.\end{remark} 
Computing potential real one parameter subgroups breaks down roughly into the following steps:
\begin{enumerate}
    \item Enumerate primitive collections of the fan $\Sigma$, and compute the set $\mathcal{L}$.
    \item Compute a $\mathbf{Q}$-basis for $W_{Z}$ for each $Z\in\mathcal{L}$.
    \item For each $Z\in\mathcal{L}$, solve \Cref{potential_one_PS_matrix_eq} to obtain $-\Proj_{W_{Z}}\chi^{\ast}_{D}$.
\end{enumerate}

Each output $-\Proj_{W_{Z}}\chi^{\ast}_{D}$ is a vector in $\mathbf{R}^{\Sigma(1)}$ whose components are $\mathbf{Q}$-combination of the $a_{\rho}$'s. The user may specialize the results to the ample cone. We now explain how each step is carried out. 

For step (1), the primitive collections can be obtained by the .primitive\_collections() method applied to the fan. The set $\mathcal{L}$ being the union of the power set of $\Sigma(1)-C$ over all primitive collections $C$, can be obtained by elementary set operations in SageMath.\\
\indent 
For step (2), note that the $\mathbf{Q}$-extension of sequence (\ref{sesWDiv}) is 
\begin{equation}\label{Q_extension_SES_divisor_class_group}\begin{tikzcd}[sep = small]
0\arrow[r]&\mathbf{Q}^{\dim \Sigma}\arrow[r,"B"]&\mathbf{Q}^{\Sigma(1)}\arrow[r,"B^{\perp}"]&\Cl(X_{\Sigma})_{\mathbf{Q}}\arrow[r]&0
\end{tikzcd}\end{equation} where $B$ is a matrix whose $\rho$-th row is the coordinate of the ray generator $u_{\rho}$ in $\mathbf{Z}^{\Sigma(1)}$. The dual of the above sequence is 
\begin{equation} \label{one_PS_SES_rational}\begin{tikzcd}[sep=small]
0\arrow[r]&\pmb{\Gamma}(G)_{\mathbf{Q}}\arrow[r,"(B^{\perp})^{t}"]&[+15pt]\mathbf{Q}^{\Sigma(1)}\arrow[r,"B^{t}"]&\mathbf{Q}^{\dim\Sigma}\arrow[r] &0
\end{tikzcd}.\end{equation} Hence $\pmb{\Gamma}(G)_{\mathbf{R}}$ is the space of vectors $u\in\mathbf{R}^{\Sigma(1)}$ such that $B^{t}\cdot u=0.$ Let us augment the matrix $B$ by inserting column vectors $e_{\rho}$ for $\rho\in Z$. Call the augmented matrix $B_{Z}$. Then by \Cref{remark_subspaces_for_worst_one_PS}, $W_{Z}$ is the space of vectors $v\in\mathbf{R}^{\Sigma(1)}$ such that $(B_{Z})^{t}\cdot v=0$. The .right\_kernel() method when applied to $(B_{Z})^{t}$, yields a matrix $C$ whose columns consist of a $\mathbf{Q}$-basis of $W_{Z}$, finishing step (2). The matrix $C$ is called the right kernel of $(B_{Z})^{t}$ simply because $C$ consists of vectors $v$ such that $(B_{Z})^{t}\cdot v=0$.
\\
\indent
For step (3), suppose $W_{Z}$ has the $\mathbf{Q}$-basis $\{\lambda_{1},\ldots,\lambda_{q}\}\subset\mathbf{Q}^{\Sigma(1)}$. Write $\Proj_{W_{Z}}\chi^{\ast}_{D}=\sum_{j=1}^{q}b_{j}\lambda_{j}$. It then comes down to solve \Cref{potential_one_PS_matrix_eq} for each $b_{j}$. This is no difficult task for SageMath, finishing step (3). We now explain the data structure that is used to store the list of potential real one parameter subgroups as the list is used extensively in later parts of the program.
\subsubsection{Data structure of potential real one parmaeter subgroups}
The list of potential real one parameter subgroups has the following data structure. Each entry in the list is a list $$[v,l,||v||^2]$$ where $v$ is a column vector in $\mathbf{Q}^{\Sigma(1)}$ and $l$ is a list of sets in $\mathcal{L}$ such that 
$v=-\Proj_{W_{Z}}\chi^{\ast}_{D}$ for all $Z\in l$. The reason why we included the squares $||v||^{2}$ in the data is that they will be used to compute stratifications and to determine the strict partial order between strata (See \Cref{Compute_stratification_with_respect_to_ample_divisors}).\\
\indent 
Now take $\Sigma_{ex}$ to demonstrate the ideas. For $\Sigma_{ex}$, the primitive collections are $[0,1]$ and $[2,3]$. Hence, $$\mathcal{L}= [[], [2], [3], [2, 3], [0], [1], [0, 1]]$$ where the empty list $[]$ corresponds to the empty set $\emptyset$. For brevity, for a subset $Z\in\mathcal{L}$, say $Z=[0,1]$, instead of writing $W_{[0,1]}$ (resp. $B_{[0,1]}$), we write $W_{01}$ (resp. $B_{01}$).

In this case the matrix $B=\begin{bmatrix}
              1 & 0\\
              -1 & 0\\
              0 & 1\\
              0 & -1
\end{bmatrix}$ in sequence (\ref{Q_extension_SES_divisor_class_group}). Taking the subset $Z=[0,1]\in\mathcal{L}$ for example, we have $$B_{01}=\begin{bmatrix}
              1 & 0 & 1 & 0\\
              -1 & 0 & 0 & 1\\
              0 & 1 & 0 & 0\\
              0 & -1 & 0 & 0
\end{bmatrix}.$$ Note that the sum of the first and the fourth column of $M_{01}$ yields the third column. This implies $(B_{01})^{t}, (B_{0})^{t}$, and $(B_{1})^{t}$ have the same right kernel, which is computed as $\begin{bmatrix}
              0 \\ 0 \\ 1 \\ 1
\end{bmatrix}.$ Solving \Cref{potential_one_PS_matrix_eq}, we get $$\Proj_{W_{01}}\chi^{\ast}=\Proj_{W_{0}}\chi^{\ast}_{D}=\Proj_{W_{1}}\chi^{\ast}_{D}=\begin{bmatrix}
              0\\
              0\\
-1/2*a2 - 1/2*a3\\
-1/2*a2 - 1/2*a3\end{bmatrix}.$$ Here is the full output of potential real one parameter subgroups for $X_{\Sigma_{ex}}$.
\[\begin{array}{lll}\begin{array}{r}
\begin{bmatrix}
              0\\
              0\\
-1/2*a2 - 1/2*a3\\
-1/2*a2 - 1/2*a3\end{bmatrix}\\
\text{ }\end{array}, & [[0], [1], [0, 1]],& 1/2*a2^2 + a2*a3 + 1/2*a3^2\\
\begin{array}{r}
\begin{bmatrix}
-1/2*a0 - 1/2*a1\\
-1/2*a0 - 1/2*a1\\
              0\\
              0\end{bmatrix}\\\text{ }\end{array}, & [[2], [3], [2, 3]],& 1/2*a0^2 + a0*a1 + 1/2*a1^2\\
\begin{array}{r}
\begin{bmatrix}
-1/2*a0 - 1/2*a1\\
-1/2*a0 - 1/2*a1\\
-1/2*a2 - 1/2*a3\\
-1/2*a2 - 1/2*a3\end{bmatrix}\end{array}, & [[]], &
\begin{array}{l}1/2*a0^2 + a0*a1 + 1/2*a1^2 \\+ 1/2*a2^2 + a2*a3\\ + 1/2*a3^2\end{array}\end{array}\]
The first entry of the list says that $\Proj_{W_{0}}\chi^{\ast}_{D}=\Proj_{W_{1}}\chi^{\ast}_{D}=\Proj_{W_{01}}\chi^{\ast}_{D}$ for all $D\in\mathbf{R}^{\Sigma(1)}$, as was discussed. 

Let us specialize the results to the ample cone $\Amp(X_{\Sigma_{\text{ex}}})_{\mathbf{R}}$. 
Since condition (\ref{divisoral_simplicial}) is satisfied for the cone $\text{Cone}(\rho_{1},\rho_{3})$, we can first specialize the results to $\Pic(X_{\Sigma_{\text{ex}}})_{\mathbf{R}}$ by setting $a_{1}=a_{3}=0$. Here is the resulting list:  \begin{equation}\label{P1P1_reduced_potential_worst_one_ps}\begin{array}{lll}\begin{array}{r}
\begin{bmatrix}
              0\\
              0\\
-1/2*a2 \\
-1/2*a2 \end{bmatrix}\\
\text{ }\end{array}, & [[0], [1], [0, 1]],& 1/2*a2^2\\
\begin{array}{r}
\begin{bmatrix}
-1/2*a0 \\
-1/2*a0 \\
              0\\
              0\end{bmatrix}\\\text{ }\end{array}, & [[2], [3], [2, 3]],& 1/2*a0^2\\
\begin{array}{r}
\begin{bmatrix}
-1/2*a0 \\
-1/2*a0 \\
-1/2*a2 \\
-1/2*a2 \end{bmatrix}\end{array}, & [[]], &
1/2*a0^2 + 1/2*a2^2.\end{array}\end{equation}

Finally, as was discussed earlier, $D=a_{0}D_{0}+a_{2}D_{2}$ is ample if and only if  $a_{0},a_{2}>0$. This implies $-\Proj_{W_{0}}\chi^{\ast}_{D}$ (and therefore $-\Proj_{W_{1}}\chi^{\ast}_{D},-\Proj_{W_{01}}\chi^{\ast}_{D})$ is on the ray $\mathbf{R}_{>0}\cdot (0,0,-1-1)$ whenever $D$ is ample. Similarly, $-\Proj_{W_{2}}\chi^{\ast}_{D}$ (and therefore $-\Proj_{W_{3}}\chi^{\ast}_{D},-\Proj_{W_{23}}\chi^{\ast}_{D}$) is on the ray $\mathbf{R}_{>0}\cdot(-1,-1,0,0)$ whenever $D$ is ample. Hence we  know concretely what indivisible one parameter subgroup is on $\mathbf{R}_{>0}\cdot(-\Proj_{W_{0}}\chi^{\ast}_{D})$ or $\mathbf{R}_{>0}\cdot(-\Proj_{W_{2}}\chi^{\ast}_{D})$ for all ample $D$. In general, it is not possible to write the indivisible one parameter subgroup on the ray of the projection $-\Proj_{W_{Z}}\chi^{\ast}_{D}$ for any $Z\in\mathcal{L}$ abstractly in terms of $D$.
\subsection{Enumerating walls}\label{Enumerating_walls}
\subsubsection{Type one walls}
Recall a type one wall is a non-empty proper collection of $D\in\Amp(X_{\Sigma})_{\mathbf{R}}$ such that $\Proj_{W_{Z}}\chi^{\ast}_{D}=\Proj_{W_{Z\cup\{\rho\}}}\chi^{\ast}_{D}$ where $Z,\{\rho\}\in\mathcal{L}$ and $W_{Z\cup\{\rho\}}$ is a codimension one subspace of $W_{Z}.$ 
The main reasoning behind type one wall computation is the following elementary observation:
\begin{proposition}\label{type_one_critical_subset_reasoning}
Let $V$ be a finite dimensional real vector space and $V^{\vee}$ be its dual. Let $\sigma\subset V$ be a full dimensional polyhedral cone and $f\in V^{\vee}$. Then the following statements are equivalent:
\begin{enumerate}
    \item The hyperplane $H_{f}=\{v\in V\mid f(v)=0\}$ intersects the interior of $\sigma$,
    \item Neither $f$ nor $-f$ is in the dual cone $\sigma^{\vee}$.
\end{enumerate}
\end{proposition}
Enumerating type one wall breaks down into the following steps:
\begin{enumerate}
    \item Find $Z,\{\rho\}$ of sets in $\mathcal{L}$ where $W_{Z\cup\{\rho\}}$ is a codimension one subspace of $W_{Z}$.
    \item Let $\nu_{\rho}:\Pic(X_{\Sigma})_{\mathbf{R}}\rightarrow\mathbf{R}$ be defined by $$D\mapsto\langle\chi_{D_{\rho}},\Proj_{W_{Z}}\chi^{\ast}_{D}\rangle.$$ We then apply \Cref{type_one_critical_subset_reasoning} to the vector space $\Pic(X_{\Sigma})_{\mathbf{R}}$, the cone $\Nef(X_{\Sigma})$, and the function $\nu_{\rho}$ to test if $\ker \nu_{\rho}\cap\Amp(X_{\Sigma})_{\mathbf{R}}=\emptyset$. If not, then $\ker \nu_{\rho}\cap\Amp(X_{\Sigma})_{\mathbf{R}}$ is a type one wall.
\end{enumerate}

For step (1), in the program each $W_{Z}$ is represented as the column space of the right kernel of $B_{Z}^{t}$ (See \Cref{Potential_one_parameter_subgroups}). SageMath can count the dimensions of $W_{Z}$ and $W_{Z\cup\{\rho\}}$ to finish step (1). \\
\indent
For step (2), if $Z,\{\rho\}$ is a pair obtained in step (1), the function $\nu_{\rho}$ is described by the $\rho$-th component of $\Proj_{W_{Z}}\chi^{\ast}_{D}$ stored in the list of potential real one parameter subgroups. Next, the nef cone was already known in the first step of the program outlined in \Cref{Computing_the_ample_cone}. Taking $V$ and $\sigma$ in \Cref{type_one_critical_subset_reasoning} to be $\Pic(X_{\Sigma})_{\mathbf{R}}$ and $\Nef(X_{\Sigma})$ respectively, we apply the .dual() method in SageMath to return the dual of $\Nef(X_{\Sigma})$. The .contains() method can then be used to determine if $\pm \nu_{\rho}$ is in $\Nef(X_{\Sigma})^{\vee}$, finishing step (3).\\
\indent
For the fan $\Sigma_{\text{ex}}$, consider the pair of subsets $([],[0])$. It can be checked that $W_{0}$ is a codimension one subspace of $\pmb{\Gamma}(G)_{\mathbf{R}}$. The condition that $\chi^{\ast}_{D}=\Proj_{W_{0}}\chi^{\ast}_{D}$ amounts to setting the first component of $\chi^{\ast}_{D}$ to $0$. Referring back to list (\ref{P1P1_reduced_potential_worst_one_ps}), this amounts to setting $-\frac{1}{2}a_{0}=0$, which is impossible in the ample cone. In fact, 
there are no type one walls for the fan $\Sigma_{\text{ex}}$.
\subsubsection{Type two walls}
Recall a type two wall with respect to $Z_{1},Z_{2}\in\mathcal{L}$ is of the form 
\begin{equation*}\{D\in\Amp(X_{\Sigma})_{\mathbf{R}}\mid ||\Proj_{W_{Z_{1}}}\chi^{\ast}_{D}||=||\Proj_{W_{Z_{2}}}\chi^{\ast}_{D}||\}\end{equation*} where $W_{Z_{1}}\not\subset W_{Z_{2}}$ and $W_{Z_{2}}\not\subset W_{Z_{1}}$. To detect the containment, we would use the .is\_subspace() method. The equations of type two walls are immediate as we already recorded the norm of each potential real one parameter subgroups. The non-trivial task is to determine if the equation has solutions inside the ample cone. Due to computational difficulties, we currently do not test if a quadratic homogeneous polynomial has a solution in the ample cone or not. Instead, we list equations $||\Proj_{W_{Z_{1}}}\chi^{\ast}_{D}||=||\Proj_{W_{Z_{2}}}\chi^{\ast}_{D}||$ for any pair $W_{Z_{1}},W_{Z_{2}}$ without containment. Doing this does not affect visualization of the decomposition of ample cone. SageMath can plot graphs in a certain region so anything outside the ample cone will not be plotted.

\subsubsection{Data structure of walls}
We store walls as a list where each entry is a list
$$[f,l].$$ We explain what $f$ and $l$ are for each type of wall.

For type one walls, the $l$ in the entry $[f,l]$ is a list of pairs $(Z,\{\rho\})$ defining a type one wall with respect to $Z$ and $\{\rho\}$ with equation $f=0$. For type two walls, the $l$ in the entry $[f,l]$ is a list of pairs $(l_{1},l_{2})$ of lists of sets in $\mathcal{L}$ and $f$ is the homogeneous polynomial such that $f=0$ corresponds to the type two walls with respect to $Z_{1}$ and $Z_{2}$ for all $Z_{1}\in l_{1}$ and for all $Z_{2}\in l_{2}$.

Take $\Sigma_{ex}$ for example, we have already seen that there is no type one wall. As for type two walls, note that the only pair of subspaces without containment is $W_{0}$ and $W_{2}$. The output is  \begin{equation*}\begin{split}&-1/2*a0^2 - a0*a1 - 1/2*a1^2 + 1/2*a2^2 + a2*a3 + 1/2*a3^2,\\
 & ([[0], [1], [0, 1]], [[2], [3], [2, 3]]).\end{split}\end{equation*} Setting $a_{1}=a_{3}=0$, we specialize the result to $\Pic(X_{\Sigma_{\text{ex}}})_{\mathbf{R}}$. The equation becomes $-1/2*a0^2+1/2*a2^2=0$. Since $a_{0},a_{2}>0$ in the ample cone, it is the line $-a_{0}+a_{2}=0$. Moreover, the pair of lists in the end tells us that the equation $-a_{0}+a_{2}=0$ corresponds to the condition $$||\Proj_{W_{Z_{1}}}\chi^{\ast}_{D}||=||\Proj_{W_{Z_{2}}}\chi^{\ast}_{D}||$$ for any $Z_{1}\in[[0], [1], [0, 1]]$ and any $Z_{2}\in[[2], [3], [2, 3]]$ in the ample cone. We will see that this wall swaps the ordering of a pair of strata in \Cref{Compute_stratification_with_respect_to_ample_divisors}. 

\subsection{Computing stratifications with respect to ample divisors}\label{Compute_stratification_with_respect_to_ample_divisors}
Given an ample divisor $D$ on $X_{\Sigma}$, the program 
\begin{enumerate}
    \item computes the stratification of $Z(\Sigma)$ induced by $\chi_{D}$, and 
    \item plots the Hasse diagram of the stratification as a poset.
\end{enumerate}
Moreover, given two ample divisors $D$ and $D^{\prime}$, the program determines if $\chi_{D}$ and $\chi_{D^{\prime}}$ induce equivalent  stratifications (\Cref{strat_def}). 
We explain briefly how each of the three functions works.
\subsubsection{Stratification induced by an ample divisor}
Given $D\in\Amp(X_{\Sigma})$, computing the stratification induced by $\chi_{D}$ boils down to two steps:
\begin{enumerate}
    \item Compute the real one parameter subgroup $\lambda^{D}_{S}$ that is $\chi_{D}$-adapted to $L(S)$ for each $S\in\mathcal{L}$. 
    \item Group $\{L(S)\}_{S\in\mathcal{L}}$ by $\{\lambda^{D}_{S}\}_{S\in\mathcal{L}}$.
\end{enumerate} 

As was discussed in \Cref{Extensions_to_the_ample_cone}, step (1) is equivalent to finding the longest vector in $$\Lambda^{D}_{S}=\{-\Proj_{W_{Z}}\chi^{\ast}_{D}\bigm|-\Proj_{W_{Z}}\chi^{\ast}_{D}\in\sigma_{S},Z\subset S\}.$$ This is done by a combination of basic set operations and sorting on the list of potential real one parameter subgroups obtained earlier. Step (2) is also just a sequence of set operations.

\subsubsection{Stratification as a poset}\label{Stratification as a poset}
We store the stratification as a poset in SageMath where each node is a tuple of sets in $\mathcal{L}$, corresponding to a stratum. For example, if a node is the tuple $(S_{1},\ldots,S_{k})$ of subsets in $\mathcal{L}$, then it corresponds to the stratum $\cup_{i=1}^{k}L(S_{i}).$ The strict partial order is defined by the norms of strata's indexing real one parameter subgroups. 

We also supply the feature to visualize posets. This is simply done by applying the .plot() method for posets in SageMath. We demonstrate the visualization for the fan $\Sigma_{ex}$.

The stratification induced by the ample divisors $(a_{0},a_{2})=(2,1)$, $(1,1)$, $(1,2)$ respectively looks like
$$\includegraphics[height = 55mm, width = 20mm]{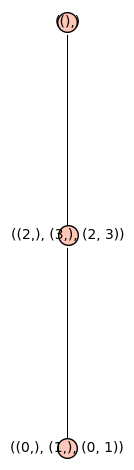}\includegraphics[height = 55mm, width = 40mm]{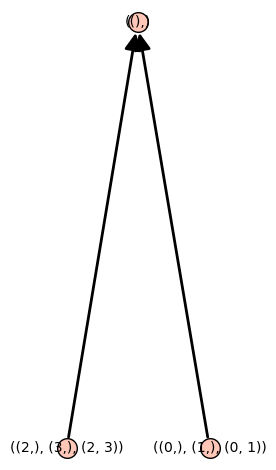}\includegraphics[height = 55mm, width = 20mm]{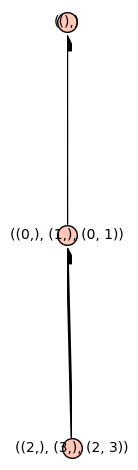}.$$ 

Each node in the diagram is a tuple. The empty tuple $((),)$ corresponds to the stratum $L(\emptyset)$, which is the origin. The tuple $((0,), (1,), (0, 1))$ corresponds to the stratum $L(\{0\})\cup L(\{1\})\cup L(\{0,1\})$ and similarly for others. 

For some reason sageMath is not very consistent in drawing arrows or just an edge. In any case, higher order strata are placed higher in the diagram. In \Cref{Enumerating_walls}, we calculated that $a_{0}=a_{2}$ defines a type two wall. Note that the divisor $(a_{0},a_{2})=(1,1)$ is on the type two wall and $(2,1)$ $(1,2)$ are on different sides of the wall. According to the outputs, crossing this type two wall swaps the orderings of a pair of strata and on the wall the order between these two strata come together, breaking up the poset into two chains.

\subsubsection{Comparing stratifications induced by two ample divisors}
To determine if the stratifications induced by two ample divisors $D,D^{\prime}$ are equivalent, we construct the poset for each ample divisor given. Then there is the .is\_isomorphic() method in SageMath available to compares if the two posets are isomorphic. By \Cref{strat_def}, $\chi_{D}$ and $\chi_{D^{\prime}}$ induce equivalent stratifications if the posets constructed here for $D$ and $D^{\prime}$ are isomorphic. 
\subsection{Visualizing walls and semi-chambers in the ample cone}\label{Visualizing_the_ample_cone_decomposition}
The program can plot the ample cone and its walls when the dimension of the ample cone is less than three. For two dimensional ample cones, the .implicit\_plot() method can be applied directly to plot walls in the ample cone. For three dimensional ample cones the task boils down to the two steps:
\begin{enumerate}
    \item Produce a two dimensional affine hyperplane $H\subset \Pic(X_{\Sigma})_{\mathbf{R}}$ such that $\text{Cone }(H\cap\Amp(X_{\Sigma})_{\mathbf{R}})=\Amp(X_{\Sigma})_{\mathbf{R}}.$
    \item Use the implcit\_plot() function in SageMath to plot walls in the region $H\cap\Amp(X_{\Sigma})_{\mathbf{R}}.$
\end{enumerate}
The result is a two dimensional slice of the original picture.

For step(1), \Cref{nef_and_ample_cone_thm} tells us that the ample cone is the interior of the nef cone for projective toric varieties. Hence it is sufficient to find an affine hyperplane $H$ that meets all the rays of the nef cone $\Nef(X_{\Sigma})$. This can be achieved with the help of 
\begin{proposition}\label{cross_section_poly_cone}
Let $\sigma$ be a full dimensional, strongly convex polyhedral cone and $\sigma^{\vee}$ be its dual. Then 
\begin{enumerate}
    \item The sum of ray generators of $\sigma$ is in the interior of $\sigma$.
    \item Let $m\in\sigma^{\vee}$. Then $H_{m}\cap\sigma=\{0\}$ if and only if $m$ is in the interior of $\sigma^{\vee}$.
\end{enumerate}
\end{proposition}

We now explain how \Cref{cross_section_poly_cone} can be applied. We already had the nef cone $\Nef(X_{\Sigma})$ in earlier stages of the program described in \Cref{Computing_the_ample_cone}. Since SageMath can compute the dual and the list of ray generators of a polyhedral cone, we have the list of ray generators for the dual cone $\Nef(X_{\Sigma})^{\vee}$. Let $f$ be the sum of ray generators of $\Nef(X_{\Sigma})^{\vee}$.

Since the nef cone $\Nef(X_{\Sigma})$ is full dimensional and strongly convex in $\Pic(X_{\Sigma})_{\mathbf{R}}$ by \Cref{nef_and_ample_cone_thm}, so is $\Nef(X_{\Sigma})^{\vee}.$  We see from \Cref{cross_section_poly_cone} that $f$ is in the interior of $\Nef(X_{\Sigma})^{\vee}$ and $H_{f}\cap\Nef(X_{\Sigma})=\{0\}.$ In particular, $f$ does not vanish on the ray generators of $\Nef(X_{\Sigma})$. 

Let $u_{\rho}$ be any ray generator of $\Nef(X_{\Sigma})$. If $u_{\rho^{\prime}}$ is any other ray generator, we then have $$f(\frac{f(u_{\rho})}{f(u_{\rho^{\prime}})}u_{\rho^{\prime}})-f(u_{\rho})=0.$$ This means the affine plane $H\subset\Pic(X_{\Sigma})_{\mathbf{R}}$ defined by the vanishing locus of the function $$v\mapsto f(v)-f(u_{\rho})$$ contains all rays of $\Nef(X_{\Sigma})$, finishing step (1).

For step (2), the linear equation for $H$ allows us to solve one of the variable in terms of the other two. Hence we can rewrite the equations of walls and the inequalities that define the ample cone in two variables. With these we then use the .implicit\_plot() function to plot walls within the region $H\cap\Amp(X_{\Sigma})_{\mathbf{R}}$, finishing step (2).

The result is a two dimensional slice of the ample cone and its decomposition by walls and semi-chambers. Interested readers may look at \Cref{blow_up_2_pts} for a two dimensional slice of the wall and semi-chamber decomposition of a three dimensional ample cone.
\section{Examples and counter examples}\label{Examples_and_counter_examples}
The highlights of this section are examples where 
\begin{enumerate}
\item there are walls that are both type one and type two (\Cref{A wall can be of type one and type two}),
\item two semi-chambers can be contained in a single SIT-equivalence class (\Cref{Wall and semi-chamber decomposition is finer than SIT-equivalence classes}),
\item for each primitive collection $C$, the one parameter subgroup $\lambda_{C}$ in
  \Cref{closure_strata_primitive_collection} does not relate to the primitive relation we cooked up in \Cref{non_simplicial_primitive_relation} (\Cref{Counter example to primitive relations}), and 
\item a semi-chamber does not have to be convex (\Cref{blow_up_2_pts}).
\end{enumerate}

For brevity, in the discussion that follows, whenever $S\in\mathcal{L}$ is a concrete set like $\{1,3\}$, instead of writing $L(S)$ as $L(\{1,3\})$, we will simply write $L_{13}$ without commas separating the elements in $S$. Likewise for $\sigma_{S}$ and $W_{S}$ $\lambda^{D}_{S}$ and $\Lambda^{D}_{S}$, we will not include commas. However, we will keep the notation $L(S)$ when we talk about an abstract $S$.
\subsection{Blow-up of the Hirzebruch surface at a point}\label{Blow-up of the Hirzebruch surface at a point}
This is a rich example where phenomena (1),(2),(3) listed in the beginning of \Cref{Examples_and_counter_examples} can be found. 
\subsubsection{The set up}
Let $\Sigma\subset\mathbf{R}^{2}$ be the smooth complete fan whose rays are given by \begin{equation*}
\begin{split}
\rho_{0}=&\text{Cone }(e_{1})\\
\rho_{1}=&\text{Cone }(e_{2})\\
\rho_{2}=&\text{Cone }(-e_{2})\\
\rho_{3}=&\text{Cone }(-e_{1}+e_{2})\\
\rho_{4}=&\text{Cone }(-e_{1}+2e_{2}).
\end{split}
\end{equation*}
The primitive collection consists of
\begin{enumerate}
\item $\{\rho_{0},\rho_{3}\}$
\item $\{\rho_{0},\rho_{4}\}$
\item $\{\rho_{1},\rho_{2}\}$
\item $\{\rho_{1},\rho_{3}\}$
\item $\{\rho_{2},\rho_{4}\}$.
\end{enumerate}
Moreover, $$\begin{array}{ll}\mathcal{L}=&[[],
 [1],
 [2],
 [1, 2],
 [4],
 [1, 4],
 [2, 4],
 [1, 2, 4],
 [3],
 [1, 3],
 [2, 3],
 [1, 2, 3],
 [0],
 [0, 3],\\
 &
 [0, 4],
 [3, 4],
 [0, 3, 4],
 [0, 2],
 [0, 2, 4],
 [0, 1],
 [0, 1, 3]]
\end{array}$$
By \Cref{CDiv_basis} applied to the maximal cone $\text{Cone}(\rho_{0},\rho_{1})$, $\Pic(X_{\Sigma})$ is free with basis
$D_{2},D_{3},D_{4}$. The computer program computes that a divisor
$\sum_{i=2}^{4}a_{i}D_{i}$ is ample if and only if 
$$a_{2}+a_{4}>a_{3}>a_{4}>0.$$
\subsubsection{The potential real one parameter subgroups}\label{The potential one parameter subgroups}
For a point $D=\sum_{i=2}^{4}a_{i}D_{i}\in\Amp(X_{\Sigma})_{\mathbf{R}}$, the
potential real one parameter subgroups in $\pmb{\Gamma}(G)_{\mathbf{R}}\subset\mathbf{R}^{\Sigma(1)}$ are \begin{enumerate}
\item
  $-\chi_{D}^{\ast}=\begin{bmatrix}\begin{array}{r}
-1/4*a2 - 1/3*a3 - 1/12*a4\\
-1/4*a2 + 1/4*a4\\
-3/4*a2 - 1/4*a4\\
-2/3*a3 + 1/3*a4\\
-1/4*a2 + 1/3*a3 - 5/12*a4\end{array}\end{bmatrix},$
\item
  $-\Proj_{W_{0}}\chi_{D}^{\ast}=\begin{bmatrix}
\begin{array}{r}
0\\
-2/5*a2 - 1/5*a3 + 1/5*a4\\
-3/5*a2 + 1/5*a3 - 1/5*a4\\
1/5*a2 - 2/5*a3 + 2/5*a4\\
-1/5*a2 + 2/5*a3 - 2/5*a4\end{array}\end{bmatrix},$
\item
  $-\Proj_{W_{1}}\chi_{D}^{\ast}=\begin{bmatrix}\begin{array}{r}
-1/3*a2 - 1/3*a3\\
0\\
-2/3*a2 - 1/3*a4\\
-2/3*a3 + 1/3*a4\\
-1/3*a2 + 1/3*a3 - 1/3*a4\end{array}\end{bmatrix},$
\item
  $-\Proj_{W_{2}}\chi_{D}^{\ast}=\begin{bmatrix}\begin{array}{r}
-1/3*a3\\
1/3*a4\\
0\\
-2/3*a3 + 1/3*a4\\
1/3*a3 - 1/3*a4\end{array}\end{bmatrix},$
\item
  $-\Proj_{W_{3}}\chi_{D}^{\ast}=\begin{bmatrix}\begin{array}{r}
-1/4*a2 - 1/4*a4\\
-1/4*a2 + 1/4*a4\\
-3/4*a2 - 1/4*a4\\
0\\
-1/4*a2 - 1/4*a4\end{array}\end{bmatrix},$
\item $-\Proj_{W_{4}}\chi_{D}^{\ast}=\begin{bmatrix}\begin{array}{r}
-1/5*a2 - 2/5*a3\\
-2/5*a2 + 1/5*a3\\
-3/5*a2 - 1/5*a3\\
-1/5*a2 - 2/5*a3\\
0\end{array}\end{bmatrix},$
\item $-\Proj_{W_{14}}\chi_{D}^{\ast}=\begin{bmatrix}\begin{array}{r}
-1/3*a2 - 1/3*a3\\
0\\
-1/3*a2 - 1/3*a3\\
-1/3*a2 - 1/3*a3\\
0
\end{array}
\end{bmatrix}$,
\item $-\Proj_{W_{01}}\chi^{\ast}_{D}=\begin{bmatrix}\begin{array}{r}
0\\
0\\
-1/3*a2 + 1/3*a3 - 1/3*a4\\
1/3*a2 - 1/3*a3 + 1/3*a4\\
-1/3*a2 + 1/3*a3 - 1/3*a4
\end{array}
\end{bmatrix}$,
\item $-\Proj_{W_{12}}\chi^{\ast}_{D}=\begin{bmatrix}\begin{array}{r}
-1/3*a3 + 1/6*a4\\
0\\
0\\
-2/3*a3 + 1/3*a4\\
1/3*a3 - 1/6*a4
\end{array}
\end{bmatrix}$,
\item $-\Proj_{W_{13}}\chi^{\ast}_{D}=\begin{bmatrix}\begin{array}{r}
-1/3*a2 - 1/6*a4\\
0\\
-2/3*a2 - 1/3*a4\\
0\\
-1/3*a2 - 1/6*a4
\end{array}
\end{bmatrix},$
\item $-\Proj_{W_{23}}\chi^{\ast}_{D}=\begin{bmatrix}\begin{array}{r}
-1/6*a4\\
1/3*a\\
0\\
0\\
-1/6*a4
\end{array}
\end{bmatrix}$,
\item
  $-\Proj_{W_{03}}\chi^{\ast}_{D}=-\Proj_{W_{04}}\chi^{\ast}_{D}=-\Proj_{W_{34}}=-\Proj_{W_{034}}\chi^{\ast}_{D}=\begin{bmatrix}\begin{array}{r}
  0\\
  -1/2*a2\\
  -1/2*a2\\
  0\\
  0
  \end{array}\end{bmatrix}$,
\item $-\Proj_{W_{24}}\chi^{\ast}_{D}=\begin{bmatrix}\begin{array}{r}
-1/3*a3\\
 1/3*a3\\
 0\\
-1/3*a3\\
  0
\end{array}\end{bmatrix}$,
\item $-\Proj_{W_{02}}\chi^{\ast}_{D}=\begin{bmatrix}\begin{array}{r}
0\\
-1/3*a3 + 1/3*a4\\
0\\
-1/3*a3 + 1/3*a4\\
 1/3*a3 - 1/3*a4
\end{array}\end{bmatrix}$.
\end{enumerate}
Although there are 21 subsets in $\mathcal{L}$, we listed the seventeen $-\Proj_{W_{Z}}\chi^{\ast}_{D}$ where $W_{Z}\neq 0$.
\subsubsection{Stratification induced by ample divisors}
For each $S\in\mathcal{L}$, we may describe the vectors $\lambda^{D}_{S}$ that is $\chi_{D}$-adapted to $L(S)$ as a
piecewise function of $D\in\Amp(X_{\Sigma})_{\mathbf{R}}$ like we did in \Cref{Variation of stratification-an_elementary_example}. To get the stratification induced by $\chi_{D}$ one simply groups $\{L(S)\}_{S\in\mathcal{L}}$ by $\{\lambda^{D}_{S}\}_{S\in\mathcal{L}}$. 

Take $S=[0]\in\mathcal{L}$ for example. By the list of potential real one parameter subgroups provided in \Cref{The potential one parameter subgroups}, for $-\chi_{D}^{\ast}$ to be inside $\sigma_{0}$, we must have
$-1/4*a2 - 1/3*a3 - 1/12*a4\geq 0$ . This is clearly impossible inside the
ample cone. We conclude that $\lambda^{D}_{0}=-\Proj_{W_{0}}\chi^{\ast}_{D}$. Let us look at one more
example where $S=[1,4]$. 

It can be checked that $-\chi^{\ast}_{D}\notin\sigma_{14}$ and
$-\Proj_{W_{1}}\chi^{\ast}_{D}\notin\sigma_{14}$. However, 
$-\Proj_{W_{4}}\chi^{\ast}_{D}\in\sigma_{14}$ if $−2/5*a2+1/5*a3\geq
0$, which is is possible in the ample cone. We then conclude that $$\lambda^{D}_{14}=\begin{cases}-\Proj_{W_{4}}\chi^{\ast}_{D} \text{ if } −2/5*a2+1/5*a3\geq
0, \text{ or}\\ -\Proj_{W_{14}}\chi^{\ast}_{D} \text{ otherwise}.\end{cases}$$ Below is a
complete list of piecewise descriptions of the real one parameter subgroup $\lambda^{D}_{S}$ that is $\chi_{D}$-adapted to $L(S)$ for all $S\in\mathcal{L}$ and for all $D\in\Amp(X_{\Sigma})_{\mathbf{R}}$. 

\begin{enumerate}
\item $S=\emptyset$,\\
  $\lambda^{D}_{\emptyset}=-\chi_{D}^{\ast}$.
\item $S=[0],$\\
 $\lambda^{D}_{0}=-\Proj_{W_{0}}\chi^{\ast}_{D}$.
\item $S=[1]$,\\
$\lambda^{D}_{1}=\begin{cases}-\chi^{\ast}_{D}\text{ if }-1/4*a2 +
1/4*a4\geq 0,\text{ or}\\ -\Proj_{W_{1}}\chi^{\ast}_{D}\text{ otherwise}.\end{cases}$
\item $S=[1,2,4], [2,4],[1,2]$, or $[2],$ \\
$\lambda^{D}_{S}=-\Proj_{W_{2}}\chi^{\ast}_{D}$
\item $S=[3]$,\\
 $\lambda^{D}_{3}=-\Proj_{W_{3}}\chi^{\ast}_{D}$.
\item $S=[4]$,\\
$\lambda^{D}_{4}=\begin{cases}-\chi^{\ast}_{D}\text{ if }-1/4*a2 +
1/3*a3 - 5/12*a4\geq 0,\text{ or }\\-\Proj_{W_{4}}\chi^{\ast}_{D}\text{ otherwise.}\end{cases}$
\item $S=[0,3,4]$ or $[3,4]$,\\
$\lambda^{D}_{S}=-\Proj_{W_{34}}\chi^{\ast}_{D}$.
\item $S=[1,2,3]$ or $[2,3]$,\\
$\lambda^{D}_{S}=-\Proj_{W_{23}}\chi^{\ast}_{D}$. 
\item $S=[0,2,4]$ or $[0,2]$,\\
$\lambda^{D}_{S}=-\Proj_{W_{02}}\chi^{\ast}_{D}.$
\item $S=[0,1,3]$,\\
$\lambda^{D}_{013}=-\Proj_{W_{01}}\chi^{\ast}_{D}.$ 
\item $S=[1,4]$,\\
$\lambda^{D}_{14}=\begin{cases}-\Proj_{W_{4}}\chi^{\ast}_{D}\text{ if }-2/5*a2 +
1/5*a3\geq 0,\text{ or}\\
-\Proj_{_{14}}\chi^{\ast}_{D}\text{ otherwise.}\end{cases}$ 
\item $S=[0,3]$,\\
$\lambda^{D}_{03}=\begin{cases}-\Proj_{W_{0}}\chi^{\ast}_{D}\text{ if }1/5*a2 - 2/5*a3 + 2/5*a4\geq 0,\text{ or}\\ 
-\Proj_{W_{03}}\text{ otherwise.}\end{cases}$ 
\item $S=[0,4]$,\\ 
$\lambda^{D}_{04}=\begin{cases}-\Proj_{W_{0}}\chi^{\ast}_{D}\text{ if }-1/5*a2 + 2/5*a3 - 2/5*a4\geq
0,\text{ or}\\-\Proj_{W_{04}}\chi^{\ast}_{D}\text{ otherwise.}\end{cases}$ 
\item $S=[1,3]$,\\
$\lambda^{D}_{13}=\begin{cases}-\Proj_{W_{3}}\chi^{\ast}_{D}\text{
if }-1/4*a2 + 1/4*a4\geq 0,\text{ or}\\ 
-\Proj_{W_{13}}\chi^{\ast}_{D}\text{ otherwise.} \end{cases}$
\item $S=[0,1]$, \\
$\lambda^{D}_{01}=-\Proj_{W_{01}}\chi^{\ast}_{D}$. 
\end{enumerate}

\subsubsection{The list of type one Walls}
The program returns the following list of type one walls. We refer the reader back to \Cref{Enumerating_walls} for relevant data structures.
\begin{equation}\label{type_one_wall_list_blow-up_Hirz}
\begin{array}{l}
2/5*a2 - 1/5*a3, [\big([4], [1]\big)]\\
-1/5*a2 + 2/5*a3 - 2/5*a4,
  [\big([0], [3]\big), \big([0], [4]\big)]\\
1/4*a2 - 1/3*a3 + 5/12*a4, [\big([], [4]\big)]\\
1/4*a2 - 1/4*a4, [\big([3], [1]\big), \big([], [1]\big)].
\end{array}
\end{equation}
We do not include type two walls here as it is rather long and complicated. Instead, we only cite type two walls that are used as we go along.
\subsubsection{A wall can be of type one and type two}\label{A wall can be of type one and type two}
We examine the type
one wall $$W_{0,03}:=\{D\in\Amp(X_{\Sigma})_{\mathbf{R}}\mid\Proj_{W_{0}}\chi^{\ast}_{D}=\Proj_{W_{03}}\chi^{\ast}_{D}\}.$$
According to the type one wall list (\ref{type_one_wall_list_blow-up_Hirz}), this corresponds to the hyperplane $$-1/5*a2 + 2/5*a3 - 2/5*a4=0.$$ Let 
$$W^{+}_{0,03}=\{(a2,a3,a4)\in\Amp(X_{\Sigma})_{\mathbf{R}}\mid-1/5*a2 + 2/5*a3 - 2/5*a4>0 \},$$
$$W_{0,03}^{-}=\{(a2,a3,a4)\in\Amp(X_{\Sigma})_{\mathbf{R}}\mid-1/5*a2 + 2/5*a3 - 2/5*a4<0\}.$$
With the piecewise information provided in \Cref{The potential one parameter subgroups}, we deduced the following:
\begin{enumerate}
\item When $D\in W^{-}_{0,03}$, $L_{03}$ is in the stratum indexed by $-\Proj_{W_{0}}\chi^{\ast}_{D}$, and $L_{04}\cup L_{034}\cup L_{34}$ is in the different stratum indexed $-\Proj_{W_{03}}\chi^{\ast}_{D}$.
\item When $D\in W_{0,03}$, $\Proj_{W_{0}}\chi^{\ast}_{D}=\Proj_{W_{03}}\chi^{\ast}_{D}$. Hence $L_{03}\cup L_{04}\cup L_{034}\cup L_{34}$ all come together in the same stratum.
\item When $D\in W^{+}_{0,03}$, $L_{04}$ is in the stratum indexed $-\Proj_{W_{0}}\chi^{\ast}_{D}$, separated from $L_{03}\cup L_{034}\cup L_{34}$, which is in the stratum indexed by $-\Proj_{W_{03}}\chi^{\ast}_{D}$\end{enumerate} For concreteness we consider the three ample divisors 
\begin{enumerate}
    \item $D^{-}=45D_{2}+65D_{3}+55D_{4}\in W^{-}_{0,03}$\label{some_div},
    \item $D^{0}=450D_{2}+775D_{3}+550D_{4}\in W_{0,03}$,
    \item $D^{+}=45D_{2}+80D_{3}+55D_{4}\in W^{+}_{0,03}$.
\end{enumerate}
Here are the stratifications induced by $D^{-},D^{0}$ and $D^{+}$ respectively.
$$\includegraphics[width = 4.1cm, height = 8cm]{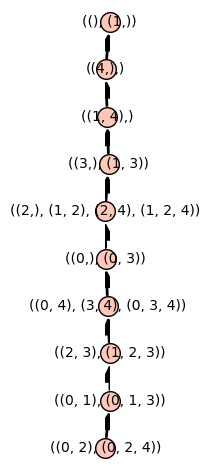}\includegraphics[width = 4.1cm, height = 8cm]{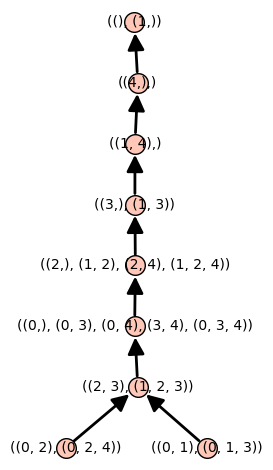}\includegraphics[width = 4.1cm, height = 8cm]{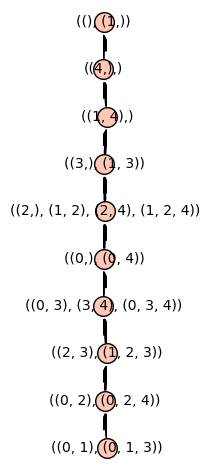}.$$ 

We refer the reader back to \Cref{Stratification as a poset} for how to understand the graphic representation of a stratification. Notice how $L_{03}$ and $L_{04}$ move in the diagrams. To understand why the chain breaks up into two at $D^{0}\in W_{0,03}$ , we investigate the type two wall
$$Z_{01,02}:=\{D\in\Amp(X_{\Sigma})_{\mathbf{R}}\mid ||\Proj_{W_{01}}\chi^{\ast}_{D}||=||\Proj_{W_{02}}\chi^{\ast}_{D}||\}.$$ Using the computer program, we see that $Z_{01,02}$ corresponds to the condition $$-1/3*a2^2 + 2/3*a2*a3 - 2/3*a2*a4=0.$$ Since $a_{2}>0$ in the ample cone, this condition is equivalent to 
$$-a_{2}+2a_{3}-a_{4}=0,$$ which is exactly the defining equation for $W_{0,03}$. Notice how $L_{01},L_{02}$ move in the diagrams. Crossing this wall results in two types of variations. This is an example where a wall can be of both type.

\subsubsection{Two semi-chambers can be contained in an SIT-equivalence class}\label{Wall and semi-chamber decomposition is finer than SIT-equivalence classes}
In this example we are also able to see that the stratifications induced by ample divisors in two different semi-chambers can be the same. 
To establish such a pair of semi-chambers, it is sufficient to identify a wall $H$ together with two ample divisors $D,D^{\prime}$ such that the following conditions hold:
\begin{enumerate}
    \item Both $D$ and $D^{\prime}$ are not on any walls. Namely, they are in some semi-chambers. \label{3conditions}
    \item $D$ and $D^{\prime}$ are on different sides of $H.$ Namely, if $F=0$ is the defining  equation for $H$, we require that $F(D)\cdot F(D^{\prime})<0$.
    \item $D$ and $D^{\prime}$ induce equivalent stratifications. 
\end{enumerate}
Indeed, by the definition of a semi-chamber (\Cref{semi-chamber_def}), any pair $(D,D^{\prime})$ of ample divisors that satisfy the first two  conditions are in different semi-chambers. 

We provide two walls here where the above three conditions are met. The first one is given by the type two wall $$Z_{3,24}=\{D\in\Amp(X_{\Sigma})_{\mathbf{R}}\mid||\Proj_{W_{3}}\chi^{\ast}_{D}||=||\Proj_{W_{24}}\chi^{\ast}_{D}||\}.$$ The equation for  $Z_{3,24}$ is
$$3/4*a2^2 - 1/3*a3^2 + 1/2*a2*a4 + 1/4*a4^2=0.$$

We chose this wall because of the two reasons. First, by the piecewise information provided earlier, there are no subsets $S\in\mathcal{L}$ with $\lambda^{D}_{S}=-\Proj_{W_{24}}\chi^{\ast}_{D}$ throughout the ample cone. Second, it is checked that there are no other walls whose equation is the same as the equation for $Z_{3,24}$. We therefore expect $Z_{3,24}$ never swaps the ordering of any pair of strata. Indeed, it can be checked by the computer that the pair of ample divisors $D=30D_{2}+92D_{3}+70D_{4}$ and $D^{\prime}=30D_{2}+99D_{3}+70D_{4}$ and the wall $Z_{3,24}$  satisfy the three conditions required in the beginning.

Here is a visualization of $Z_{3,24}$, drawn as a blue curve in a 2D slice of the ample cone.
$$\includegraphics[height = 50mm]{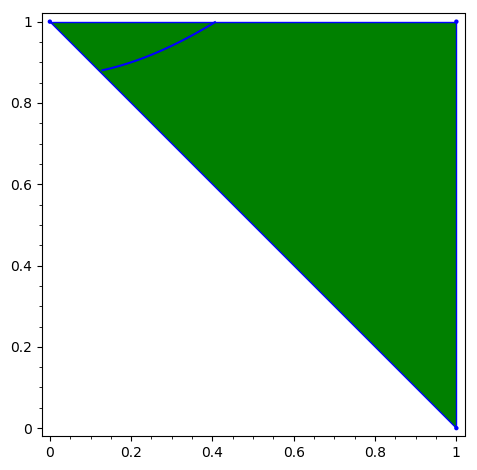}.$$

The slice is the intersection of the ample cone with the hyperplane $a2+a4=1$. With $a4$ replaced by $1-a2$ the green region on the $a2,a3$-plane  defined by $1>a3>1-a2$ is the 2D slice of the ample cone. Also note that $\dim W_{3}=2$, $\dim W_{24}=1$ and $W_{3}\oplus W_{24}=\pmb{\Gamma}(G)_{\mathbf{R}}$. By \Cref{equi-distant_simplified}, $Z_{3,24}$ looks like a quadric cone of the form $x_{1}^{2}=y_{1}^{2}+y_{2}^{2}$, giving the slice of $Z_{3,24}$ the curviness we see in the figure.

The second wall is more interesting. It has two connected components where producing a pair of divisors that meet the three conditions is only possible around one component. The wall is given by 
$$Z_{1,4}:=\{D\in\Amp(X_{\Sigma})_{\mathbf{R}}\mid||\Proj_{W_{1}}\chi^{\ast}_{D}||=||\Proj_{W_{4}}\chi^{\ast}_{D}||\}.$$ This corresponds to the hypersurface $$-1/15*a2^2 + 2/5*a2*a3 - 4/15*a3^2 - 2/3*a2*a4 + 2/3*a3*a4 - 1/3*a4^2 = 0.$$ Note that since $\dim W_{1}=\dim W_{4}=2$ and $W_{1}+W_{2}=\pmb{\Gamma}(G)_{\mathbf{R}}$, \Cref{equi-distant_simplified} implies that $Z_{1,4}$ is a union of two hyperplanes. The line where these two hyperplanes meet is contained in the supporting hyperplane of the nef cone. The wall splits up into two components in the ample cone (See figure (\ref{parts_type_II_wall_redundant})).

To fully describe wall crossing behavior for $Z_{1,4}$, we will need other two type one walls:
\begin{enumerate}
    \item $-a2+a4=0$ that corresponds to the condition $\chi_{D}^{\ast}=\Proj_{W_{1}}\chi^{\ast}_{D}$, and 
    \item $-3*a2+4*a3-5*a4=0$ that corresponds to the condition $\chi^{\ast}_{D}=\Proj_{W_{4}}\chi^{\ast}_{D}$.
\end{enumerate}
The visualization of these three walls in the ample cone is given by \begin{equation}\label{parts_type_II_wall_redundant}
\includegraphics[height = 50mm]{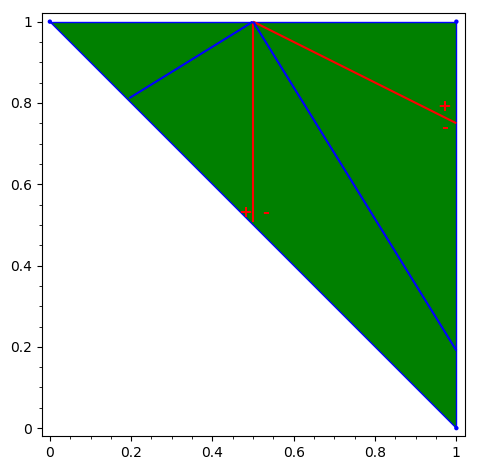}\end{equation}

The union of blue lines is $Z_{1,4}$. The vertical red line is defined by $1-2*a2=0$, which corresponds to the type one wall $-a2+a4=0$. The other red line is defined by $2*a2+4*a3-5=0$, which corresponds to the type one wall $-3*a2+4*a3-5*a4=0$. The blue lines and red lines all meet at the point $(a2,a3)=(1/2,1)$, which is on the boundary of the ample cone. The signs are inserted into regions to indicate the signs the defining polynomials of the type one walls assume there. 

Let us call the region bounded by the positive side of the vertical red line region one. The region bounded by the negative sides of the two red lines will be called region two. We will prove the 
\begin{claim}\label{claim}
It is only possible to produce a pair of ample divisors around the component of $Z_{1,4}$ bounded in region one to meet the three conditions mentioned in the beginning of \Cref{Wall and semi-chamber decomposition is finer than SIT-equivalence classes}.\end{claim}
Let $D$ and $D^{\prime}$ be two ample divisors on different sides of $Z_{1,4}$. We have to establish the impossibility for $D,D^{\prime}$ to satisfy the three conditions in the following two cases:
\begin{enumerate}
    \item Both $D$ and $D^{\prime}$ are in region two.
    \item $D$ is in region one and $D^{\prime}$ is in region two.
\end{enumerate}
For reader's convenience, we clip the piecewise information specifically for $S=[1],[4]$:
\begin{itemize}
    \item $S=[1]$,\\
$\lambda^{D}_{1}=\begin{cases}-\chi^{\ast}_{D}\text{ if }-1/4*a2 +
1/4*a4\geq 0,\text{ or}\\
-\Proj_{W_{1}}\chi^{\ast}_{D}\text{ otherwise.}\end{cases}$
\item $S=[4]$,\\
$\lambda^{D}_{S}=\begin{cases}-\chi^{\ast}_{D}\text{ if }-1/4*a2 +
1/3*a3 - 5/12*a4\geq 0,\text{ or}\\
-\Proj_{W_{4}}\chi^{\ast}_{D}\text{ otherwise.}\end{cases}$ 
\end{itemize}

Hence in region two, $L_{1}$ is contained in the strata indexed by $-\Proj_{W_{1}}\chi^{\ast}_{D}$ while $L_{4}$ is contained in the stratum indexed by $-\Proj_{W_{4}}\chi^{\ast}_{D}$. Moreover, it can be checked from the list of type one walls (\ref{type_one_wall_list_blow-up_Hirz}) that $\Proj_{W_{1}}\chi^{\ast}_{D}\neq\Proj_{W_{4}}\chi^{\ast}_{D}$ for any $D\in\Amp(X_{\Sigma})_{\mathbf{R}}$. Therefore, the following two facts are immediate:
\begin{itemize}
    \item $L_{1}$ and $L_{4}$ are in different strata around the component of $Z_{1,4}$ in region two, and 
    \item crossing $Z_{1,4}$ in region two swaps the ordering between the stratum that contains $L_{1}$ and the stratum that contains $L_{4}$.
\end{itemize} 
Therefore, if $D,D^{\prime}$ are divisors on different sides of $Z_{1,4}$ in region two, the stratifications induced by $D$ and $D^{\prime}$ are  never equivalent. The impossibility for case (1) is established.

For case (2), note that the origin $L_{\emptyset}$ is always in the stratum indexed by $-\chi^{\ast}_{D}$. Hence on the negative side of the vertical red line, $L_{\emptyset}$ and $L_{1}$ are in different stratum while on the positive side, they come together in a stratum. Therefore $D$ and $D^{\prime}$ in case (2) never induce equivalent stratifications. \Cref{claim} is proved.

We now produce a pair of ample divisors $D,D^{\prime}$ in region one that satisfy the three conditions mentioned in the beginning of \Cref{Wall and semi-chamber decomposition is finer than SIT-equivalence classes}. 
We take the two ample divisors $D=430D_{2}+960D_{3}+570D_{4}$ and $D^{\prime}=430D_{2}+955D_{3}+570D_{4}$. It can be verified by the computer that 
\begin{itemize}
    \item Both $D$ and $D^{\prime}$ are not on any walls. 
    \item Both $D$ and $D^{\prime}$ are in region one. Namely, they both satisfy $-a2+a4>0$.
    \item $D$ and $D^{\prime}$ are on different sides of $Z_{1,4}$, but 
    \item $D$ and $D^{\prime}$ induces the same stratification.
\end{itemize}  

The reason why it is possible to produce such a pair in region one is that in region one, $L_{4}$ is in the stratum indexed by  $-\Proj_{W_{4}}\chi^{\ast}_{D}$ while $L_{1}$ is in the stratum indexed by $-\chi^{\ast}_{D}$. Therefore, the stratum that contains $L_{1}$ is always higher than the stratum that contains $L_{4}$. Namely, $Z_{1,4}$ no longer swaps the ordering between the stratum that contains $L_{1}$ and the stratum that contains $L_{4}$.

\subsubsection{A counter example about primitive relations}\label{Counter example to primitive relations}
Recall in \Cref{closure_strata_primitive_collection} we proved that for every $\mathbf{R}$-ample divisor $D\in \Amp(X_{\Sigma})_{\mathbf{R}}$ and for every primitive collection $C\subset\Sigma(1)$, there exists a unique vector $\lambda_{C}\in\Lambda^{D}$ such that $$\overline{S^{\chi_{D}}_{\lambda_{C}}}=V(\{x_{\rho}|\rho \in C\}).$$ We will show that $-\lambda_{C}$ need not be a positive multiple of the primitive relation (\Cref{primitive_relation}) for $C$, even when $\Sigma$ is smooth. In fact, we will show by this example that such $\lambda_{C}$ depends on $D$  while the primitive relation of $C$ does not. 

For this, take the primitive collection $C=\{0,3\}$. Since 
$$u_{\rho_{0}}+u_{\rho_{3}}=u_{\rho_{1}},$$ the primitive relation of $C$ is $$(1,-1,0,1,0)\in\pmb{\Gamma}(G)_{\mathbf{R}}.$$ According to the piecewise information, the locally closed subvariety $$L_{2}\cup L_{1 2}\cup L_{2 4}\cup L_{1 2 4}=V(x_{0},x_{3})-V(x_{2})$$ is in the stratum indexed by $-\Proj_{W_{2}}\chi^{\ast}_{D}$. We concluded that $\lambda_{C}=-\Proj_{W_{2}}\chi^{\ast}_{D}$ where we calculated earlier that  $$-\Proj_{W_{2}}\chi_{D}^{\ast}=\begin{bmatrix}\begin{array}{r}
-1/3*a3\\
1/3*a4\\
0\\
-2/3*a3 + 1/3*a4\\
1/3*a3 - 1/3*a4\end{array}\end{bmatrix}.$$ Hence $\lambda_{C}$ depends on $D.$ In fact $-\lambda_{C}$ is never equal to a positive multiple of the primitive relation $(1,-1,0,1,0)$ as $a3>a4$ in the ample cone. 

\subsection{Blow-up of \texorpdfstring{$\mathbf{P}^{3}$}{P3} at two points}\label{blow_up_2_pts}
We end this section with a picture of a 2D slice of the wall and semi-chamber decomposition of the ample cone for the toric variety obtained by blowing up $\mathbf{P}^{3}$ at two torus invariant points.
$$\includegraphics[height = 50mm, width = 50mm]{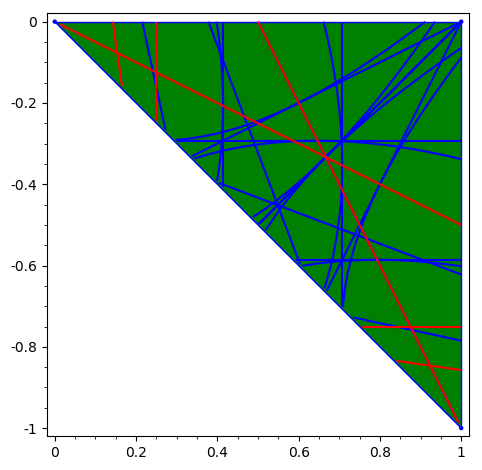}$$
The red lines correspond to type one walls and the blue ones to type two walls. It is also clear from the picture that some semi-chambers are not convex and some have closures that are not polyhedral cones.

\bibliography{mybibitem}
\bibliographystyle{alpha}
\end{document}